\documentclass[11pt,letterpaper]{article}
\usepackage[T1]{fontenc}
\usepackage[margin=1in]{geometry}
\usepackage{amsmath,amsthm,amssymb,mathtools,booktabs,array}
\usepackage{lmodern}
\usepackage{microtype}
\usepackage[authoryear,round]{natbib}
\usepackage{xcolor}
\usepackage{hyperref}
\hypersetup{
  colorlinks=true,
  allcolors=blue,
  pdftitle={Critical fluctuations in the edge--triangle ERGM},
  pdfauthor={Songhao Liu, Qiman Shao, Hao Shi, Jingyu Xu},
  pdfsubject={},
  pdfkeywords={Exponential random graph, Stein method, critical fluctuations, critical window, finite-size correction}
}
\numberwithin{equation}{section}
\theoremstyle{plain}
\newtheorem{theorem}{Theorem}[section]
\newtheorem{proposition}[theorem]{Proposition}
\newtheorem{lemma}[theorem]{Lemma}

\theoremstyle{definition}
\newtheorem{remark}{Remark}[section]
\newcommand{\E}{\mathbb E}
\newcommand{\PP}{\mathbb P}
\newcommand{\R}{\mathbb R}

\newcommand{\var}{\operatorname{Var}}
\newcommand{\cov}{\operatorname{Cov}}
\newcommand{\dd}{\,\mathrm d}
\newcommand{\ee}{\mathrm e}
\newcommand{\calN}{\mathcal N}
\newcommand{\calL}{\mathcal L}
\newcommand{\calF}{\mathcal F}

\newcommand{\norm}[1]{\lVert#1\rVert}
\newcommand{\meanM}{\mathord{\mkern2mu\overline{\mkern-2mu M\mkern-2mu}\mkern2mu}}
\DeclareMathOperator{\inj}{inj}
\DeclareMathOperator{\diag}{diag}
\DeclareMathOperator{\Tr}{Tr}
\allowdisplaybreaks[1]

\title{Critical fluctuations in the edge--triangle exponential random graph models}
\author{Song-Hao Liu$^{1}$ \and Qi-Man Shao$^{2}$ \and Hao Shi$^{3}$ \and Jing-Yu Xu$^{3}$}
\date{}

\begin{document}
\maketitle
\vspace{-1.5em}
\begin{center}
\small
\begin{minipage}{0.98\textwidth}
\centering
$^{1}$School of Mathematical Sciences, Dalian University of Technology,\\
Dalian, Liaoning, China\\[0.4em]
$^{2}$Department of Statistics and Data Science,\\
Shenzhen International Center for Mathematics,\\
Southern University of Science and Technology,\\
Shenzhen, Guangdong, 518055, China\\[0.4em]
$^{3}$Department of Statistics and Data Science,\\
Southern University of Science and Technology,\\
Shenzhen, Guangdong, 518055, China\\[0.7em]
\href{mailto:liusonghao@dlut.edu.cn}{\texttt{liusonghao@dlut.edu.cn}};\quad
\href{mailto:shaoqm@sustech.edu.cn}{\texttt{shaoqm@sustech.edu.cn}}\\
\href{mailto:Hao.Shi@link.cuhk.edu.hk}{\texttt{Hao.Shi@link.cuhk.edu.hk}};\quad
\href{mailto:12131253@mail.sustech.edu.cn}{\texttt{12131253@mail.sustech.edu.cn}}
\end{minipage}
\end{center}

\begin{abstract}
We prove quantitative limit theorems for the edge density $Q$ in the dense
edge--triangle exponential random graph model at criticality and under
parameter perturbations of order $n^{-1}$.
At the fixed critical parameters $(\alpha_c,h_c),$ the centered fluctuations satisfy
$n^{2/3}(Q-\E Q)\Rightarrow\calN(0,\sigma_0^2)$, with explicit variance
$\sigma_0^2>0$. This disproves the quartic fluctuation conjecture of
\citet{BCM2024}.
Gaussian limits also hold  $(\alpha_c+ \lambda/n,h_c + \gamma/n)$,
except when the coefficients satisfy $-\frac{11}{16}+\frac{2}{9}(\gamma+\frac49\lambda)=0$.
Under this linearity condition, the limiting laws of $\sqrt n(Q-2/3)$ have densities proportional to
$\exp\{-81z^4/64+9\theta z^2/8\}$, where $\theta=-57/16+8\lambda/27$.
We obtain a sharp $n^{-1/3}$ rate in $1$-Wasserstein distance for the
Gaussian limits and an $O(n^{-1/2})$ bound for the quartic limits.
The proofs combine Stein's method for exchangeable pairs with conditional
row estimates and a local  Dobrushin contraction.
\end{abstract}

\medskip

\medskip

\section{Introduction}\label{sec:introduction}
Exponential random graph models (ERGMs) assign probabilities to graphs through an exponential family whose sufficient statistics are some specific subgraph counts. They allow dependence between edges through local interactions that affect global graph structure; see, for example, \citet{HL1981,FS1986,PN2004}. In this paper, we focus on the dense edge--triangle ERGM, in which the sufficient statistics are the numbers of edges and triangles.

Let \([n]=\{1,\ldots,n\}\) and \(N=\binom n2\). We identify a simple undirected graph\footnote{ no loops or multiple edges} on \([n]\) by a vector of edge indicators
\[
x=(x_{ij})_{1\le i<j\le n}\in\mathcal X_n:=\{0,1\}^{N},
\]
where \(x_{ij}=1\) if the edge \(\{i,j\}\) is present and \(x_{ij}=0\) otherwise. Thus the random graph is represented by the random vector \(X=(X_{ij})_{1\le i<j\le n}\) taking values in \(\mathcal X_n\).
For \(x\in\mathcal X_n\), define
\[
S(x)=\sum_{1\le i<j\le n}x_{ij},
\qquad
T(x)=\sum_{1\le i<j<k\le n}x_{ij}x_{ik}x_{jk}.
\]
The edge--triangle ERGM has probability law
\begin{equation}\label{eq:triangle-model}
 \PP_n(X=x)=\mathcal Z_n^{-1}\exp\{H_n(x)\},\qquad
 H_n(x)=h_nS(x)+\frac{\alpha_n}{n}T(x).
\end{equation}
where \(\mathcal Z_n=\sum_{x\in\mathcal X_n}\exp\{H_n(x)\}\) is the normalizing constant.
The factor \(1/n\) places both interaction terms on the \(n^2\) scale in the dense regime.
When \(\alpha_n=0\), the model reduces to the Erd\H{o}s--R\'enyi model
\(G(n,e^{h_n}/(1+e^{h_n}))\).
We study the edge density \(Q=S(X)/N\) at the critical endpoint and under finite-size perturbations of the parameters.

For fixed parameters \(h_n=h\in\R\) and \(\alpha_n=\alpha\ge0\), write
\(\mathcal Z_n(\alpha,h)\) for the normalizing constant and define the
variational function, with $0\log0=0$, by
\begin{equation}\label{eq:variational-function}
 \Phi_{\alpha,h}(q)=\frac h2q+\frac\alpha6q^3
 -\frac12\{q\log q+(1-q)\log(1-q)\},\qquad 0\le q\le1.
\end{equation}
\citet{CD2013} showed that the limiting free energy
$\psi(\alpha,h)$ satisfies
\begin{equation}\label{eq:variational-free-energy}
 \psi(\alpha,h)=\lim_{n\to\infty}\frac1{n^2}\log\mathcal Z_n(\alpha,h)
 =\max_{0\le q\le1}\Phi_{\alpha,h}(q).
\end{equation}
If $\Phi_{\alpha,h}$ has a unique maximizer $q_*$, then $Q\to q_*$
in probability.
\citet{RY2013} identified a first-order transition curve.
Its endpoint parameters, denoted by $\alpha_c$ and $h_c$, are
\begin{equation}\label{eq:critical-point}
 \alpha_c=\frac{27}{8},\qquad h_c=\log2-\frac32.
\end{equation}
For $\alpha>\alpha_c$, write $h_{\mathrm{coex}}(\alpha)$ for the edge
parameter on this curve. At $h=h_{\mathrm{coex}}(\alpha)$, the two global
maximizers $q_-<q_+$ satisfy
\begin{equation}\label{eq:coexistence-values}
 \Phi_{\alpha,h_{\mathrm{coex}}(\alpha)}(q_-)
 =\Phi_{\alpha,h_{\mathrm{coex}}(\alpha)}(q_+)
 =\psi\bigl(\alpha,h_{\mathrm{coex}}(\alpha)\bigr).
\end{equation}
They describe the low-density and high-density phases. Crossing the curve
changes the limiting edge density between them. The two maximizers merge
at the endpoint.

The limiting free energy describes the macroscopic phase diagram, but a separate analysis is needed to determine fluctuations. For the edge--triangle ERGM, \citet[Theorem~3.6 and Remark~3.7]{BCM2024} established a central limit theorem, with an explicit variance, outside the coexistence curve and its endpoint in the replica-symmetric region. Their result applies, in particular, to the noncritical parameters with \(\alpha\ge0\) considered here. Quantitative normal approximation for general subcritical ERGMs, including fixed subgraph counts, was developed by \citet{FLSZ2025}. In the subcritical edge--triangle ERGM, the limiting density \(q_*\) is the unique solution of
\[
q_*=\frac{e^{h+\alpha q_*^2}}{1+e^{h+\alpha q_*^2}},
\]
and the edge-density limit is
\[
n(Q-EQ)\Rightarrow
N\!\left(0,\frac{2q_*(1-q_*)}{1-2\alpha q_*^2(1-q_*)}\right).
\]
The denominator in the variance is positive in this regime. \citet{Winstein2026} obtains quantitative Gaussian limits beyond the subcritical region, including noncritical phases with a unique global maximizer for the free energy. On the coexistence curve, \citet{MW2026} establish Gaussian limits within individual phases, using phase-conditioned means. These results leave the critical endpoint as a distinct fluctuation problem.

At the endpoint, the critical edge density is $2/3$
\citep[Proposition~3.2]{RY2013}.
For real $t\to0$, expansion at $2/3$ gives
\begin{equation}\label{eq:critical-variational-expansion}
 \Phi_{\alpha_c,h_c}(2/3+t)-\Phi_{\alpha_c,h_c}(2/3)=-\frac{81}{64}t^4+O(t^5).
\end{equation}
The variational function has a quartic maximum. The fluctuation law at
this endpoint requires a separate analysis.
A useful comparison is provided by the edge--two-star ERGM, whose interaction counts paths of length two rather than triangles. Write $Q_\star$ for its
edge density and $Z_\star$ for the critical limit variable, with density
$\pi_\star$. \citet{MX2023} proved that, at the critical
parameters,
\begin{equation}\label{eq:two-star-critical-comparison}
 2\sqrt n\left(Q_\star-\frac12\right)\Rightarrow Z_\star,\qquad
 \pi_\star(z)\propto
 \exp\left\{-\frac{z^2}{2}-\frac{z^4}{24}\right\},\quad z\in\R.
\end{equation}
They rewrite the edge--two-star ERGM as an Ising model on the line graph
of the complete graph and analyze the resulting spin system using
auxiliary variables. This reformulation exploits the quadratic
interaction in the edge indicators. The triangle interaction is
cubic, so the same reduction does not directly apply to the
edge--triangle ERGM.

For the edge--triangle ERGM,  \citet{BCM2024} studied a mean-field approximation whose Hamiltonian depends only on the edge count and whose limiting free energy agrees with that of the original model. They proved a quartic critical limit for this approximation and conjectured the same behavior for the original graph law at the fixed critical parameters (Conjecture 3.8):
\begin{equation}\label{eq:BCM-critical-conjecture}
 \sqrt n(Q-\E Q)\Rightarrow Z_0,\qquad
 \pi_0(z)\propto\exp\left\{-\frac{81}{64}z^4\right\},
 \quad z\in\R.
\end{equation}
The distinction between the original model and its approximation is essential here: agreement of the limiting free energies does not by itself control the lower-order terms related to critical fluctuations.

At the fixed critical parameters, we obtain a different fluctuation scale and limit law. With \(a_0=(11/36)^{1/3}\) and \(\sigma_0^2=16/(243a_0^2)\), Theorem \ref{thm:triangle-gaussian} yields
\begin{equation}\label{eq:intro-gaussian-limit}
      EQ=\frac23-a_0n^{-1/3}+O(n^{-2/3}),
      \qquad
      n^{2/3}(Q-EQ)\Rightarrow N(0,\sigma_0^2).
\end{equation}
In particular, \(\sqrt n(Q-EQ)\to0\) in probability, disproving the nondegenerate limit in \citet[Conjecture 3.8]{BCM2024}.

The Gaussian limit also holds under fixed perturbations of order \(n^{-1}\) around the critical parameters. More precisely, for fixed coefficients \(\lambda,\gamma\), let \(\alpha_n=\alpha_c+\lambda/n\) and \(h_n=h_c+\gamma/n\). Theorem \ref{thm:triangle-gaussian} gives Gaussian limits unless
\[
\tau:= -\frac{11}{16}+\frac{2}{9}(\gamma+\frac49\lambda)=0.
\]
The condition \(\tau=0\) imposes one linear constraint on the perturbation coefficients
\((\lambda,\gamma)\), leaving one degree of freedom along the line
\(\gamma+4\lambda/9=99/32\). We parameterize this line by
\(\theta=-57/16+8\lambda/27\). Theorem~\ref{thm:triangle-quartic} shows that
the fluctuations have scale \(n^{-1/2}\) and limiting density
\[
\pi_{\theta}(z)\propto
\exp\left\{-\frac{81}{64}z^4+\frac{9\theta}{8}z^2\right\},
\qquad \theta\in\mathbb R.
\]
The pure quartic density \(\pi_0\) corresponds to \(\theta=0\).

We also obtain a sharp \(n^{-1/3}\) rate in 1-Wasserstein distance for the Gaussian limits and an \(O(n^{-1/2})\) bound for the critical-window limits. Theorem~\ref{thm:subgraph-limits} gives the corresponding joint limits for fixed subgraph densities.

The proofs use Stein's method with an exchangeable pair formed by
resampling one edge given the others \citep{CS2011,EL2010,SZ2019}.
To estimate the pair's conditional moments, we choose a vertex set
$\mathcal I$ of fixed size $k$ containing the relevant vertices. Let $\mathcal F_{\mathcal I}$ fix the
edges within $\mathcal I$ and within $\mathcal I^c$, and write
$B=(X_{uv})_{u,v\in\mathcal I^c}$, $e_{ab}=X_{ab}$ for $a,b\in\mathcal I$,
and $\mathbf z_a=(X_{au})_{u\in\mathcal I^c}$, viewed as column vectors.
With $(\alpha,h)=(\alpha_n,h_n)$, the conditional law of
$\mathbf z=(\mathbf z_a)_{a\in\mathcal I}$ is
\[
\begin{aligned}
p_{B,1}(\mathbf z)
&\propto\exp\left\{
\sum_{a\in\mathcal I}\left(
h\mathbf1^\top\mathbf z_a+\tfrac12\mathbf z_a^\top J_B\mathbf z_a
\right)+J_{\mathrm{cross}}\right\},\\
J_B&=\frac{\alpha}{n}B,
\qquad J_{\mathrm{cross}}=\frac{\alpha}{n}
\sum_{\substack{a<b\\a,b\in\mathcal I}}
e_{ab}\mathbf z_a^\top\mathbf z_b.
\end{aligned}
\]
This conditional system satisfies a local Dobrushin condition: its influence
matrix obeys $\|\mathcal C\|_{\mathrm{op}}\le|\alpha|/4$, with
$\alpha_c/4=27/32<1$ at the critical endpoint. Lemma~\ref{lem:comparison} therefore gives
concentration and covariance bounds uniform in the conditioned edges for
parameters near the endpoint.

The summation in $J_{\mathrm{cross}}$ involves only the order of $n$
terms, fewer than in the within-row interaction involving $J_B$.
Thus, removing $J_{\mathrm{cross}}$ gives $p_{B,0}=p_B^{\otimes k}$,
a product measure, at an affordable cost. More precisely,
we bound the change in expectations by interpolating along
$p_{B,t}\propto e^{tJ_{\mathrm{cross}}}p_{B,0}$, $0\le t\le1$.
Writing $\mathbb E_t$ for expectation under $p_{B,t}$,
Lemma~\ref{lem:row-interpolation} in Section~\ref{sec:conditional-rows}
gives, for any observable $U$ of the free edges,
\[
 \mathbb E_1U-\mathbb E_0U =\int_0^1\operatorname{Cov}_t(U,J_{\mathrm{cross}})\,dt,\qquad |\mathbb E_1U-\mathbb E_0U| \le C_k n^{-1/2}\|\boldsymbol\delta U\|_{\ell^2},
\]
where $\boldsymbol\delta U$ records the maximal absolute changes in $U$ when individual
free edges are flipped. The bound follows from the same Dobrushin estimate
along the interpolation and
$\|\boldsymbol\delta J_{\mathrm{cross}}\|_{\ell^2}=O_k(n^{-1/2})$.
Compared with an independent-edge reference, this product of row laws preserves
the within-row covariances. The moment estimates in
Lemmas~\ref{thm:fields}--\ref{thm:tested-triangle} give the finite-size drift
correction and error bounds in Proposition~\ref{prop:conditional-moments}.

\noindent\textbf{Notation.}
We write
$\ell(t)=\ee^t/(1+\ee^t)$
for the logistic function.
All limits are taken as $n\to\infty$, and $\Rightarrow$ denotes
convergence in distribution.
{For a probability measure $p$ and $1\le r<\infty$, write
\[
 \norm{Y}_{L^r(p)}=\left(\int |Y|^r\,dp\right)^{1/r}.
\]
We use the abbreviation $\norm{Y}_r=\norm{Y}_{L^r(\PP_n)}$
only for the graph law $\PP_n$ in \eqref{eq:triangle-model}.
For every other probability measure, we retain the full subscript $L^r(p)$.
In particular, for a sigma-field $\mathcal F$, set
\[
 \norm{Y}_{L^r(\PP_n(\cdot\mid\mathcal F))}
 =\bigl(\E[|Y|^r\mid\mathcal F]\bigr)^{1/r}.
\]
This conditional norm is an $\mathcal F$-measurable random variable.
For a positive deterministic sequence $a_n$, the notation
$Y_n=O_{L^r}(a_n)$ means $\norm{Y_n}_r=O(a_n)$, and
$Y_n=o_{L^r}(a_n)$ means $\norm{Y_n}_r=o(a_n)$.

Vectors are written in boldface, and their coordinates in ordinary type.
For a vector $\mathbf v=(v_i)$, define
\[
 \norm{\mathbf v}_{\ell^2}=\left(\sum_i|v_i|^2\right)^{1/2},
 \qquad
 \norm{\mathbf v}_{\ell^\infty}=\max_i|v_i|.
\]
For a matrix $A$, $\norm{A}_{\mathrm{op}}
=\sup_{\norm{\mathbf v}_{\ell^2}=1}\norm{A\mathbf v}_{\ell^2}$
is the induced operator norm. These vector and matrix norms are
evaluated pointwise when their arguments are random.
For a function $f$ on its stated domain $D$, write
$\norm f_\infty=\sup_{x\in D}|f(x)|$.
The space $C_b^j(\R)$ consists of functions with bounded continuous
derivatives through order $j$, with
$\norm{f}_{C_b^j}=\sum_{i=0}^j\norm{f^{(i)}}_\infty$.
For $f\in C^j([0,1])$, we use the same sum to define
$\norm{f}_{C^j}$, with each supremum taken over $[0,1]$.}
Constants denoted by $C$ may vary from line to line and are independent of $n$;
their dependence on fixed parameters, moment orders or graphs is indicated
where needed. For the scalar families $F_{\alpha,h}$, $g_{\alpha,h}$,
$A_{\alpha,h}$ and $K_{\alpha,h}$ defined below, we omit the parameter
subscripts only at $(\alpha_c,h_c)$, writing
$F=F_{\alpha_c,h_c}$, $g=g_{\alpha_c,h_c}$,
$A=A_{\alpha_c,h_c}$ and $K=K_{\alpha_c,h_c}$.

\noindent\textbf{Organization.}
Section~\ref{sec:main} states the limit theorems for edge and subgraph
densities.  In Section~\ref{sec:proofs}, we gives the common Stein bound and proves
these results. Our main technical lemma is stated in Section~\ref{sec:conditional-rows}. 
Section~\ref{sec:conditional-row-proofs} gives their proofs.
Sections~\ref{sec:moment-reduction} and~\ref{sec:motif-proof} prove the useful propositions. The remaining lemma proofs are given in
Section~\ref{sec:lemma-proofs}.

\section{Main results}\label{sec:main}
Let $X$ be a random simple graph on $[n]=\{1,\ldots,n\}$
with law \eqref{eq:triangle-model}. We identify $X$ with its
adjacency matrix $(X_{ij})_{i,j=1}^n$, where
$X_{ij}=X_{ji}\in\{0,1\}$ and $X_{ii}=0$.
All expectations $\E$ are taken under this law, and all limits are
as $n\to\infty$.
Let $S$ and $T$ denote the numbers of edges and triangles in $X$:
\begin{equation}\label{eq:graph-statistics}
 S=\sum_{1\le i<j\le n}X_{ij},\qquad
 T=\sum_{1\le i<j<k\le n}X_{ij}X_{ik}X_{jk}.
\end{equation}
Write $N=\binom n2$ for the number of possible edges and
$Q=S/N$ for the edge density.

For the fixed critical parameters, write
\begin{equation}\label{eq:fixed-critical-constants}
 a_0=\left(\frac{11}{36}\right)^{1/3},\qquad
 \sigma_0^2=\frac{16}{243a_0^2}.
\end{equation}
For fixed $\lambda,\gamma\in\R$, put
\begin{equation*}
 \alpha_n=\alpha_c+\lambda/n,\qquad h_n=h_c+\gamma/n,\qquad
 \tau=-\frac{11}{16}+\frac29\left(\gamma+\frac49\lambda\right).
\end{equation*}

\begin{theorem}\label{thm:triangle-gaussian}
Consider the edge--triangle ERGM \eqref{eq:triangle-model} with the
parameters above, and suppose $\tau\ne0$.
Let $a=(4\tau/9)^{1/3}$, and set
\begin{equation*}
 c_2=-\frac{a^2}{8}
      +\frac{4}{27a}\left(-\frac{81}{32}-\frac{\gamma}{3}
                          +\frac{4\lambda}{27}\right),
 \qquad \sigma_a^2=\frac{16}{243a^2}.
\end{equation*}
Then
\begin{equation}\label{eq:critical-mean}
\begin{aligned}
 \E Q=\frac23+an^{-1/3}
 +c_2n^{-2/3}+o(n^{-2/3}), \qquad
 n^{4/3}\var(Q)&\longrightarrow\sigma_a^2,
\end{aligned}
\end{equation}
and
\begin{equation}\label{eq:critical-Gaussian}
 cn^{-1/3}\le
 d_{\mathrm W}\!\left(\calL(n^{2/3}(Q-\E Q)),\calN(0,\sigma_a^2)\right)
 \le Cn^{-1/3}.
\end{equation}
Here $0<c\le C<\infty$ may depend on $\lambda,\gamma$ but not on $n$; the bounds hold for all
sufficiently large $n$, and $d_{\mathrm W}$ is the
$1$-Wasserstein distance:
$d_{\mathrm W}(\mu,\nu)=\sup_{\operatorname{Lip}(\varphi)\le1}
|\mu\varphi-\nu\varphi|$, where $\mu\varphi=\int\varphi\,\dd\mu$.
The same distributional limit holds with center
$2/3+an^{-1/3}+c_2n^{-2/3}$.
\end{theorem}

Theorem~\ref{thm:triangle-gaussian} shows that
\(\E Q-2/3\) is of order \(n^{-1/3}\), whereas the standard
deviation of \(Q\) is of order \(n^{-2/3}\).
Consequently, the term \(c_2n^{-2/3}\) cannot in general be
omitted when choosing an explicit deterministic center.
Indeed,
\[
 n^{2/3}\bigl(Q-2/3-an^{-1/3}\bigr)
 \Rightarrow \calN(c_2,\sigma_a^2).
\]
Including the correction \(c_2n^{-2/3}\) therefore yields
a centered Gaussian limit with an explicit deterministic center.

At the fixed critical parameters, \(\lambda=\gamma=0\), we have
\(\tau=-11/16\), \(a=-a_0\), \(c_2=97/(288a_0)\), and
\(\sigma_a^2=\sigma_0^2\).
In this case, Theorem~\ref{thm:triangle-gaussian} recovers
the Gaussian limit in \eqref{eq:intro-gaussian-limit}.
The mean expansion in \eqref{eq:critical-mean} refines
the corresponding statement by identifying the coefficient
of the \(n^{-2/3}\) term.

The two-sided bound in \eqref{eq:critical-Gaussian} establishes
the sharp order of the Wasserstein error for the stated
normalization and fixed Gaussian target.
The coefficients \(\lambda,\gamma\) are fixed throughout;
the theorem does not give a uniform approximation as
\(\tau\to0\).

\begin{remark}\label{rem:critical-comparison}
At \((\alpha_c,h_c)\),
Theorem~\ref{thm:triangle-gaussian} implies
\[
 \sqrt n\,(Q-\E Q)
 =n^{-1/6}\bigl[n^{2/3}(Q-\E Q)\bigr]
 \longrightarrow 0\qquad\text{in probability}.
\]
Thus our result disproves the critical fluctuation conjecture
of \citet{BCM2024} for the original edge--triangle ERGM.
Their quartic limit theorem applies to a mean-field
approximation whose Hamiltonian depends only on the edge count.
For the original model, our proof identifies the finite-size
drift correction and controls the remainder needed to establish
the Gaussian limit.
Even when \((\alpha_n,h_n)\) lies in the subcritical region
for all sufficiently large \(n\), our Gaussian limit describes
a critical-window regime.
The parameters approach \((\alpha_c,h_c)\) at rate \(n^{-1}\),
whereas the Gaussian limits of \citet{BCM2024} and
\citet{FLSZ2025} concern fixed noncritical parameters.
The fluctuation scale here is \(n^{-2/3}\), compared with
\(n^{-1}\) in those fixed-parameter limits.
\end{remark}

{
The remaining case is the compensation line \(\tau=0\), equivalently \(\gamma+4\lambda/9=99/32\). Parameterizing this line by \(\theta=-57/16+8\lambda/27\) leads to the following critical window.
}

\begin{theorem}\label{thm:triangle-quartic}
Let $\theta\in\R$ be fixed, and take
\begin{equation}\label{eq:cancellation-line}
 \alpha_n=\alpha_c+\frac1n\left(\frac{1539}{128}+\frac{27}{8}\theta\right),\qquad h_n=h_c-\frac1n\left(\frac94+\frac32\theta\right).
\end{equation}
For these parameters, $\tau=0$.
Let $Z_{\theta}$ have density $\pi_{\theta}$,
with normalizing constant $\mathcal Z_{\theta}$, given by
\begin{equation}\label{eq:quartic-window}
 \pi_{\theta}(z)=\mathcal Z_{\theta}^{-1}
 \exp\left\{-\frac{81}{64}z^4+\frac{9\theta}{8}z^2\right\}.
\end{equation}
The density is even, so $\E Z_{\theta}=0$. Then
\begin{equation}\label{eq:quartic-moments}
\begin{aligned}
 \E Q=\frac23+o(n^{-1/2}), \qquad
 n\var(Q)&\to\var(Z_{\theta}),
\end{aligned}
\end{equation}
and
\begin{equation}\label{eq:quartic-limit}
 d_{\mathrm W}\!\left(\calL(\sqrt n(Q-2/3)),\calL(Z_{\theta})\right)
 \le C_{\theta}n^{-1/2}.
\end{equation}
The constant $C_{\theta}$ is independent of $n$.
\end{theorem}

The order-\(n^{-1}\) perturbations in~\eqref{eq:cancellation-line} satisfy
\[
 \gamma+\frac49\lambda=\frac{99}{32},
\]
which cancels the leading constant correction to the critical drift. This is a line in the \((\lambda,\gamma)\)-plane, rather than a single exceptional point. Along this line, \(\theta\) parametrizes the remaining degree of freedom and determines the quadratic term in the limiting log density.

The limiting density is symmetric, so \(\E Z_{\theta}=0\); centering by the true mean gives
\[
 \sqrt n\,(Q-\E Q)\Rightarrow Z_{\theta}.
\]
The pure quartic density in~\eqref{eq:BCM-critical-conjecture} is obtained at \(\theta=0\), corresponding to
\[
 \alpha_n=\alpha_c+\frac{1539}{128n},
 \qquad h_n=h_c-\frac{9}{4n}.
\]

\begin{remark}\label{rem:two-star-comparison}
\citet[Theorem~1.4(c)]{MX2023} prove a non-Gaussian limit for the edge density $Q_\star$ at the fixed critical parameters of the edge--two-star ERGM. They showed that
\[
 2\sqrt n\,(Q_\star-1/2)\Rightarrow Z_\star,
 \qquad \pi_\star(z)\propto
 \exp\{-z^2/2-z^4/24\}.
\]
The quadratic term is part of their limiting density and cannot be omitted. In the edge--triangle ERGM, Theorem~\ref{thm:triangle-gaussian} instead gives Gaussian fluctuations at the fixed critical parameters, whereas Theorem~\ref{thm:triangle-quartic} gives quartic limits.
There is also a methodological distinction. Mukherjee and Xu rewrite the edge--two-star ERGM as an Ising model on the line graph of the complete graph and analyze the resulting spin system using auxiliary variables. Our Hamiltonian is cubic in the edge indicators; we obtain tractable quadratic systems by conditioning on all edges except a fixed collection of rows, retaining dependence within each row.
\end{remark}

{
The edge-density limits in Theorems~\ref{thm:triangle-gaussian} and~\ref{thm:triangle-quartic}
also determine the leading fluctuations of fixed subgraph densities.
}
Let $H$ be
a fixed finite simple graph with $v_H$ vertices and $e_H$ edges.
Let $\inj(H,X)$ denote the number of injective maps
$\varphi:V(H)\to[n]$ such that
$X_{\varphi(u)\varphi(v)}=1$ for every $\{u,v\}\in E(H)$.
No condition is imposed on nonedges of $H$.
Define the injective subgraph density by
\begin{equation}\label{eq:motif-density}
 t_H=\frac{\inj(H,X)}{(n)_{v_H}},\qquad
 (n)_j=n(n-1)\cdots(n-j+1),\quad (n)_0=1.
\end{equation}
The denominator counts all injective maps from $V(H)$ to $[n]$,
so $t_H$ is the proportion of these maps that preserve every
edge of $H$. In particular,
\begin{equation*}
 t_{K_2}=Q,\qquad t_{K_3}=\frac{T}{\binom n3}.
\end{equation*}

The coefficients in the joint limits below are
\begin{equation*}
 \left.\frac{\mathrm d}{\mathrm d q}q^{e_H}\right|_{q=2/3}
     =e_H(2/3)^{e_H-1}.
\end{equation*}
Isolated vertices are allowed. If $H$ has no edges, then
$t_H=1$ and $e_H(2/3)^{e_H-1}=0$.

\begin{theorem}\label{thm:subgraph-limits}
Let $H_1,\ldots,H_m$ be a fixed finite list of finite simple graphs.

\noindent\textup{(i)} Under the assumptions of Theorem~\ref{thm:triangle-gaussian}, let
$Z\sim\calN(0,\sigma_a^2)$. Then, for $1\le i,j\le m$,
\begin{equation}\label{eq:joint-covariance}
 n^{4/3}\cov(t_{H_i},t_{H_j})
 \longrightarrow e_{H_i}e_{H_j}(2/3)^{e_{H_i}+e_{H_j}-2}\sigma_a^2,
\end{equation}
and
\begin{equation}\label{eq:joint}
 \bigl(n^{2/3}(t_{H_j}-\E t_{H_j})\bigr)_{j=1}^m
 \Rightarrow (e_{H_j}(2/3)^{e_{H_j}-1}Z)_{j=1}^m.
\end{equation}

\noindent\textup{(ii)} Under the assumptions of Theorem~\ref{thm:triangle-quartic},
we have, for $1\le i,j\le m$,
\begin{equation}\label{eq:joint-quartic-covariance}
 n\cov(t_{H_i},t_{H_j})
 \longrightarrow e_{H_i}e_{H_j}(2/3)^{e_{H_i}+e_{H_j}-2}\var(Z_{\theta}),
\end{equation}
and
\begin{equation}\label{eq:joint-quartic}
 \bigl(\sqrt n(t_{H_j}-\E t_{H_j})\bigr)_{j=1}^m
 \Rightarrow
 \bigl(e_{H_j}(2/3)^{e_{H_j}-1}Z_{\theta}\bigr)_{j=1}^m.
\end{equation}
\end{theorem}

\begin{remark}\label{rem:subgraph-fluctuations}
At either scale in Theorem~\ref{thm:subgraph-limits}, the leading fluctuation of \(t_H\) is \(e_H(2/3)^{e_H-1}\) times that of \(Q\), which is similar to that in \citet{FLSZ2025}.
\end{remark}

\section{Proofs of the main results}\label{sec:proofs}\label{sec:stein}
Let $W_n=r_n(Q-q_n)$, with a deterministic center $q_n$ and scale $r_n>0$.
The center and scale are specified in each proof.
For a target density $\pi$ and a $1$-Lipschitz test function
$\varphi:\R\to\R$, let $f$ be the bounded solution of the Stein equation
\citep{CS2011,SZ2019}
\begin{equation}\label{eq:stein-equation}
 \frac49f'(z)+b(z)f(z)=\varphi(z)-\pi\varphi,
\end{equation}
where $\pi\varphi=\int_{\R}\varphi(z)\pi(z)\dd z$ and
\begin{equation}\label{eq:stein-target}
 b(z)=
 \begin{cases}
  -27a^2z/4,&\pi\text{ is the density of }\calN(0,\sigma_a^2),\\
  -9z^3/4+\theta z,&\pi=\pi_{\theta}.
 \end{cases}
\end{equation}
Here $C_b^j$ denotes the functions with bounded continuous
derivatives through order $j$, with
$\norm{\varphi}_{C_b^j}=\sum_{i=0}^j\norm{\varphi^{(i)}}_\infty$.

We construct an exchangeable pair by one step of the single-edge
heat-bath Glauber dynamics. For each possible edge
$e=\{i,j\}$, write $X_e=X_{ij}$ and $X_{-e}$ for all the other edges.
Choose $I$ uniformly from the $N$ possible edges, independently of $X$.
Given $I$ and $X_{-I}$, let $X_I'$ be an independent copy of $X_I$,
and set
\begin{equation*}
 S'=S-X_I+X_I',\qquad Q'=S'/N,\qquad W_n'=r_n(Q'-q_n).
\end{equation*}
Then $(S,S')$ and $(W_n,W_n')$ are exchangeable pairs, and
\begin{equation*}
 Q'-Q=\frac{X_I'-X_I}{N},\qquad
 \Delta=W_n'-W_n=r_n(Q'-Q).
\end{equation*}
For each possible edge $e$, define
\begin{equation*}
 L_e=\PP(X_e=1\mid X_{-e}).
\end{equation*}
By the resampling construction, conditional on $X$ and $I=e$,
$X_I'$ has distribution $\operatorname{Bernoulli}(L_e)$. In particular,
\begin{equation*}
 \E(X_I'\mid X,I=e)
 =\E((X_I')^2\mid X,I=e)=L_e.
\end{equation*}
Since $X_e$ is fixed when conditioning on $X$ and $X_e^2=X_e$,
\begin{align*}
 \E(Q'-Q\mid X,I=e) &=\frac{\E(X_I'\mid X,I=e)-X_e}{N} =\frac{L_e-X_e}{N},\\
 \E((Q'-Q)^2\mid X,I=e) &=\frac{L_e-2X_eL_e+X_e}{N^2} =\frac{X_e(1-L_e)+(1-X_e)L_e}{N^2}.
\end{align*}
Write $\bar L=N^{-1}\sum_eL_e$.
Averaging over $I$ yields
\begin{equation}\label{eq:heatbath}
\begin{aligned}
 b_n&=N\E(Q'-Q\mid X)
       =\frac1N\sum_e(L_e-X_e)=\bar L-Q,\\
 v_n&=N^2\E((Q'-Q)^2\mid X)
       =\frac1N\sum_e\{X_e+L_e-2X_eL_e\}.
\end{aligned}
\end{equation}
Multiplying these conditional moments by $r_n$ and $r_n^2$,
respectively, gives
\begin{equation}\label{eq:jumps}
 \E(\Delta\mid X)=\frac{r_n}{N}b_n,\qquad \E(\Delta^2\mid X)=\frac{r_n^2}{N^2}v_n, \qquad |\Delta|\le r_n/N, \qquad 0\le v_n\le1.
\end{equation}
The last two bounds follow from $|X_I'-X_I|\le1$.
For every $f\in C_b^1(\R)$ with $\operatorname{Lip}(f')<\infty$,
exchangeability gives
\begin{equation*}
 0=\E[\Delta\{f(W_n')+f(W_n)\}].
\end{equation*}
Since
$|f(x+t)-f(x)-tf'(x)|\le\tfrac12\operatorname{Lip}(f')t^2$,
we obtain
\begin{equation}\label{001}
 \left|2\E[\Delta f(W_n)]
       +\E[\Delta^2f'(W_n)]\right|
 \le\frac12\operatorname{Lip}(f')\E|\Delta|^3.
\end{equation}
Substituting \eqref{eq:jumps} into \eqref{001}, multiplying by $Nn^2/(2r_n^2)$,
and using $\E|\Delta|^3\le(r_n/N)^3$, we obtain
\begin{equation}\label{eq:exchangeable-identity}
 \left|\E\left[
       \frac{n^2}{r_n}b_nf(W_n)+\frac{n^2}{2N}v_nf'(W_n)
       \right]\right|
 \le\frac{n^2r_n}{4N^2}\operatorname{Lip}(f').
\end{equation}
Taking expectations in \eqref{eq:stein-equation}, adding and subtracting
the expectation on the left above, and applying the triangle inequality
yields
\begin{equation}\label{eq:stein-master}
 |\E\varphi(W_n)-\pi\varphi|
 \le\mathrm I+\mathrm{II}+\mathrm{III},
\end{equation}
where
\begin{align*}
 \mathrm I&=\left|\E\left[
       \left\{b(W_n)-\frac{n^2}{r_n}b_n\right\}f(W_n)\right]\right|,\\
 \mathrm{II}&=\norm{f'}_\infty
       \E\left|\frac49-\frac{n^2}{2N}v_n\right|,\qquad
 \mathrm{III}=\frac{n^2r_n}{4N^2}\operatorname{Lip}(f').
\end{align*}
The following lemma gives bounds for the solutions of
\eqref{eq:stein-equation}. Its proof is given in
Section~\ref{sec:smooth-factors}.
\begin{lemma}\label{lem:smooth-stein}
For either fixed target in \eqref{eq:stein-target} and every
$\varphi:\R\to\R$ with $\operatorname{Lip}(\varphi)\le1$,
the unique bounded solution of \eqref{eq:stein-equation} is
\begin{equation}\label{eq:stein-solution}
 f(z)=\frac{9}{4\pi(z)}
       \int_{-\infty}^{z}\{\varphi(t)-\pi\varphi\}\pi(t)\dd t.
\end{equation}
This solution belongs to $C_b^1(\R)$, its derivative is Lipschitz, and
\begin{equation}\label{eq:smooth-stein-factors}
 \norm f_\infty+\norm{f'}_\infty+\operatorname{Lip}(f')\le C.
\end{equation}
The constant depends only on the fixed target parameters.

For the quartic target and $\varphi(z)=z^j$, $j=1,2$,
\eqref{eq:stein-solution} also defines a solution in $C_b^2(\R)$, with
\begin{equation}\label{eq:quartic-polynomial-stein-factors}
 \norm{f}_\infty+\norm{f'}_\infty+\operatorname{Lip}(f')\le C.
\end{equation}
Here $C$ depends only on $\theta$.
\end{lemma}

\subsection{Proof of Theorem~\ref{thm:triangle-gaussian}}\label{sec:gaussian-proof}

{
Write $b_n=\mathcal B_n(Q)+\varepsilon_n$.
We first use an auxiliary lemma to locate the zero $u_n$ of $\mathcal B_n$.
Stationarity and the dissipation estimate then give
\[
\begin{aligned}
 u_n&=\frac23+an^{-1/3}+c_2n^{-2/3}+O(n^{-1}),\\
 \E(Q-u_n)^2&=O(n^{-4/3}),\qquad
 |\E Q-u_n|=O(n^{-5/6}).
\end{aligned}
\]
The last bound gives the mean expansion.
Substituting $n^{2/3}\mathcal B_n'(u_n)\to-27a^2/4$ and
$\E v_n\to4/9$ into the quadratic stationarity identity gives the variance limit.
For $W_n=n^{2/3}(Q-\E Q)$, subtracting expectations from the drift
expansion gives a centered drift identity.
The moment bounds and the tested remainder estimate then bound the drift,
conditional variance and Taylor remainder terms in the Stein identity by
\[
 \mathrm I\le C(n^{-1/3}+n^{-1/2}),\qquad
 \mathrm{II}\le Cn^{-1/3},\qquad
 \mathrm{III}\le Cn^{-4/3}.
\]
The uniform Stein factors give the Wasserstein upper bound.
For the matching lower bound, retain the quadratic drift correction
$-27an^{-1/3}(W_n^2-\E W_n^2)/4$.
A nonzero, nonnegative even $f\in C_c^\infty(\R)$ supported where
$|z|>\sigma_a$ removes the limiting linear and noise corrections by parity,
while $\E[(Z^2-\sigma_a^2)f(Z)]>0$ for
$Z\sim\calN(0,\sigma_a^2)$.
Testing with the odd Lipschitz function
$\psi(z)=\frac49f'(z)-\frac{27a^2}{4}zf(z)$ detects this term.
\par}

\begin{proof}
Under the assumptions of Theorem~\ref{thm:triangle-gaussian},
$(\alpha_n,h_n)\to(\alpha_c,h_c)$.
The limiting critical edge density is $2/3$
\citep[Proposition~3.2]{RY2013}.
Constants in this proof may depend on the fixed $\lambda,\gamma$
with $\tau\ne0$.
Take $r_n=n^{2/3}$.
We begin with estimates for $b_n$ and $v_n$.
Define
\begin{equation}\label{eq:critical-F}
 F_{\alpha,h}(q)=\ell(h+\alpha q^2),\qquad
 F(q)=F_{\alpha_c,h_c}(q),\qquad 0\le q\le1.
\end{equation}

\begin{proposition}\label{prop:conditional-moments}
For each $q\in[0,1]$, let $\varrho_c(q)$ be the unique root
$r\in(0,1)$ of
\begin{equation}\label{eq:critical-row-root}
 r-\ell(h_c+\alpha_cqr)=0.
\end{equation}
Set
\begin{equation}\label{eq:critical-row-response}
 v_c(q)=\varrho_c(q)(1-\varrho_c(q)),\qquad A(q)=\frac{v_c(q)}{1-\alpha_cqv_c(q)},\qquad K(q)=v_c(q)^2+2\varrho_c(q)^2A(q).
\end{equation}
Define
\begin{equation}\label{eq:critical-g}
 g(q)=\alpha_cF(q)(1-F(q))\biggl[A(q)-q^2-q +\frac{\alpha_c}{2}(1-2F(q))K(q)\biggr].
\end{equation}
Then $g\in C^2([0,1])$.
Fix $\lambda,\gamma\in\R$ and take
\begin{equation}\label{eq:proof-parameter-family}
 \alpha_n=\alpha_c+\lambda/n,\qquad
 h_n=h_c+\gamma/n.
\end{equation}
\begin{enumerate}
\renewcommand{\labelenumi}{\textup{(\roman{enumi})}}
\item
The drift $b_n$ satisfies
\begin{equation}\label{eq:endpoint-closure}
 b_n={}F(Q)-Q+\frac{g(Q)}n +F(Q)(1-F(Q))\frac{\gamma+\lambda Q^2}{n}+\varepsilon_n, \qquad \norm{\varepsilon_n}_2\le Cn^{-3/2}.
\end{equation}
\item
The quantity $v_n$ in \eqref{eq:heatbath} satisfies
\begin{equation}\label{eq:taylor-noise}
 \left\lVert v_n-\{Q+(1-2Q)F(Q)\}\right\rVert_1\le Cn^{-1/2}.
\end{equation}
\item
The remainder $\varepsilon_n$ in \eqref{eq:endpoint-closure} satisfies, for every $f\in C_b^1(\R)$,
\begin{equation}\label{eq:quartic-tested-error}
 n^{3/2}\left|\E[\varepsilon_n f(\sqrt n(Q-2/3))]\right|
 \le Cn^{-1/2}\norm{f}_{C_b^1}.
\end{equation}
\end{enumerate}
The constant $C$ may depend on $\lambda$ and $\gamma$.
\end{proposition}
The proof is given in Section~\ref{sec:moment-reduction}.
These estimates are the main technical ingredients in the proofs of
Theorems~\ref{thm:triangle-gaussian} and \ref{thm:triangle-quartic}.
{
The \(L^2\) bound in \eqref{eq:endpoint-closure} is used in the mean and variance estimates.
The bound in \eqref{eq:quartic-tested-error} controls
\(\E[\varepsilon_n f(\sqrt n(Q-2/3))]\) in the Stein identity
and is key to obtaining the sharper Wasserstein bounds.
}

A straightforward calculation from \eqref{eq:critical-row-root},
\eqref{eq:critical-row-response} and
\eqref{eq:critical-g} gives
\begin{equation}\label{eq:g-values}
 g(2/3)=-\frac{11}{16},\qquad g'(2/3)=-\frac{81}{32}.
\end{equation}
Define
\begin{equation}\label{eq:critical-effective-drift}
 \beta(q)=g(q)+F(q)(1-F(q))(\gamma+\lambda q^2),\qquad \mathcal B_n(q)=F(q)-q+\frac{\beta(q)}n.
\end{equation}
Then $\beta(2/3)=\tau$ and $\beta\in C^2([0,1])$.
Exchangeability gives $\E Q'=\E Q$. Hence
\begin{equation}\label{eq:critical-mean-balance}
 0=N\E(Q'-Q)=\E b_n
   =\E\mathcal B_n(Q)+\E\varepsilon_n, \qquad
 \norm{\varepsilon_n}_2\le Cn^{-3/2}.
\end{equation}
We then prove the mean expansion in \eqref{eq:critical-mean}
from \eqref{eq:critical-mean-balance}
by expanding $\mathcal B_n$ and bounding the remainder.
The next lemma gives bounds for $F$ and locates the zero $u_n$ of $\mathcal B_n$.
The drift bound in \eqref{eq:triangle-dissipation} will control the moments of $Q-u_n$.

\begin{lemma}\label{lem:gaussian-center}
There is a constant $c>0$ such that
\begin{equation}\label{eq:triangle-global-derivative}
 1-F'(q)\ge c(q-2/3)^2,\qquad
 -(q-2/3)\{F(q)-q\}\ge\frac c3(q-2/3)^4
 \quad(0\le q\le1).
\end{equation}
For fixed $\lambda,\gamma$ with $\tau\ne0$ and all large $n$,
$\mathcal B_n$ has a unique zero
$u_n\in[0,1]$, with
\begin{equation}\label{eq:endpoint-root}
 u_n=\frac23+an^{-1/3}+c_2n^{-2/3}+O(n^{-1}).
\end{equation}
There is $c_*>0$ such that
\begin{equation}\label{eq:triangle-dissipation}
 -(q-u_n)\mathcal B_n(q)
 \ge c_*\{n^{-2/3}(q-u_n)^2+(q-u_n)^4\}
 \quad(0\le q\le1).
\end{equation}
We also have
\begin{equation}\label{eq:Gaussian-root-slope}
 \mathcal B_n'(u_n)=-\frac{27a^2}{4}n^{-2/3}+O(n^{-1}).
\end{equation}
\end{lemma}
The proof is given in Section~\ref{sec:center-moments}.\\
Since $\mathcal B_n(u_n)=0$, Taylor's formula gives
\begin{equation}\label{eq:Gaussian-root-Taylor}
\begin{aligned}
 \mathcal B_n(Q)&=\mathcal B_n'(u_n)(Q-u_n)+R_n,\\
 R_n&=(Q-u_n)^2\int_0^1(1-t)
 \left\{F''(u_n+t(Q-u_n))+\frac{\beta''(u_n+t(Q-u_n))}n\right\}\dd t.
\end{aligned}
\end{equation}
Substituting \eqref{eq:Gaussian-root-Taylor} into
\eqref{eq:critical-mean-balance} gives
\begin{equation}\label{eq:Gaussian-mean-identity}
 0=\mathcal B_n'(u_n)(\E Q-u_n)
       +\E R_n+\E\varepsilon_n.
\end{equation}

To bound $\E R_n$, we first estimate the moments of $Q-u_n$.
We derive these moment bounds by a recursive argument similar to the one
used by \citet{FLSW2025}.
We begin with the second moment.
Since $Q$ and $Q'$ have the same distribution, \eqref{eq:jumps} gives
\begin{equation*}
 0=\E\{(Q'-u_n)^2-(Q-u_n)^2\}
  =\frac2N\E[(Q-u_n)b_n]+\frac{\E v_n}{N^2}.
\end{equation*}
Substituting $b_n=\mathcal B_n(Q)+\varepsilon_n$, we obtain
\begin{equation}\label{eq:quadratic-stationarity}
 \E[-(Q-u_n)\mathcal B_n(Q)]
 =\E[(Q-u_n)\varepsilon_n]+\frac{\E v_n}{2N}.
\end{equation}
By \eqref{eq:triangle-dissipation} and
\eqref{eq:quadratic-stationarity},
\begin{equation*}
\begin{aligned}
 c_*\{n^{-2/3}\E(Q-u_n)^2+\E(Q-u_n)^4\}
 &\le \E[(Q-u_n)\varepsilon_n]+\frac{\E v_n}{2N}\\
 &\le\{\E(Q-u_n)^2\}^{1/2}\norm{\varepsilon_n}_2+\frac1{2N}\\
 &\le Cn^{-3/2}\{\E(Q-u_n)^2\}^{1/2}+Cn^{-2},
\end{aligned}
\end{equation*}
where in the second line we use the Cauchy--Schwarz inequality and $0\le v_n\le1$,
and the last inequality holds by \eqref{eq:endpoint-closure} and $N=n(n-1)/2$.
Dropping the fourth-moment term and multiplying by $n^{2/3}$ gives
\begin{equation*}
 \E(Q-u_n)^2\le Cn^{-5/6}\{\E(Q-u_n)^2\}^{1/2}+Cn^{-4/3} \le\frac12\E(Q-u_n)^2+C(n^{-5/3}+n^{-4/3}),
\end{equation*}
where the last step follows from Young's inequality. Hence
\begin{equation}\label{eq:Gaussian-tightness}
 \E(Q-u_n)^2\le Cn^{-4/3}.
\end{equation}

To bound the third moment, we first estimate the fourth moment.
Since $Q$ and $Q'$ have the same distribution,
\begin{equation*}
 0=\E\{(Q'-u_n)^4-(Q-u_n)^4\}.
\end{equation*}
Since $Q'-Q\in\{0,\pm N^{-1}\}$, we have
$(Q'-Q)^3=(Q'-Q)/N^2$ and $(Q'-Q)^4=(Q'-Q)^2/N^2$.
Expanding the fourth powers and conditioning on $X$ as in
\eqref{eq:jumps}, we obtain
\begin{equation}\label{eq:Gaussian-fourth-expansion}
\begin{aligned}
 0={}&4\E[(Q-u_n)^3(Q'-Q)]+6\E[(Q-u_n)^2(Q'-Q)^2]\\
 &+4\E[(Q-u_n)(Q'-Q)^3]+\E(Q'-Q)^4\\
 ={}&\frac4N\E[(Q-u_n)^3b_n] +\frac6{N^2}\E[(Q-u_n)^2v_n] +\frac4{N^3}\E[(Q-u_n)b_n]+\frac{\E v_n}{N^4}.
\end{aligned}
\end{equation}
Since $\E[(Q-u_n)b_n]=-\E v_n/(2N)$, the last two terms satisfy
\begin{equation*}
 \frac4{N^3}\E[(Q-u_n)b_n]+\frac{\E v_n}{N^4}
 =-\frac{2\E v_n}{N^4}+\frac{\E v_n}{N^4}
 =-\frac{\E v_n}{N^4}.
\end{equation*}
Substituting this simplification into
\eqref{eq:Gaussian-fourth-expansion}, multiplying by $N/4$,
and using $b_n=\mathcal B_n(Q)+\varepsilon_n$, we obtain
\begin{equation}\label{eq:Gaussian-fourth-stationarity}
 -\E[(Q-u_n)^3\mathcal B_n(Q)] =\E[(Q-u_n)^3\varepsilon_n] +\frac{3}{2N}\E[(Q-u_n)^2v_n]-\frac{\E v_n}{4N^3}.
\end{equation}
Multiplying \eqref{eq:triangle-dissipation} by $(Q-u_n)^2$ and using
\eqref{eq:Gaussian-fourth-stationarity}, we obtain
\begin{equation*}
\begin{aligned}
 c_*\{n^{-2/3}\E(Q-u_n)^4+\E(Q-u_n)^6\}
 &\le\norm{\varepsilon_n}_2\{\E(Q-u_n)^6\}^{1/2}
       +\frac{3}{2N}\E(Q-u_n)^2\\
 &\le\frac{c_*}{2}\E(Q-u_n)^6+Cn^{-3}+Cn^{-10/3},
\end{aligned}
\end{equation*}
where we used the Cauchy--Schwarz inequality and  Young's inequality,
\eqref{eq:endpoint-closure} and \eqref{eq:Gaussian-tightness}.
It then follows that
\begin{equation}\label{eq:Gaussian-root-moments}
 \E(Q-u_n)^4\le Cn^{-7/3},\qquad \E|Q-u_n|^3 \le\{\E(Q-u_n)^2\,\E(Q-u_n)^4\}^{1/2}\le Cn^{-11/6}.
\end{equation}

We now use these moment bounds to estimate the Taylor remainder in
\eqref{eq:Gaussian-root-Taylor}.
Differentiating \eqref{eq:critical-F} gives $F''(2/3)=0$. Since $F$ is smooth and $\beta\in C^2([0,1])$,
both $F'''$ and $\beta''$ are bounded on $[0,1]$.
By \eqref{eq:endpoint-root}, $u_n-2/3=O(n^{-1/3})$.
Combining this estimate with the boundedness of $F'''$ and $\beta''$ gives
\begin{equation}\label{002}
 |R_n|\le C\{n^{-1/3}(Q-u_n)^2+|Q-u_n|^3\}.
\end{equation}
Thus \eqref{eq:endpoint-closure}, \eqref{eq:Gaussian-tightness} and \eqref{eq:Gaussian-root-moments}
imply
\begin{equation}\label{eq:Gaussian-mean-remainders}
 \E|R_n|
 \le C\{n^{-1/3}\E(Q-u_n)^2+\E|Q-u_n|^3\}
 \le Cn^{-5/3},\qquad
 |\E\varepsilon_n|\le\norm{\varepsilon_n}_2\le Cn^{-3/2}.
\end{equation}
Substituting \eqref{eq:Gaussian-mean-remainders} into
\eqref{eq:Gaussian-mean-identity}, we obtain
\begin{equation}\label{eq:Gaussian-mean-error}
\begin{aligned}
 0&=\E b_n =\E\mathcal B_n(Q)+\E\varepsilon_n =\mathcal B_n'(u_n)(\E Q-u_n) +\E R_n+\E\varepsilon_n\\
 &=\mathcal B_n'(u_n)(\E Q-u_n) +O(n^{-5/3})+O(n^{-3/2}) =\mathcal B_n'(u_n)(\E Q-u_n)+O(n^{-3/2}).
\end{aligned}
\end{equation}
Since $|\mathcal B_n'(u_n)|\ge c n^{-2/3}$ by
\eqref{eq:Gaussian-root-slope}, \eqref{eq:Gaussian-mean-error} gives
\begin{equation}\label{eq:Gaussian-mean-root-gap}
 |\E Q-u_n|\le Cn^{-5/6}=o(n^{-2/3}).
\end{equation}
Together with \eqref{eq:endpoint-root}, this proves the mean expansion
in \eqref{eq:critical-mean}. It also justifies replacing $\E Q$
by $u_n$ or by $2/3+an^{-1/3}+c_2n^{-2/3}$ in the distributional limit.

We next prove the variance limit in \eqref{eq:critical-mean}.
By the Cauchy--Schwarz inequality, \eqref{eq:Gaussian-tightness} and \eqref{eq:endpoint-root},
\begin{equation*}
 \E|Q-2/3|\le\{\E(Q-u_n)^2\}^{1/2}+|u_n-2/3| \le Cn^{-2/3}+Cn^{-1/3}\le Cn^{-1/3}.
\end{equation*}
The function $q\mapsto q+(1-2q)F(q)$ has a bounded derivative on $[0,1]$.
Its value at $q=2/3$ is $4/9$, since $F(2/3)=2/3$.
Hence \eqref{eq:taylor-noise} gives
\begin{equation}\label{eq:Gaussian-noise-bound}
\begin{aligned}
 \norm{v_n-4/9}_1
 &\le\norm{v_n-\{Q+(1-2Q)F(Q)\}}_1
       +\E|Q+(1-2Q)F(Q)-4/9|\\
 &\le Cn^{-1/2}+C\E|Q-2/3|\le Cn^{-1/3}.
\end{aligned}
\end{equation}
Using \eqref{eq:Gaussian-root-moments}, the bound \eqref{002} on $R_n$
and the Cauchy--Schwarz inequality, we obtain
\begin{equation}\label{eq:Gaussian-variance-errors}
\begin{aligned}
 |\E[(Q-u_n)R_n]|
 &\le C\{n^{-1/3}\E|Q-u_n|^3+\E(Q-u_n)^4\}\le Cn^{-13/6},\\
 |\E[(Q-u_n)\varepsilon_n]|
 &\le\{\E(Q-u_n)^2\}^{1/2}\norm{\varepsilon_n}_2\le Cn^{-13/6}.
\end{aligned}
\end{equation}
Substituting \eqref{eq:Gaussian-root-Taylor} into
\eqref{eq:quadratic-stationarity}, we therefore obtain
\begin{equation}\label{eq:Gaussian-variance-identity}
 -\mathcal B_n'(u_n)\E(Q-u_n)^2
 =\frac{\E v_n}{2N}+O(n^{-13/6}).
\end{equation}
The expansion \eqref{eq:Gaussian-root-slope} gives
$n^{2/3}\mathcal B_n'(u_n)\to-27a^2/4$, and
\eqref{eq:Gaussian-noise-bound} gives $\E v_n\to4/9$.
Since $n^2/(2N)\to1$, \eqref{eq:Gaussian-variance-identity} yields
\begin{equation*}
 n^{4/3}\E(Q-u_n)^2\longrightarrow\frac{16}{243a^2}=\sigma_a^2.
\end{equation*}
By \eqref{eq:Gaussian-mean-root-gap},
$n^{4/3}(\E Q-u_n)^2=O(n^{-1/3})$. Hence
\begin{equation}\label{eq:Gaussian-variance-limit}
 n^{4/3}\var(Q)
 =n^{4/3}\E(Q-u_n)^2-n^{4/3}(\E Q-u_n)^2
 \longrightarrow\sigma_a^2.
\end{equation}
This proves \eqref{eq:critical-mean}.

It remains to establish the Gaussian approximation bound in \eqref{eq:critical-Gaussian}.
We set $q_n=\E Q$ and put $W_n=n^{2/3}(Q-q_n)$.
Equation~\eqref{eq:Gaussian-variance-limit} gives
\begin{equation}\label{eq:Gaussian-centered-tightness}
 \E W_n^2=n^{4/3}\var(Q)\longrightarrow\sigma_a^2.
\end{equation}
In particular, $\sup_n\E W_n^2<\infty$.
Take $\pi$ to be the density of $\calN(0,\sigma_a^2)$.
For any $1$-Lipschitz $\varphi:\R\to\R$, let $f$ be the bounded solution of
\eqref{eq:stein-equation} with $b(z)=-\frac{27a^2}{4}z$.
Lemma~\ref{lem:smooth-stein} gives \eqref{eq:smooth-stein-factors}
uniformly over these $\varphi$.
We estimate the three terms in \eqref{eq:stein-master}.

\noindent\textbf{For $\mathrm I$.} Using
$b_n=\mathcal B_n(Q)+\varepsilon_n$ and $\E b_n=0$, subtract
expectations from \eqref{eq:Gaussian-root-Taylor} to obtain
\begin{equation}\label{eq:Gaussian-centered-drift}
 b_n=\mathcal B_n'(u_n)(Q-\E Q)
       +(R_n-\E R_n)+(\varepsilon_n-\E\varepsilon_n).
\end{equation}
Here $Q-\E Q=n^{-2/3}W_n$ and $b(W_n)=-\frac{27a^2}{4}W_n$.
The expansion \eqref{eq:Gaussian-root-slope} and the Cauchy--Schwarz inequality
applied to \eqref{eq:Gaussian-centered-tightness} give
\begin{equation*}
 \left|n^{2/3}\mathcal B_n'(u_n)+\frac{27a^2}{4}\right|
 \le Cn^{-1/3},\qquad
 \E|W_n|\le(\E W_n^2)^{1/2}\le C.
\end{equation*}
To estimate the centered $\varepsilon_n$ term, put
\begin{equation*}
 \widetilde f_n(x)=f\!\left(n^{1/6}x+n^{2/3}(2/3-q_n)\right)-\E f(W_n).
\end{equation*}
Then $\norm{\widetilde f_n}_{C_b^1}\le2\norm f_\infty+n^{1/6}\norm{f'}_\infty
\le Cn^{1/6}$, so \eqref{eq:quartic-tested-error} gives
\begin{equation}\label{eq:Gaussian-centered-tested-error}
 n^{4/3}\left|\E[(\varepsilon_n-\E\varepsilon_n)f(W_n)]\right| =n^{4/3}\left|\E[\varepsilon_n \widetilde f_n(\sqrt n(Q-2/3))]\right| \le Cn^{-2/3}\norm{\widetilde f_n}_{C_b^1}\le Cn^{-1/2}.
\end{equation}
Applying the triangle inequality to \eqref{eq:Gaussian-centered-drift}
and using \eqref{eq:Gaussian-mean-remainders} and
\eqref{eq:Gaussian-centered-tested-error}, we obtain
\begin{equation}\label{eq:Gaussian-term-I}
\begin{aligned}
 \mathrm I &\le\norm{f}_\infty\biggl[ \left|n^{2/3}\mathcal B_n'(u_n)+\frac{27a^2}{4}\right|\E|W_n| +n^{4/3}\E|R_n-\E R_n|\biggr]\\
 &\quad+n^{4/3}\left|\E[(\varepsilon_n-\E\varepsilon_n)f(W_n)]\right|\\
 &\le C\bigl[n^{-1/3}+2n^{4/3}\E|R_n|+n^{-1/2}\bigr] \le C(n^{-1/3}+n^{-1/2}).
\end{aligned}
\end{equation}

\noindent\textbf{For $\mathrm{II}$.} We use \eqref{eq:Gaussian-noise-bound}.
Since $n^2/(2N)=1+1/(n-1)$ and $0\le v_n\le1$,
the bound on $\norm{f'}_\infty$ in \eqref{eq:smooth-stein-factors} gives
\begin{equation}\label{eq:Gaussian-smooth-noise-error}
\begin{aligned}
 \mathrm{II} &=\norm{f'}_\infty\E\left|\frac49-v_n-\frac{v_n}{n-1}\right|\\
 &\le\norm{f'}_\infty \left\{\norm{v_n-4/9}_1+\frac{\E v_n}{n-1}\right\} \le C(n^{-1/3}+n^{-1}) \le Cn^{-1/3}.
\end{aligned}
\end{equation}

\noindent
\textbf{For $\mathrm{III}$.} Since $N=n(n-1)/2$ and $r_n=n^{2/3}$,
the bound on $\operatorname{Lip}(f')$ in \eqref{eq:smooth-stein-factors} gives
\begin{equation}\label{eq:Gaussian-term-III}
 \mathrm{III}=\frac{n^{2/3}}{(n-1)^2}\operatorname{Lip}(f')
       \le Cn^{-4/3}.
\end{equation}
Substituting \eqref{eq:Gaussian-term-I},
\eqref{eq:Gaussian-smooth-noise-error} and \eqref{eq:Gaussian-term-III}
into \eqref{eq:stein-master} yields
\begin{equation}\label{eq:Gaussian-smooth-error}
 |\E\varphi(W_n)-\pi\varphi|
 \le C(n^{-1/3}+n^{-1/2}+n^{-4/3})\le Cn^{-1/3}.
\end{equation}
Taking the supremum over all $1$-Lipschitz $\varphi$ proves
the upper bound in \eqref{eq:critical-Gaussian}.

\noindent\textbf{Sharpness of the rate.}
Put $\delta_n=n^{-1/3}$, $\kappa=27a^2/4$ and
$\chi=-27a/4$. For $D_n=Q-u_n$, Taylor's formula gives
\begin{equation*}
 R_n=\tfrac12F''(u_n)D_n^2+O(|D_n|^3+n^{-1}D_n^2),\qquad
 F''(u_n)=2\chi\delta_n+O(\delta_n^2).
\end{equation*}
Here we used $\beta\in C^2$, $F''(2/3)=0$, $F^{(3)}(2/3)=-27/2$
and \eqref{eq:endpoint-root}.
Equations~\eqref{eq:Gaussian-tightness}, \eqref{eq:Gaussian-root-moments}
and \eqref{eq:Gaussian-mean-root-gap} imply
\begin{equation}\label{eq:Gaussian-sharpness-remainder}
 n^{4/3}(R_n-\E R_n)
 =\chi\delta_n(W_n^2-\E W_n^2)+O_{L^1}(n^{-1/2}).
\end{equation}
In \eqref{eq:Gaussian-sharpness-remainder}, the $O(|D_n|^3)$ remainder
in the expansion of $R_n$ contributes $O_{L^1}(n^{-1/2})$ by
\eqref{eq:Gaussian-root-moments}. The cross term obtained from
$n^{2/3}D_n=W_n+n^{2/3}(\E Q-u_n)$ contributes the same order by
\eqref{eq:Gaussian-mean-root-gap} and the Cauchy--Schwarz inequality
applied to \eqref{eq:Gaussian-centered-tightness}.

For $V(q)=q+(1-2q)F(q)$, we have $V(2/3)=4/9$ and $V'(2/3)=-2/3$.
The estimate \eqref{eq:taylor-noise}, Taylor's formula for $V$,
\eqref{eq:Gaussian-tightness} and \eqref{eq:endpoint-root} give
\begin{equation}\label{eq:Gaussian-sharpness-noise}
 v_n=\tfrac49-\tfrac{2a}{3}\delta_n+o_{L^1}(\delta_n).
\end{equation}

The bound \eqref{eq:Gaussian-smooth-error} and \eqref{eq:Gaussian-centered-tightness}
give $W_n\Rightarrow Z\sim\calN(0,\sigma_a^2)$ and $\E W_n^2\to\sigma_a^2$.
Choose a nonzero, nonnegative even $f\in C_c^\infty(\R)$ supported in
$\{x:|x|>\sigma_a\}$, and put $\psi(x)=\tfrac49f'(x)-\kappa xf(x)$.
Set $a_n=n^{2/3}\mathcal B_n'(u_n)+\kappa$; \eqref{eq:Gaussian-root-slope}
gives $a_n=O(\delta_n)$.
Combining \eqref{eq:exchangeable-identity} and
\eqref{eq:Gaussian-centered-drift} with the expansions
\eqref{eq:Gaussian-sharpness-remainder} and \eqref{eq:Gaussian-sharpness-noise},
and estimating the centered $\varepsilon_n$ term as in
\eqref{eq:Gaussian-centered-tested-error}, we obtain
\begin{equation*}
 \E\psi(W_n) =-a_n\E[W_nf(W_n)] -\chi\delta_n\E[(W_n^2-\E W_n^2)f(W_n)] +\tfrac{2a}{3}\delta_n\E f'(W_n)+o(\delta_n).
\end{equation*}
By weak convergence, compact support and parity,
\begin{equation*}
 \delta_n^{-1}\E\psi(W_n)
 \longrightarrow-\chi\E[(Z^2-\sigma_a^2)f(Z)]\ne0.
\end{equation*}
Since $\psi$ is odd and Lipschitz, $\E\psi(Z)=0$ and
$L=\operatorname{Lip}(\psi)>0$. Testing with $\psi/L$ gives the lower bound in
\eqref{eq:critical-Gaussian}. This completes the proof.
\end{proof}

\subsection{Proof of Theorem~\ref{thm:triangle-quartic}}\label{sec:quartic-proof}

{
Set $W_n=\sqrt n(Q-2/3)$.
The parameter choice gives $\beta(2/3)=0$ and $\beta'(2/3)=\theta$, so
Taylor expansion yields
\[
 n^{3/2}\mathcal B_n(Q)
 =b(W_n)+O\!\left(n^{-1/2}(|W_n|^4+W_n^2)\right),
 \qquad b(z)=-\frac94z^3+\theta z.
\]
A uniform fourth-moment bound controls the expectation of this remainder.
Together with the tested remainder estimate
and $\norm{v_n-4/9}_1=O(n^{-1/2})$, the exchangeable-pair identity gives
\[
 \left|\E\left[\frac49f'(W_n)+b(W_n)f(W_n)\right]\right|
 \le C_{\theta}n^{-1/2}
 \bigl(\norm{f}_{C_b^1}+\operatorname{Lip}(f')\bigr)
\]
for $f\in C_b^1(\R)$ with Lipschitz derivative.
This is the Stein operator for $\pi_{\theta}$.
Applying the bound to its solutions for $\varphi(z)=z$ and $z^2$ gives
$\E W_n^j=\E Z_{\theta}^j+O(n^{-1/2})$, $j=1,2$, and hence the mean
and variance asymptotics.
For $1$-Lipschitz $\varphi$, uniform bounds for the Stein solution
and its derivatives give the $O(n^{-1/2})$ Wasserstein bound.
\par}

\begin{proof}
Under the assumptions of Theorem~\ref{thm:triangle-quartic}, take
$r_n=\sqrt n$ and $q_n=2/3$, so that $W_n=\sqrt n(Q-2/3)$.
In Proposition~\ref{prop:conditional-moments}, take
\begin{equation*}
 \lambda=\frac{1539}{128}+\frac{27}{8}\theta,\qquad
 \gamma=-\frac94-\frac32\theta.
\end{equation*}
Then \eqref{eq:proof-parameter-family} agrees with
\eqref{eq:cancellation-line}. Recall
\begin{equation}\label{eq:beta-correction}
 \beta(q)=g(q)+F(q)(1-F(q))(\gamma+\lambda q^2),\qquad
 \mathcal B_n(q)=F(q)-q+\beta(q)/n.
\end{equation}
Here $F$ and $g$ are evaluated at the fixed pair $(\alpha_c,h_c)$.
The parameter perturbations enter through $\beta$
in \eqref{eq:endpoint-closure}.
Since $F(2/3)=2/3$ and $F'(2/3)=1$, \eqref{eq:g-values} gives
\begin{equation}\label{eq:beta-values}
 \beta(2/3)=-\frac{11}{16}+\frac29\left(\gamma+\frac49\lambda\right)=\tau,\qquad \beta'(2/3)=-\frac{81}{32}-\frac\gamma3+\frac{4\lambda}{27}.
\end{equation}
Substituting the values of $\lambda$ and $\gamma$ gives
\begin{equation*}
 \beta(2/3)=0,\qquad \beta'(2/3)=\theta.
\end{equation*}

We use the following fourth-moment bound, which is proved in Section~\ref{sec:quartic-moment-proof}.
\begin{lemma}\label{lem:quartic-moment}
Under the assumptions of Theorem~\ref{thm:triangle-quartic},
\begin{equation}\label{eq:quartic-fourth-moment}
      \sup_n\E|\sqrt n(Q-2/3)|^4<\infty.
\end{equation}
\end{lemma}

Since $W_n=\sqrt n(Q-2/3)$, H\"older's inequality and
the fourth-moment bound \eqref{eq:quartic-fourth-moment} give
\begin{equation*}
 \E|Q-2/3|=n^{-1/2}\E|W_n|
 \le n^{-1/2}(\E|W_n|^4)^{1/4}\le Cn^{-1/2}.
\end{equation*}
The function $q\mapsto q+(1-2q)F(q)$ has a bounded derivative on $[0,1]$
and equals $4/9$ at $q=2/3$.
Thus \eqref{eq:taylor-noise}, the mean value theorem
and $\E|Q-2/3|\le Cn^{-1/2}$ give
\begin{equation}\label{eq:quartic-noise-bound}
\begin{aligned}
 \norm{v_n-4/9}_1
 &\le\norm{v_n-\{Q+(1-2Q)F(Q)\}}_1
       +\E|Q+(1-2Q)F(Q)-4/9|\\
 &\le Cn^{-1/2}+C\E|Q-2/3|\le Cn^{-1/2}.
\end{aligned}
\end{equation}

By \eqref{eq:endpoint-closure} and \eqref{eq:beta-correction},
\begin{equation*}
 n^{3/2}b_n=n^{3/2}\mathcal B_n(Q)
       +n^{3/2}\varepsilon_n.
\end{equation*}
Here $F''(2/3)=0$ and $F'''(2/3)/6=-9/4$ by \eqref{eq:critical-F}.
Taylor's formula on $[0,1]$, together with
$\beta(2/3)=0$ and $\beta'(2/3)=\theta$, gives
\begin{equation}\label{eq:quartic-global-drift}
\begin{aligned}
 \left|n^{3/2}\{F(Q)-Q\}+\frac94W_n^3\right|
 &\le Cn^{3/2}|Q-2/3|^4=Cn^{-1/2}|W_n|^4,\\
 |\sqrt n\beta(Q)-\theta W_n|
 &\le C\sqrt n(Q-2/3)^2=Cn^{-1/2}W_n^2.
\end{aligned}
\end{equation}

To prove \eqref{eq:quartic-moments}, take $\pi=\pi_{\theta}$
and $\varphi(z)=z^j$, with $j=1,2$.
Let $f$ be the solution given by Lemma~\ref{lem:smooth-stein}.
We use the bound \eqref{eq:quartic-polynomial-stein-factors}.
Taking expectations in \eqref{eq:stein-equation} and adding and subtracting
the expectation in \eqref{eq:exchangeable-identity}, we obtain
\begin{equation}\label{eq:quartic-moment-identity}
\begin{aligned}
 \E W_n^j-\E Z_{\theta}^j &=\E\left[\frac49f'(W_n)+b(W_n)f(W_n)\right]\\
 &=\E[\{b(W_n)-n^{3/2}b_n\}f(W_n)] +\E\left[\left\{\frac49-\frac{n^2}{2N}v_n\right\}f'(W_n)\right]\\
 &\quad+\E\left[n^{3/2}b_nf(W_n)+\frac{n^2}{2N}v_nf'(W_n)\right].
\end{aligned}
\end{equation}

For the drift term, the triangle inequality,
\eqref{eq:quartic-global-drift} and \eqref{eq:quartic-tested-error}
give, for every $f\in C_b^1(\R)$,
\begin{equation}\label{eq:quartic-drift-bound}
\begin{aligned}
 &\left|\E[\{b(W_n)-n^{3/2}b_n\}f(W_n)]\right|\\
 &\quad\le\norm{f}_\infty \E\left|b(W_n)-n^{3/2}\mathcal B_n(Q)\right| +n^{3/2}|\E[\varepsilon_n f(W_n)]|\\
 &\quad\le Cn^{-1/2}\bigl\{\norm{f}_\infty \E(|W_n|^4+W_n^2)+\norm{f}_{C_b^1}\bigr\} \le Cn^{-1/2}\norm{f}_{C_b^1}.
\end{aligned}
\end{equation}
The last step uses H\"older's inequality and the fourth-moment bound
\eqref{eq:quartic-fourth-moment}.
For the second expectation in \eqref{eq:quartic-moment-identity},
use the triangle inequality, $n^2/(2N)=1+1/(n-1)$ and
\eqref{eq:quartic-noise-bound}.
The final expectation in \eqref{eq:quartic-moment-identity}
is bounded by \eqref{eq:exchangeable-identity} with $r_n=\sqrt n$.
Together with \eqref{eq:quartic-polynomial-stein-factors} and $0\le v_n\le1$,
these bounds yield
\begin{equation}\label{eq:quartic-moment-error}
\begin{aligned}
 |\E W_n^j-\E Z_{\theta}^j| &\le Cn^{-1/2}\norm{f}_{C_b^1} +\norm{f'}_\infty \left\{\norm{v_n-4/9}_1+\frac{\E v_n}{n-1}\right\} +\frac{\sqrt n}{(n-1)^2}\operatorname{Lip}(f')\\
 &\le C(n^{-1/2}+n^{-1}+n^{-3/2}) \le Cn^{-1/2},\qquad j=1,2.
\end{aligned}
\end{equation}

Taking $j=1$ in \eqref{eq:quartic-moment-error} gives
\begin{equation*}
 \sqrt n(\E Q-2/3)=\E W_n\longrightarrow\E Z_{\theta}=0.
\end{equation*}
This proves the mean expansion in \eqref{eq:quartic-moments}.
Taking $j=2$ in \eqref{eq:quartic-moment-error} and using the first-moment limit,
we obtain
\begin{equation*}
 n\var(Q)-\var(Z_{\theta}) =\E W_n^2-\E Z_{\theta}^2 -(\E W_n)^2 \longrightarrow0.
\end{equation*}
This proves \eqref{eq:quartic-moments}.

It remains to establish the approximation bound in \eqref{eq:quartic-limit}.
For any $1$-Lipschitz $\varphi:\R\to\R$, let $f$ be the bounded solution of
\eqref{eq:stein-equation} with $\pi=\pi_{\theta}$.
The Stein factor estimate \eqref{eq:smooth-stein-factors} gives
$\norm f_\infty+\norm{f'}_\infty+\operatorname{Lip}(f')\le C_{\theta}$.
We estimate the three terms in \eqref{eq:stein-master}, with constants
independent of $\varphi$ and $n$.

\noindent
\textbf{For $\mathrm I$.} Combining \eqref{eq:quartic-drift-bound}
with \eqref{eq:smooth-stein-factors}, we obtain
\begin{equation}\label{eq:quartic-smooth-drift-error}
 \mathrm I
 \le Cn^{-1/2}\norm{f}_{C_b^1}
 \le Cn^{-1/2}.
\end{equation}

\noindent
\textbf{For $\mathrm{II}$.} The calculation in
\eqref{eq:Gaussian-smooth-noise-error}, with \eqref{eq:quartic-noise-bound}
and \eqref{eq:smooth-stein-factors}, gives
\begin{equation}\label{eq:quartic-smooth-noise-error}
 \mathrm{II}\le\norm{f'}_\infty \left\{\norm{v_n-4/9}_1+\frac{\E v_n}{n-1}\right\} \le Cn^{-1/2}.
\end{equation}

\noindent
\textbf{For $\mathrm{III}$.} Since $N=n(n-1)/2$ and $r_n=\sqrt n$,
the bound on $\operatorname{Lip}(f')$ in \eqref{eq:smooth-stein-factors} gives
\begin{equation}\label{eq:quartic-term-III}
 \mathrm{III}=\frac{\sqrt n}{(n-1)^2}\operatorname{Lip}(f')
       \le Cn^{-3/2}.
\end{equation}
Substituting \eqref{eq:quartic-smooth-drift-error},
\eqref{eq:quartic-smooth-noise-error} and \eqref{eq:quartic-term-III}
into \eqref{eq:stein-master} yields
\begin{equation}\label{eq:quartic-smooth-error}
 |\E\varphi(W_n)-\pi\varphi|
       \le C(n^{-1/2}+n^{-3/2})\le Cn^{-1/2}.
\end{equation}
Taking the supremum over all $1$-Lipschitz $\varphi$ proves
\eqref{eq:quartic-limit}.
This completes the proof.
\end{proof}

\subsection{Proof of Theorem~\ref{thm:subgraph-limits}}\label{sec:subgraph-proof}
{
We first approximate each subgraph density by a power of the edge density,
with $\norm{t_H-Q^{e_H}}_2\le C_Hn^{-1}$.
Centering the Taylor expansion of $Q^{e_H}$ about $\E Q$ then yields
\[
\begin{aligned}
 \norm{t_H-\E t_H-e_H(2/3)^{e_H-1}(Q-\E Q)}_2\le C_H\bigl\{n^{-1}+\norm{Q-\E Q}_4^2
              +|\E Q-2/3|\norm{Q-\E Q}_2\bigr\}.
\end{aligned}
\]
The moment estimates from the preceding two proofs make this bound
$O(n^{-1})$ in both regimes. Thus, writing $W_n=r_n(Q-\E Q)$, we obtain
\[
 \norm{r_n(t_H-\E t_H)-e_H(2/3)^{e_H-1}W_n}_2
 \le
 \begin{cases}
  C_Hn^{-1/3},&r_n=n^{2/3}\text{ in part \textup{(i)}},\\
  C_Hn^{-1/2},&r_n=\sqrt n\text{ in part \textup{(ii)}}.
 \end{cases}
\]
The distributional and second-moment limits of $W_n$ are supplied by
Theorems~\ref{thm:triangle-gaussian} and~\ref{thm:triangle-quartic}.
The Cauchy--Schwarz inequality transfers the second-moment limits of $W_n$
to the covariance limits, while Slutsky's theorem gives joint convergence
for any fixed finite collection of graphs. Each limiting coordinate is
therefore the displayed multiple of the same Gaussian or quartic random
variable.
\par}

\begin{proof}
Throughout this proof, with the model parameters fixed, $C_H$ denotes
a positive constant depending only on $H$, but not on $n$.
Its value may change from line to line.
\noindent\textup{(i)} Under the assumptions of Theorem~\ref{thm:triangle-gaussian}, put
$W_n=n^{2/3}(Q-\E Q)$.
We first prove the covariance limits in \eqref{eq:joint-covariance}.
We then obtain the joint convergence in \eqref{eq:joint}.

\begin{proposition}\label{thm:collapse}
Consider the law \eqref{eq:triangle-model}, with $(\alpha_n,h_n)$ in a fixed
compact subset of $(-4,4)\times\R$.
For any fixed finite simple graph $H$ and any fixed $s\ge1$, we have
\begin{equation}\label{eq:collapse}
 \norm{t_H-Q^{e_H}}_s\le C_{H,s}n^{-1}.
\end{equation}
The constant depends only on $H$, $s$ and the compact set.
\end{proposition}
The proof is given in Section~\ref{sec:motif-proof}.
Here $(\alpha_n,h_n)\to(\alpha_c,h_c)$ and $\alpha_c<4$, so
the parameters lie in a fixed compact subset of $(-4,4)\times\R$
for all large $n$, and \eqref{eq:collapse} applies.
If $e_H=0$, then $t_H=1$, and both conclusions are immediate.
For $e_H\ge1$, Taylor's formula gives
\begin{equation}\label{eq:subgraph-power-taylor}
\begin{aligned}
 \left|Q^{e_H}-(\E Q)^{e_H}
       -e_H(\E Q)^{e_H-1}(Q-\E Q)\right|
 &\le C_H(Q-\E Q)^2,\\
 |e_H(\E Q)^{e_H-1}-e_H(2/3)^{e_H-1}|
 &\le C_H|\E Q-2/3|.
\end{aligned}
\end{equation}
Centering does not increase the $L^2$ norm. Thus the Taylor remainder satisfies
\begin{equation}\label{eq:subgraph-centered-remainder}
\begin{aligned}
 \norm{Q^{e_H}-\E Q^{e_H}-e_H(\E Q)^{e_H-1}(Q-\E Q)}_2\le C_H\norm{Q-\E Q}_4^2.
\end{aligned}
\end{equation}
Combining the centered-remainder estimate in \eqref{eq:subgraph-centered-remainder} and the coefficient bound in \eqref{eq:subgraph-power-taylor} with \eqref{eq:collapse}, we obtain
\begin{equation}\label{eq:subgraph-linearization}
\begin{aligned}
 &\norm{t_H-\E t_H-e_H(2/3)^{e_H-1}(Q-\E Q)}_2\\
 &\quad\le C_H\left\{n^{-1}+\norm{Q-\E Q}_4^2
                +|\E Q-2/3|\norm{Q-\E Q}_2\right\}.
\end{aligned}
\end{equation}

Since $|\E Q-u_n|\le\norm{Q-u_n}_4$,
\eqref{eq:Gaussian-root-moments} gives
\begin{equation*}
 \norm{Q-\E Q}_4^2
 \le\bigl(\norm{Q-u_n}_4+|\E Q-u_n|\bigr)^2
 \le4\norm{Q-u_n}_4^2\le Cn^{-7/6}.
\end{equation*}
The mean and variance asymptotics in \eqref{eq:critical-mean} also give
\begin{equation*}
 |\E Q-2/3|\le Cn^{-1/3},\qquad
 \norm{Q-\E Q}_2\le Cn^{-2/3},\qquad
 \E W_n^2\longrightarrow\sigma_a^2.
\end{equation*}
Multiplying \eqref{eq:subgraph-linearization} by $n^{2/3}$ yields
\begin{equation}\label{eq:subgraph-Gaussian-approximation}
 \norm{n^{2/3}(t_H-\E t_H)-e_H(2/3)^{e_H-1}W_n}_2 \le C_H(n^{-1/3}+n^{-1/2}) \le C_Hn^{-1/3}.
\end{equation}
In particular,
\begin{equation*}
 \norm{n^{2/3}(t_H-\E t_H)}_2
 \le e_H(2/3)^{e_H-1}\norm{W_n}_2+C_Hn^{-1/3}\le C_H.
\end{equation*}

The Cauchy--Schwarz inequality and \eqref{eq:subgraph-Gaussian-approximation} now give
\begin{equation}\label{eq:subgraph-covariance-error}
\begin{aligned}
 &\left|n^{4/3}\cov(t_{H_i},t_{H_j})-e_{H_i}e_{H_j}(2/3)^{e_{H_i}+e_{H_j}-2}\E W_n^2\right|\\
 &\quad\le\norm{n^{2/3}(t_{H_i}-\E t_{H_i})-e_{H_i}(2/3)^{e_{H_i}-1}W_n}_2 \norm{n^{2/3}(t_{H_j}-\E t_{H_j})}_2\\
 &\qquad+e_{H_i}(2/3)^{e_{H_i}-1}\norm{W_n}_2 \norm{n^{2/3}(t_{H_j}-\E t_{H_j})-e_{H_j}(2/3)^{e_{H_j}-1}W_n}_2 \le C_{H_i,H_j}n^{-1/3}.
\end{aligned}
\end{equation}
Since $\E W_n^2\to\sigma_a^2$, this proves \eqref{eq:joint-covariance}.

By \eqref{eq:critical-Gaussian}, $W_n\Rightarrow Z$.
For the graphs $H_1,\ldots,H_m$, \eqref{eq:subgraph-Gaussian-approximation} and Slutsky's theorem give \eqref{eq:joint}.
This proves part \textup{(i)}.

\noindent\textup{(ii)} Under the assumptions of Theorem~\ref{thm:triangle-quartic}, put
$W_n=\sqrt n(Q-\E Q)$.
Since $(\alpha_n,h_n)\to(\alpha_c,h_c)$ and $\alpha_c<4$,
these parameter pairs lie in a fixed compact subset of $(-4,4)\times\R$
for all large $n$.
Proposition~\ref{thm:collapse} therefore applies,
so we may use \eqref{eq:subgraph-linearization}.
Lemma~\ref{lem:quartic-moment} and \eqref{eq:quartic-moments} give
\begin{equation*}
 \norm{Q-\E Q}_4^2\le4\norm{Q-2/3}_4^2\le Cn^{-1},\qquad \norm{Q-\E Q}_2\le Cn^{-1/2},\qquad |\E Q-2/3|\le Cn^{-1/2}.
\end{equation*}
Thus \eqref{eq:subgraph-linearization} gives
\begin{equation}\label{eq:subgraph-quartic-approximation}
 \norm{\sqrt n(t_H-\E t_H)-e_H(2/3)^{e_H-1}W_n}_2\le C_Hn^{-1/2}.
\end{equation}
By \eqref{eq:quartic-moments},
$\E W_n^2=n\var(Q)\to\var(Z_{\theta})$.
Using \eqref{eq:subgraph-quartic-approximation} and repeating the argument
leading to \eqref{eq:subgraph-covariance-error}, we obtain
\begin{equation*}
 \left|n\cov(t_{H_i},t_{H_j})-e_{H_i}e_{H_j}(2/3)^{e_{H_i}+e_{H_j}-2}\E W_n^2\right|
 \le C_{H_i,H_j}n^{-1/2}.
\end{equation*}
This proves \eqref{eq:joint-quartic-covariance}.

Equations~\eqref{eq:quartic-moments} and~\eqref{eq:quartic-limit} give
\begin{equation*}
 W_n=\sqrt n(Q-2/3)-\sqrt n(\E Q-2/3)
 \Rightarrow Z_{\theta}.
\end{equation*}
Applying Slutsky's theorem to this convergence and
\eqref{eq:subgraph-quartic-approximation} gives
\eqref{eq:joint-quartic} for $H_1,\ldots,H_m$.
This completes the proof.
\end{proof}

\section{Main technical lemmas}\label{sec:conditional-rows}
In this section, we present three main technical lemmas used in the proofs.
The three lemmas in this section provide the conditional row estimates
used in the proofs. The first lemma bounds covariances and centered moments
in terms of coordinate oscillations. The second  lemma compares expectations under
the conditional graph law and a product of row laws, and controls degree
and common-neighbor differences. The third lemma gives sharper covariance bounds
for observables depending on disjoint groups of rows. Their proofs are given
in Section~\ref{sec:conditional-row-proofs}.

\begin{lemma}\label{lem:comparison}
Let $m\ge1$ be an integer and let $p$ be a strictly positive
probability mass function on $\mathcal X=\{0,1\}^m$.
Let $\mathcal C$ be symmetric and nonnegative, with entries that bound the changes in
single-coordinate conditional probabilities. Assume $\norm{\mathcal C}_{\rm op}\le c<1$.
For functions $F,G:\mathcal X\to\R$, define
 the coordinate oscillations
$\delta_iF=\sup_{\mathbf x}|F(\mathbf x^{i,1})-F(\mathbf x^{i,0})|$
and the vector $\boldsymbol\delta F=(\delta_iF)_i$, where $\mathbf x^{i,b}$
is obtained from $\mathbf x$ by setting coordinate $i$ to $b$.
Put $\mathcal D=(I-\mathcal C)^{-1}$.
Then, for every $t\in\R$,
\begin{align}
 |\cov_p(F,G)|&\le\tfrac14(\boldsymbol\delta F)^\top\mathcal D\boldsymbol\delta G,\label{eq:comparison-cov}\\
 \log\E_p \ee^{t(F-\E_p F)}
 &\le\tfrac{t^2}{8}(\boldsymbol\delta F)^\top\mathcal D\boldsymbol\delta F
 \le\frac{t^2}{8(1-c)}\norm{\boldsymbol\delta F}_{\ell^2}^2.\label{eq:comparison-mgf}
\end{align}
Consequently
\begin{equation}\label{eq:comparison-moment}
 \norm{F-\E_p F}_{L^r(p)}\le C_{r,c}\norm{\boldsymbol\delta F}_{\ell^2}
\end{equation}
for each fixed $r\ge1$.
\end{lemma}
The proof is given in Section~\ref{sec:comparison-proof}.

Fix an integer $k\ge1$, let $n>k$, and take a vertex set
$\mathcal I\subset[n]$ of size $k$. Put $L=n-k$, and let
$\calF_{\mathcal I}$ be the sigma-field generated by all edges within
$\mathcal I$ and within $\mathcal I^c$. Write
\[
 B=(X_{uv})_{u,v\in\mathcal I^c},\qquad
 e_{ab}=X_{ab}\quad(a,b\in\mathcal I),\qquad
 \mathbf z_a=(X_{au})_{u\in\mathcal I^c}\quad(a\in\mathcal I).
\]
Here $B$ and $(e_{ab})$ are fixed. Each $\mathbf z_a$ collects the free
entries in row $a$ of the adjacency matrix and is represented as a column
vector. We refer to these vectors as rows. The tuple
$\mathbf z=(\mathbf z_a)_{a\in\mathcal I}$ denotes the full free-edge
configuration. Write $\mathbf z_{-(a,u)}$ for all free coordinates except
$z_{au}=X_{au}$, and $\mathbf z_{a,-u}$ for $\mathbf z_a$ with coordinate
$u$ removed. At parameters $(\alpha,h)$, the conditional law of $\mathbf z$
in the edge--triangle model is
\begin{equation}\label{eq:triangle-row}
\begin{aligned}
 p_{B,1}(\mathbf z)
 &\propto\exp\left\{\sum_{a\in\mathcal I}
       \left(h\mathbf1^\top\mathbf z_a+\tfrac12\mathbf z_a^\top J_B\mathbf z_a\right)
       +J_{\rm cross}\right\},\\
 J_B&=\frac\alpha nB,\qquad
 J_{\rm cross}=\frac\alpha n
       \sum_{\substack{a<b\\a,b\in\mathcal I}}e_{ab}\mathbf z_a^\top\mathbf z_b.
\end{aligned}
\end{equation}
The dependence of $p_{B,1}$ on the fixed internal edges is suppressed in
the notation. Removing $J_{\rm cross}$ gives the product law
\begin{equation}\label{eq:one-row-law}
 p_B(\mathbf z_i)\propto\exp\{h\mathbf1^\top\mathbf z_i+\mathbf z_i^\top J_B\mathbf z_i/2\},\qquad
 p_{B,0}=p_B^{\otimes k}.
\end{equation}
Write $\E_B,\var_B,\cov_B$ for expectation, variance and covariance under
$p_B$, and also under the joint law of any independent row copies appearing
in their arguments. Write $\E_0$ for expectation under $p_{B,0}$.
For $0\le t\le1$, define
\begin{equation}\label{eq:interpolated-row-law}
 p_{B,t}=\frac{\ee^{tJ_{\rm cross}}}{\E_0\ee^{tJ_{\rm cross}}}\,p_{B,0},
\end{equation}
where $\E_t,\var_t,\cov_t$ refer to $p_{B,t}$.
Thus $p_{B,0}$ has independent rows, while $p_{B,1}$ is the conditional
graph law. We use the degree and common-neighbor counts
\[
 d_i=\sum_{j\ne i}X_{ij},\qquad
 M_{\{i,j\}}=\sum_{u\notin\{i,j\}}X_{iu}X_{ju}.
\]

\begin{lemma}\label{lem:row-interpolation}
Let $(\alpha,h)$ range over a fixed compact subset of
$(-4,4)\times\R$, and use the conditional row laws above.
The following estimates hold uniformly in $n>k$, the set $\mathcal I$,
the exterior graph $B$, and the fixed internal edges.
\begin{enumerate}
\renewcommand{\labelenumi}{\textup{(\roman{enumi})}}
\item
A single symmetric nonnegative matrix $\mathcal C$ bounds the changes in
single-coordinate conditional probabilities under every $p_{B,t}$,
$0\le t\le1$, and satisfies
\begin{equation}\label{eq:row-influence}
 \norm{\mathcal C}_{\rm op}
 \le\max_{(a,u)}\sum_{(b,v)}\mathcal C_{(a,u),(b,v)}
 \le\frac{|\alpha|(n-2)}{4n}\le\frac{|\alpha|}{4}<1.
\end{equation}
In particular, for every function $U$ of the free edges and fixed $r\ge1$,
\begin{equation}\label{eq:row-uniform-concentration}
 \sup_{0\le t\le1}\norm{U-\E_tU}_{L^r(p_{B,t})}
 \le C_r\norm{\boldsymbol\delta U}_{\ell^2}.
\end{equation}
\item
For every function $U$ of the free edges that does not depend on $t$,
\begin{equation}\label{eq:row-interpolation}
\begin{aligned}
 \frac{\dd}{\dd t}\E_tU&=\cov_t(U,J_{\rm cross}),\\
 \E_tU-\E_0U&=\int_0^t\cov_s(U,J_{\rm cross})\dd s,
 \qquad0\le t\le1,
\end{aligned}
\end{equation}
and
\begin{equation}\label{eq:row-mean-comparison}
 \sup_{0\le t\le1}|\E_tU-\E_0U|
 \le C_kn^{-1/2}\norm{\boldsymbol\delta U}_{\ell^2}.
\end{equation}
For each fixed $r\ge1$, the cross-row interaction satisfies
\begin{equation}\label{eq:cross-moments}
 \norm{\boldsymbol\delta J_{\rm cross}}_{\ell^2}^2\le C_k/n,
 \qquad
 \sup_{0\le t\le1}
 \norm{J_{\rm cross}-\E_tJ_{\rm cross}}_{L^r(p_{B,t})}
 \le C_{k,r}n^{-1/2}.
\end{equation}
\item
Let $U=d_i-d_j$ with $i,j\in\mathcal I$, or let $U=M_e-M_f$
with all endpoints of $e,f$ in $\mathcal I$. Then
\begin{equation}\label{eq:interpolated-contrast-means}
 \sup_{0\le t\le1}|\E_tU|\le C_k,
\end{equation}
and, for each fixed $r\ge1$,
\begin{equation}\label{eq:interpolated-contrast-moments}
 \sup_{0\le t\le1}\norm U_{L^r(p_{B,t})}\le C_{k,r}\sqrt n.
\end{equation}
\end{enumerate}
The constants may depend on the compact parameter set and on the indicated
indices, but not on the fixed edges or on $U$.
\end{lemma}
The proof is given in Section~\ref{sec:row-interpolation-proof}.

For observables depending on disjoint groups of rows, independence at
$t=0$ gives a sharper covariance estimate.
\begin{lemma}\label{lem:separated}
Under the assumptions of Lemma~\ref{lem:row-interpolation}, let
$\mathcal I_F,\mathcal I_G\subset\mathcal I$ be disjoint, and let $F$ and
$G$ depend only on $(\mathbf z_a)_{a\in\mathcal I_F}$ and
$(\mathbf z_a)_{a\in\mathcal I_G}$, respectively.
Suppose that, for a constant $C_0$ independent of $n$ and the fixed edges,
\[
 \sup_{0\le t\le1}
 \max\left\{\norm{F-\E_tF}_{L^4(p_{B,t})},
             \norm{G-\E_tG}_{L^4(p_{B,t})}\right\}\le C_0.
\]
Then
\begin{align}
 |\E_1F-\E_0F|&\le Cn^{-1/2},\label{eq:separated-mean}\\
 |\cov_1(F,G)|&\le Cn^{-1}.\label{eq:separated-covariance}
\end{align}
The constant $C$ may depend on $k$, $C_0$ and the compact parameter set,
but is uniform in $n$, $B$ and the internal edges.
\end{lemma}
The proof is given in Section~\ref{sec:separated-proof}.

\section{Proofs of lemmas in Section~\ref{sec:conditional-rows}}\label{sec:conditional-row-proofs}
We prove the estimates of Section~\ref{sec:conditional-rows} in their order of
dependence: the comparison bound, the interpolation estimates, and the
separated-row bounds.

\subsection{Proof of Lemma~\ref{lem:comparison}}\label{sec:comparison-proof}

{
For the exponential tilt $p^{(t)}\propto\ee^{tG}p$, couple the two
heat-bath chains using the same coordinate and uniform random variable.
The tilt changes each conditional probability by at most
$|t|\delta_iG/4$. Under a stationary coupling $\nu$, the disagreement probabilities
$r_i=\PP_\nu(X_i\ne Y_i)$ therefore satisfy
\[
\begin{aligned}
 \mathbf r&\le\mathcal C\mathbf r+\frac{|t|}{4}\boldsymbol\delta G,
 \quad\Longrightarrow\quad
 \mathbf r\le\frac{|t|}{4}\mathcal D\boldsymbol\delta G.
\end{aligned}
\]
Then the coordinate oscillations convert disagreements into an expectation
difference, $|\E_{p^{(t)}}F-\E_pF|
 \le\frac{|t|}{4}(\boldsymbol\delta F)^\top\mathcal D\boldsymbol\delta G.$
Differentiating at $t=0$ yields the covariance bound.
With $G=F$, the derivative of the centered log moment-generating
function is this same expectation difference. Integrating from zero
to $t$ gives
\[
 \log\E_p\ee^{t(F-\E_pF)}
 \le\frac{t^2}{8}(\boldsymbol\delta F)^\top\mathcal D\boldsymbol\delta F
 \le\frac{t^2}{8(1-c)}\norm{\boldsymbol\delta F}_{\ell^2}^2.
\]
Optimizing the exponential Markov bound gives Gaussian tails;
integrating these tails proves
$\norm{F-\E_pF}_{L^r(p)}\le C_{r,c}\norm{\boldsymbol\delta F}_{\ell^2}$
for each fixed $r\ge1$.
\par}

\begin{proof}
For $\mathbf x\in\mathcal X$, write $\mathbf x_{-i}$ for the vector with coordinate $i$ removed.
For $\mathbf X\sim p$, put
\[
 p_i(\mathbf x_{-i})=\PP_p(X_i=1\mid\mathbf X_{-i}=\mathbf x_{-i}).
\]
We first prove \eqref{eq:comparison-cov}. For $t\in\R$, define
\begin{equation*}
 p^{(t)}(\mathbf x)=\frac{\ee^{tG(\mathbf x)}p(\mathbf x)}{\E_p\ee^{tG}}.
\end{equation*}
For $\mathbf X\sim p^{(t)}$, define
\begin{equation*}
 p_i^{(t)}(\mathbf x_{-i})
 =\PP_{p^{(t)}}(X_i=1\mid\mathbf X_{-i}=\mathbf x_{-i}).
\end{equation*}
At $t=0$, this recovers $p_i^{(0)}=p_i$.
Differentiating $\E_{p^{(t)}}F$ at $t=0$ gives
\begin{equation*}
\begin{aligned}
 \left.\frac{\dd}{\dd t}\E_{p^{(t)}}F\right|_{t=0}
 &=\left.\frac{\dd}{\dd t}
   \frac{\sum_{\mathbf x\in\mathcal X}F(\mathbf x)\ee^{tG(\mathbf x)}p(\mathbf x)}
        {\sum_{\mathbf x\in\mathcal X}\ee^{tG(\mathbf x)}p(\mathbf x)}\right|_{t=0}\\
 &=\E_p(FG)-\E_pF\,\E_pG=\cov_p(F,G).
\end{aligned}
\end{equation*}
By the definition of $p^{(t)}$,
\begin{equation*}
\frac{p_i^{(t)}(\mathbf x_{-i})}{1-p_i^{(t)}(\mathbf x_{-i})}
 =\frac{p^{(t)}(\mathbf x^{i,1})}{p^{(t)}(\mathbf x^{i,0})} =\exp\{t[G(\mathbf x^{i,1})-G(\mathbf x^{i,0})]\}
   \frac{p_i(\mathbf x_{-i})}{1-p_i(\mathbf x_{-i})}.
\end{equation*}
Taking logarithms and using $\sup\ell'=1/4$ gives
\begin{equation*}
|p_i^{(t)}(\mathbf x_{-i})-p_i(\mathbf x_{-i})|
 \le\frac{|t|}{4}|G(\mathbf x^{i,1})-G(\mathbf x^{i,0})| \le\frac{|t|}{4}\delta_iG.
\end{equation*}
Changing the other coordinates one at a time under $p$ therefore yields
\begin{equation}\label{eq:comparison-conditional-difference}
\begin{aligned}
 |p_i^{(t)}(\mathbf y_{-i})-p_i(\mathbf x_{-i})|
 &\le |p_i^{(t)}(\mathbf y_{-i})-p_i(\mathbf y_{-i})|
       +|p_i(\mathbf y_{-i})-p_i(\mathbf x_{-i})|\\
 &\le\frac{|t|}{4}\delta_iG
       +\sum_j\mathcal C_{ij}\mathbf1_{\{x_j\ne y_j\}}.
\end{aligned}
\end{equation}

We compare $p$ and $p^{(t)}$ by coupling their heat-bath updates
\citep[see also][Theorem~2.1]{RR2019}.
Given $(\mathbf X,\mathbf Y)\in\mathcal X\times\mathcal X$, draw
$I$ uniformly from $\{1,\ldots,m\}$ and $U$ uniformly from $[0,1]$,
independently of each other and of $(\mathbf X,\mathbf Y)$.
For each $i$, set
\begin{equation*}
 X_i'=\begin{cases}
 \mathbf1_{\{U\le p_i(\mathbf X_{-i})\}},&I=i,\\
 X_i,&I\ne i,
 \end{cases}
 \qquad
 Y_i'=\begin{cases}
 \mathbf1_{\{U\le p_i^{(t)}(\mathbf Y_{-i})\}},&I=i,\\
 Y_i,&I\ne i.
 \end{cases}
\end{equation*}
For $\mathbf X\sim p$ and $\mathbf Y\sim p^{(t)}$, we can verify that $\mathbf X'\sim p$ and $\mathbf Y'\sim p^{(t)}$.
Let $K_t$ be the transition matrix of this update on $\mathcal X\times\mathcal X$.
Since the state space is finite, $K_t$ has a stationary distribution $\nu$.
Take $(\mathbf X,\mathbf Y)\sim\nu$, whose marginals are $p$ and $p^{(t)}$.
Then $(\mathbf X',\mathbf Y')$ has the same joint distribution.
Put $r_i=\PP_\nu(X_i\ne Y_i)$ and
$\mathbf r=(r_1,\ldots,r_m)^\top$.
On $I=i$, the common uniform variable gives
\begin{equation*}
 \PP(X_i'\ne Y_i'\mid\mathbf X,\mathbf Y,I=i)
 =|p_i(\mathbf X_{-i})-p_i^{(t)}(\mathbf Y_{-i})|.
\end{equation*}
On $I\ne i$, both $i$th coordinates remain unchanged.
Conditioning on these two events, using stationarity and
\eqref{eq:comparison-conditional-difference}, gives
\begin{equation*}
\begin{aligned}
 r_i=\PP_\nu(X_i'\ne Y_i')
 &=\E_\nu[\mathbf1_{\{I\ne i\}}\mathbf1_{\{X_i\ne Y_i\}}]\\
 &\quad+\PP(I=i)\,
   \E_\nu\!\left[\PP(X_i'\ne Y_i'\mid\mathbf X,\mathbf Y,I=i)\right]\\
 &=\left(1-\frac1m\right)r_i
   +\frac1m\E_\nu|p_i(\mathbf X_{-i})-p_i^{(t)}(\mathbf Y_{-i})|\\
 &\le\left(1-\frac1m\right)r_i
   +\frac1m\left\{\sum_j\mathcal C_{ij}r_j+\frac{|t|}{4}\delta_iG\right\}.
\end{aligned}
\end{equation*}
Thus with $\mathcal D=(I-\mathcal C)^{-1}$,
\begin{equation*}
 \mathbf r\le
   \frac{|t|}{4}\mathcal D\boldsymbol\delta G.
\end{equation*}
For any two configurations,
$|F(\mathbf x)-F(\mathbf y)|\le\sum_i\delta_iF\mathbf1_{\{x_i\ne y_i\}}$.
Taking expectation under the stationary coupling gives
\begin{equation*}
 |\E_{p^{(t)}}F-\E_pF|
 \le\sum_i\delta_iF r_i
 \le\frac{|t|}{4}(\boldsymbol\delta F)^\top\mathcal D\boldsymbol\delta G.
\end{equation*}
For $t\ne0$, divide by $|t|$ and let $t\to0$ to obtain
\eqref{eq:comparison-cov}.

For \eqref{eq:comparison-mgf}, apply the bound for
$|\E_{p^{(t)}}F-\E_pF|$ with $G=F$.
Differentiation gives
\begin{equation*}
\frac{\dd}{\dd s}\log\E_p\ee^{s(F-\E_pF)}
 =\E_{p^{(s)}}F-\E_pF,\qquad |\E_{p^{(s)}}F-\E_pF|
 \le\frac{|s|}{4}(\boldsymbol\delta F)^\top\mathcal D\boldsymbol\delta F.
\end{equation*}
For either sign of $t$,
\begin{equation*}
\begin{aligned}
 \log\E_p\ee^{t(F-\E_pF)}
 &=\int_0^t(\E_{p^{(s)}}F-\E_pF)\dd s \le\int_{\min\{0,t\}}^{\max\{0,t\}}
          |\E_{p^{(s)}}F-\E_pF|\dd s\\
 &\le\frac14(\boldsymbol\delta F)^\top\mathcal D\boldsymbol\delta F
         \int_{\min\{0,t\}}^{\max\{0,t\}}|s|\dd s =\frac{t^2}{8}(\boldsymbol\delta F)^\top\mathcal D\boldsymbol\delta F.
\end{aligned}
\end{equation*}
The convergent matrix series also gives
\begin{equation*}
\norm{\mathcal D}_{\rm op}
 \le\sum_{j=0}^{\infty}\norm{\mathcal C}_{\rm op}^j\le\frac1{1-c},\qquad (\boldsymbol\delta F)^\top\mathcal D\boldsymbol\delta F
 \le\norm{\mathcal D}_{\rm op}\norm{\boldsymbol\delta F}_{\ell^2}^2
 \le\frac{\norm{\boldsymbol\delta F}_{\ell^2}^2}{1-c}.
\end{equation*}
This proves \eqref{eq:comparison-mgf}.

If $\norm{\boldsymbol\delta F}_{\ell^2}=0$, changing any coordinate leaves $F$ unchanged, so $F$ is constant.
Otherwise, for $u>0$ and $t>0$, the exponential Markov inequality gives
\begin{equation*}
 \PP_p(F-\E_pF\ge u)
 \le\exp\left\{-tu+\frac{t^2\norm{\boldsymbol\delta F}_{\ell^2}^2}{8(1-c)}\right\}.
\end{equation*}
Taking $t=4(1-c)u/\norm{\boldsymbol\delta F}_{\ell^2}^2$ and applying the same argument to $-F$ yields
\begin{equation*}
 \PP_p(|F-\E_pF|\ge u)
 \le2\exp\left\{-\frac{2(1-c)u^2}{\norm{\boldsymbol\delta F}_{\ell^2}^2}\right\},
 \qquad u>0.
\end{equation*}
For a fixed $r\ge1$, integrating the tail bound and substituting
$u=\norm{\boldsymbol\delta F}_{\ell^2}\,v$ gives
\begin{equation*}
\begin{aligned}
 \E_p|F-\E_pF|^r
 &=r\int_0^\infty u^{r-1}\PP_p(|F-\E_pF|>u)\dd u\\
 &\le2r\norm{\boldsymbol\delta F}_{\ell^2}^r
       \int_0^\infty v^{r-1}\ee^{-2(1-c)v^2}\dd v =C_{r,c}^{\,r}\norm{\boldsymbol\delta F}_{\ell^2}^r,\\
 C_{r,c}
 &=\left\{2r\int_0^\infty v^{r-1}\ee^{-2(1-c)v^2}\dd v\right\}^{1/r}<\infty.
\end{aligned}
\end{equation*}
Taking the $r$th root proves the moment bound.
\end{proof}

\subsection{Proof of Lemma~\ref{lem:row-interpolation}}\label{sec:row-interpolation-proof}

{
Using the notation in \eqref{eq:triangle-row}, we interpolate from the independent-row law through
$p_{B,t}\propto\ee^{tJ_{\rm cross}}p_{B,0}$.
Then we gives a common influence matrix with
$\norm{\mathcal C}_{\rm op}\le|\alpha|/4<1$, so concentration holds
uniformly in $t$. 
Each coordinate changes $J_{\rm cross}$ by $O(n^{-1})$, then Lemma~\ref{lem:comparison} gives
\[
 \norm{\boldsymbol\delta J_{\rm cross}}_{\ell^2}^2\le C_k/n,
 \qquad
 \sup_{0\le t\le1}\norm{J_{\rm cross}-\E_tJ_{\rm cross}}_{L^r(p_{B,t})}
 \le C_{k,r}n^{-1/2}.
\]
For a fixed observable $U$, differentiating and integrating the normalized
tilt, then applying the covariance bound, gives
\[
\begin{aligned}
 \E_tU-\E_0U&=\int_0^t\cov_s(U,J_{\rm cross})\dd s,\\
 \sup_{0\le t\le1}|\E_tU-\E_0U|
 &\le C_kn^{-1/2}\norm{\boldsymbol\delta U}_{\ell^2}.
\end{aligned}
\]
For degree or common-neighbor differences $U$, we have
$|\E_0U|=O(1)$ and
$\norm{\boldsymbol\delta U}_{\ell^2}=O(\sqrt n)$.
Applying the preceding interpolation and concentration bounds then gives
$\sup_{0\le t\le1}|\E_tU|=O(1)$ and
$\sup_{0\le t\le1}\norm U_{L^r(p_{B,t})}=O(\sqrt n)$.
\par}

\begin{proof}
We first verify the conditional law \eqref{eq:triangle-row}.
The free edges contribute
\[
 \sum_{a\in\mathcal I}\sum_{u\in\mathcal I^c}X_{au}
 =\sum_{a\in\mathcal I}\mathbf1^\top\mathbf z_a
\]
to the edge count. A triangle containing a free edge has either one or two
vertices in $\mathcal I$. The two contributions are
\[
\begin{aligned}
 \sum_{a\in\mathcal I}\sum_{\substack{u<v\\u,v\in\mathcal I^c}}
 B_{uv}X_{au}X_{av}
 &=\frac12\sum_{a\in\mathcal I}\mathbf z_a^\top B\mathbf z_a,\\
 \sum_{\substack{a<b\\a,b\in\mathcal I}}\sum_{u\in\mathcal I^c}
 e_{ab}X_{au}X_{bu}
 &=\sum_{\substack{a<b\\a,b\in\mathcal I}}e_{ab}\mathbf z_a^\top\mathbf z_b.
\end{aligned}
\]
The factor $1/2$ in the first identity uses $B_{uv}=B_{vu}$ and $B_{uu}=0$.
All other edges and triangles are fixed under $\calF_{\mathcal I}$.
The Hamiltonian \eqref{eq:triangle-model}, with parameters $(\alpha,h)$,
therefore gives \eqref{eq:triangle-row}.

For \textup{(i)}, the conditional probability of $z_{au}=1$ under
$p_{B,t}$ is
\[
 \PP_{p_{B,t}}(z_{au}=1\mid\mathbf z_{-(a,u)})
 =\ell\left(h+\frac\alpha n\sum_{v\in\mathcal I^c}B_{uv}z_{av}
              +\frac{t\alpha}n\sum_{b\in\mathcal I\setminus\{a\}}e_{ab}z_{bu}\right).
\]
Since $\sup\ell'=1/4$, changing $z_{av}$ for $v\ne u$ changes this
probability by at most $|\alpha|B_{uv}/(4n)$, and changing $z_{bu}$ for
$b\ne a$ changes it by at most $t|\alpha|e_{ab}/(4n)$.
No other free coordinate appears in the conditional probability.
Thus the same symmetric nonnegative matrix
\[
 \mathcal C_{(a,u),(b,v)}=\frac{|\alpha|}{4n}
 \begin{cases}
 B_{uv},&a=b,\ u\ne v,\\
 e_{ab},&a\ne b,\ u=v,\\
 0,&\text{otherwise}
 \end{cases}
\]
is an influence bound for every $0\le t\le1$.
For each $(a,u)$, its row sum is
\[
\begin{aligned}
 \sum_{(b,v)}\mathcal C_{(a,u),(b,v)}
 &=\frac{|\alpha|}{4n}
 \left(\sum_{v\in\mathcal I^c\setminus\{u\}}B_{uv}
       +\sum_{b\in\mathcal I\setminus\{a\}}e_{ab}\right)\\
 &\le\frac{|\alpha|}{4n}\{(L-1)+(k-1)\}
 =\frac{|\alpha|(n-2)}{4n}.
\end{aligned}
\]
Symmetry bounds the operator norm by the largest row sum and proves
\eqref{eq:row-influence}. The compact parameter set gives a constant
$c<1$ with $|\alpha|/4\le c$. Applying
\eqref{eq:comparison-moment} under each $p_{B,t}$ proves
\eqref{eq:row-uniform-concentration}.

For \textup{(ii)}, the definition of $p_{B,t}$ gives
\[
 \E_tU=\frac{\E_0(U\ee^{tJ_{\rm cross}})}{\E_0\ee^{tJ_{\rm cross}}}.
\]
The state space is finite, so differentiation gives
\[
 \frac{\dd}{\dd t}\E_tU
 =\E_t(UJ_{\rm cross})-\E_tU\,\E_tJ_{\rm cross}
 =\cov_t(U,J_{\rm cross}).
\]
Integration proves \eqref{eq:row-interpolation}.
Changing $z_{au}$ from $0$ to $1$ changes $J_{\rm cross}$ by
\[
 \frac\alpha n\sum_{b\in\mathcal I\setminus\{a\}}e_{ab}z_{bu}.
\]
Consequently,
\[
 \delta_{(a,u)}J_{\rm cross}\le\frac{|\alpha|(k-1)}n,
 \qquad
 \norm{\boldsymbol\delta J_{\rm cross}}_{\ell^2}^2
 \le kL\left\{\frac{|\alpha|(k-1)}n\right\}^2\le C_k/n.
\]
The concentration estimate \eqref{eq:row-uniform-concentration} applied to
$J_{\rm cross}$ now gives
\[
 \sup_{0\le t\le1}
 \norm{J_{\rm cross}-\E_tJ_{\rm cross}}_{L^r(p_{B,t})}
 \le C_r\norm{\boldsymbol\delta J_{\rm cross}}_{\ell^2}
 \le C_{k,r}n^{-1/2},
\]
which proves \eqref{eq:cross-moments}.
For any $U$, the covariance estimate \eqref{eq:comparison-cov} and
\eqref{eq:row-influence} imply
\[
\begin{aligned}
 |\cov_t(U,J_{\rm cross})|
 &\le\frac{\norm{\boldsymbol\delta U}_{\ell^2}
              \norm{\boldsymbol\delta J_{\rm cross}}_{\ell^2}}{4-|\alpha|}\\
 &\le C_kn^{-1/2}\norm{\boldsymbol\delta U}_{\ell^2}.
\end{aligned}
\]
Integrating this bound in \eqref{eq:row-interpolation} proves
\eqref{eq:row-mean-comparison}.

For \textup{(iii)}, the degree and common-neighbor counts with endpoints
in $\mathcal I$ satisfy
\begin{equation}\label{eq:conditional-degree-overlap}
 d_i=\mathbf1^\top\mathbf z_i+\sum_{b\in\mathcal I\setminus\{i\}}e_{ib},\qquad
 M_{\{i,j\}}=\mathbf z_i^\top\mathbf z_j+
       \sum_{b\in\mathcal I\setminus\{i,j\}}e_{ib}e_{jb}.
\end{equation}
Each free coordinate changes either difference $U$ by at most two.
There are $kL$ free coordinates, and hence
\[
 \norm{\boldsymbol\delta U}_{\ell^2}\le2\sqrt{kL}\le2\sqrt{kn}.
\]
Under $p_{B,0}$, the rows are independent with the same law. In particular,
\[
\begin{aligned}
 \E_0d_i
 &=\sum_{u\in\mathcal I^c}\E_Bz_{iu}
     +\sum_{b\in\mathcal I\setminus\{i\}}e_{ib},\\
 \E_0M_{\{i,j\}}
 &=\sum_{u\in\mathcal I^c}(\E_Bz_{iu})(\E_Bz_{ju})
     +\sum_{b\in\mathcal I\setminus\{i,j\}}e_{ib}e_{jb}.
\end{aligned}
\]
The sums over $u$ are the same for every $i$ and every pair $\{i,j\}$,
respectively, and cancel in each difference defining $U$.
The internal counts lie between $0$ and $k-1$, so $|\E_0U|\le k$.
Using \eqref{eq:row-mean-comparison}, we obtain
\[
 \sup_{0\le t\le1}|\E_tU|
 \le |\E_0U|+C_kn^{-1/2}\norm{\boldsymbol\delta U}_{\ell^2}
 \le k+C_k\le C_k.
\]
This proves \eqref{eq:interpolated-contrast-means}. Finally, Minkowski's
inequality and \eqref{eq:row-uniform-concentration} give
\[
\begin{aligned}
 \sup_{0\le t\le1}\norm U_{L^r(p_{B,t})}
 &\le\sup_{0\le t\le1}\norm{U-\E_tU}_{L^r(p_{B,t})}
       +\sup_{0\le t\le1}|\E_tU|\\
 &\le C_r\norm{\boldsymbol\delta U}_{\ell^2}+C_k
 \le C_{k,r}\sqrt n,
\end{aligned}
\]
which is \eqref{eq:interpolated-contrast-moments}.
\end{proof}

\subsection{Proof of Lemma~\ref{lem:separated}}\label{sec:separated-proof}

{
The mean bound follows by integrating the identity in
Lemma~\ref{lem:row-interpolation}.
For $c(t)=\cov_t(F,G)$, independence gives $c(0)=0$, so Taylor's formula yields
\[
 |c(1)|\le |c'(0)|+\int_0^1(1-t)|c''(t)|\dd t.
\]
Independence also gives
\[
 c'(0)=\frac\alpha n
 \sum_{a\in\mathcal I_F}\sum_{b\in\mathcal I_G}e_{ab}
 \sum_{u\in\mathcal I^c}\cov_0(F,z_{au})\cov_0(G,z_{bu}).
\]
Lemma~\ref{lem:comparison} bounds the variance of every linear row statistic.
Duality and the moment assumptions
give uniform $\ell^2$ bounds on the coordinate-covariance vectors,
so Cauchy--Schwarz yields $|c'(0)|\le C/n$.
Differentiating again and applying H\"older's inequality with the assumed
fourth-moment bounds and Lemma~\ref{lem:row-interpolation}, we obtain
$\sup_{0\le t\le1}|c''(t)|\le C/n$.
Taylor's formula then gives $|\cov_1(F,G)|\le C/n$.
\par}

\begin{proof}
Let $C_0$ be a common uniform upper bound for
$\norm{F-\E_tF}_{L^4(p_{B,t})}$ and $\norm{G-\E_tG}_{L^4(p_{B,t})}$
along $0\le t\le1$, as assumed in the lemma.
Condition on $\calF_{\mathcal I}$, which fixes all edges within
$\mathcal I$ and within $\mathcal I^c$.
Only the edges joining these two sets remain random.
\par\medskip
\noindent\textit{Proof of \eqref{eq:separated-mean}}: By \eqref{eq:row-interpolation} and the Cauchy--Schwarz inequality,
\begin{equation*}
\begin{aligned}
 |\E_1F-\E_0F|
 &\le\int_0^1|\cov_t(F,J_{\rm cross})|\dd t\\
 &\le\int_0^1\norm{F-\E_tF}_{L^2(p_{B,t})}
          \norm{J_{\rm cross}-\E_tJ_{\rm cross}}_{L^2(p_{B,t})}\dd t
 \le Cn^{-1/2}.
\end{aligned}
\end{equation*}
Here $\norm{F-\E_tF}_{L^2(p_{B,t})}
\le\norm{F-\E_tF}_{L^4(p_{B,t})}\le C_0$, and
\eqref{eq:cross-moments} bounds the other factor by $Cn^{-1/2}$.

\par\medskip
\noindent\textit{Proof of \eqref{eq:separated-covariance}}: Put $c(t)=\cov_t(F,G)$. Under $p_{B,0}$, the families
$(\mathbf z_a)_{a\in\mathcal I_F}$ and
$(\mathbf z_b)_{b\in\mathcal I_G}$ are independent, so $c(0)=0$.
Taylor's formula gives
\begin{equation}\label{eq:separated-covariance-budget}
 |c(1)|\le |c'(0)|+\int_0^1(1-t)|c''(t)|\dd t.
\end{equation}
It therefore suffices to prove
\begin{equation}\label{eq:separated-derivative-targets}
 |c'(0)|\le C/n,\qquad \sup_{0\le t\le1}|c''(t)|\le C/n.
\end{equation}
Write $\var_t$ for variance under $p_{B,t}$.
For a differentiable family of functions $A_t$, differentiating its
expectation under the exponential tilt gives
\begin{equation}\label{eq:row-tilt-derivative}
 \frac{\dd}{\dd t}\E_tA_t
 =\E_t(\partial_tA_t)+\cov_t(A_t,J_{\rm cross}).
\end{equation}
Applying this identity to $(F-\E_tF)(G-\E_tG)$ and using the zero
means of both factors gives
\begin{equation*}
 c'(t)=\E_t\bigl[(F-\E_tF)(G-\E_tG)
                  (J_{\rm cross}-\E_tJ_{\rm cross})\bigr].
\end{equation*}

For the first bound in \eqref{eq:separated-derivative-targets}, expand
$J_{\rm cross}$. A pair that does not join
$\mathcal I_F$ to $\mathcal I_G$ has zero contribution to $c'(0)$ by independence.
For a pair joining the groups, label its endpoints
so that $a\in\mathcal I_F$ and $b\in\mathcal I_G$. Then
\begin{equation*}
 \E_0\bigl[(F-\E_0F)(G-\E_0G)z_{au}z_{bu}\bigr]
 =\cov_0(F,z_{au})\cov_0(G,z_{bu}).
\end{equation*}
Also, $\E_0[(F-\E_0F)(G-\E_0G)]=0$.
Hence \eqref{eq:triangle-row} yields
\begin{equation*}
 c'(0)=\frac\alpha n
 \sum_{a\in\mathcal I_F}\sum_{b\in\mathcal I_G}e_{ab}
 \sum_{u\in\mathcal I^c}\cov_0(F,z_{au})\cov_0(G,z_{bu}).
\end{equation*}
For a fixed $a\in\mathcal I$, Lemma~\ref{lem:comparison} gives
$\var_0(\mathbf w^\top\mathbf z_a)\le C\norm{\mathbf w}_{\ell^2}^2$.
Thus
\begin{equation*}
\begin{aligned}
\left\{\sum_{u\in\mathcal I^c}\cov_0(F,z_{au})^2\right\}^{1/2}
 &=\sup_{\norm{\mathbf w}_{\ell^2}\le1}|\cov_0(F,\mathbf w^\top\mathbf z_a)|\\
 &\le\norm{F-\E_0F}_{L^2(p_{B,0})}
       \sup_{\norm{\mathbf w}_{\ell^2}\le1}
             \sqrt{\var_0(\mathbf w^\top\mathbf z_a)}
 \le C.
\end{aligned}
\end{equation*}
The same bound holds for $G$. Since $|\mathcal I_F||\mathcal I_G|$
is bounded independently of $n$,
the Cauchy--Schwarz inequality in the coordinate sum gives
\begin{equation*}
\begin{aligned}
 |c'(0)|
 &\le\frac Cn\sum_{a\in\mathcal I_F}\sum_{b\in\mathcal I_G}
 \left\{\sum_u\cov_0(F,z_{au})^2\right\}^{1/2}
 \left\{\sum_u\cov_0(G,z_{bu})^2\right\}^{1/2}
 \le C/n.
\end{aligned}
\end{equation*}

For the second bound in \eqref{eq:separated-derivative-targets},
differentiating the expression for $c'(t)$ using
\eqref{eq:row-tilt-derivative} gives
\begin{equation*}
\begin{aligned}
 c''(t)
 &=\E_t\bigl[(F-\E_tF)(G-\E_tG)
                 (J_{\rm cross}-\E_tJ_{\rm cross})^2\bigr]\\
 &\quad-2\cov_t(F,J_{\rm cross})\cov_t(G,J_{\rm cross})
       -\cov_t(F,G)\var_t(J_{\rm cross}).
\end{aligned}
\end{equation*}
By H\"older's inequality, the Cauchy--Schwarz inequality, the hypothesis,
and \eqref{eq:cross-moments},
\begin{equation*}
 |c''(t)|\le 4\norm{F-\E_tF}_{L^4(p_{B,t})}
             \norm{G-\E_tG}_{L^4(p_{B,t})}
             \norm{J_{\rm cross}-\E_tJ_{\rm cross}}_{L^4(p_{B,t})}^2
 \le C/n.
\end{equation*}
Thus $\sup_{0\le t\le1}|c''(t)|\le C/n$.
Substituting \eqref{eq:separated-derivative-targets} into
\eqref{eq:separated-covariance-budget} gives
\begin{equation*}
 |\cov_1(F,G)|=|c(1)|
 \le C/n+\int_0^1(1-t)C/n\dd t\le C/n,
\end{equation*}
which proves \eqref{eq:separated-covariance}.
\end{proof}

\section{Proof of Proposition~\ref{prop:conditional-moments}}\label{sec:technical}\label{sec:moment-reduction}\label{sec:tested-proof}

{
Set $(\alpha,h)=(\alpha_n,h_n)$ and $\Delta_e=M_e-nQ^2$.
For part~\textup{(i)}, Taylor expansion and averaging over edges give
\[
\begin{aligned}
 b_n={}F_{\alpha,h}(Q)-Q
 +\frac{\alpha\ell'(h+\alpha Q^2)}{nN}\sum_e\Delta_e+\frac{\alpha^2\ell''(h+\alpha Q^2)}{2n^2N}\sum_e\Delta_e^2
 +O_{L^2}(n^{-2}).
\end{aligned}
\]
The signed cubic estimate bounds the omitted term
$\alpha^3\ell^{(3)}(h+\alpha Q^2)\sum_e\Delta_e^3/(6n^3N)$
in $L^2$ by $Cn^{-2}$.
The fourth-order remainder is bounded by $C(n^4N)^{-1}\sum_e|\Delta_e|^4$,
whose $L^2$ norm is $O(n^{-2})$ by the eighth-moment bound.
Counting identities and conditional row estimates approximate
$N^{-1}\sum_e\Delta_e$ and $(nN)^{-1}\sum_e\Delta_e^2$ by functions of $Q$.
Parameter expansion then identifies $g(Q)/n$ and the parameter correction.
The approximation errors give $\norm{\varepsilon_n}_2\le Cn^{-3/2}$.

For part~\textup{(ii)}, the Lipschitz bound for $\ell$ and local moment bounds
give the conditional second-moment estimate;
replacing $F_{\alpha,h}$ by $F$ costs $O(n^{-1})$.
For part~\textup{(iii)}, estimating the errors in the two averaged drift terms
directly against the test function retains cancellation and gives
\[
 \bigl|\E[\varepsilon_n f(\sqrt n(Q-2/3))]\bigr|
 \le Cn^{-2}\norm{f}_{C_b^1}.
\]
\par}

\begin{proof}
Take the parameters in \eqref{eq:proof-parameter-family} and write
$h=h_n$ and $\alpha=\alpha_n$.
We first derive the drift at these parameters, then expand the
resulting functions about the fixed pair $(\alpha_c,h_c)$.
The local moment estimates below hold uniformly on compact subsets
of $(-4,4)\times\R$. For this proof, fix
\begin{equation*}
 \mathcal U=[3,15/4]\times[h_c-1,h_c+1]\subset(-4,4)\times\R.
\end{equation*}
Since $\alpha_c=27/8$ and $(\alpha_n,h_n)\to(\alpha_c,h_c)$,
we have $(\alpha_n,h_n)\in\mathcal U$ for all sufficiently large $n$.

For \textup{(i)}, fix an edge $e=\{i,j\}$ and $X_{-e}$.
Recall the common-neighbor count $M_e$ from Section~\ref{sec:conditional-rows}.
Changing $X_e$ from $0$ to $1$ adds one edge and $M_e$ triangles.
Thus \eqref{eq:triangle-model} gives
\begin{equation*}
 \frac{\PP(X_e=1\mid X_{-e})}{\PP(X_e=0\mid X_{-e})}
 =\exp\left\{h+\frac\alpha nM_e\right\}.
\end{equation*}
Hence $L_e=\PP(X_e=1\mid X_{-e})$ satisfies
\begin{equation}\label{eq:fields}
 L_e=\frac{\exp\{h+\alpha M_e/n\}}{1+\exp\{h+\alpha M_e/n\}}
     =\ell\left(h+\frac\alpha n M_e\right).
\end{equation}
\begin{samepage}
By the Taylor expansion of $\ell$ at $h+\alpha Q^2$, together with
\eqref{eq:heatbath} and \eqref{eq:fields}, we obtain
\begin{equation}\label{eq:raw-drift-expansion}
 b_n=\frac1N\sum_e\ell\left(h+\frac\alpha nM_e\right)-Q =F_{\alpha,h}(Q)-Q+R_1+R_2+R_3+R_4.
\end{equation}
\end{samepage}
Here $F_{\alpha,h}$ is defined in \eqref{eq:critical-F},
\begin{equation}\label{eq:drift-Taylor-terms}
\begin{aligned}
 R_1&=\frac{\alpha\ell'(h+\alpha Q^2)}{nN} \sum_e(M_e-nQ^2),\qquad R_2=\frac{\alpha^2\ell''(h+\alpha Q^2)}{2n^2N} \sum_e(M_e-nQ^2)^2,\\
 R_3&=\frac{\alpha^3\ell^{(3)}(h+\alpha Q^2)}{6n^3N} \sum_e(M_e-nQ^2)^3,\qquad |R_4|\le\frac C{n^4N}\sum_e|M_e-nQ^2|^4.
\end{aligned}
\end{equation}

To control $R_4$, we first state a local moment estimate.
With $d_i$ as in Section~\ref{sec:conditional-rows}, define
\begin{equation}\label{eq:triangle-empirical}
 V_n=\frac1{n(n-1)}\sum_i\{d_i-(n-1)Q\}^2,\qquad K_n=\frac1{nN}\sum_e(M_e-nQ^2)^2.
\end{equation}
\begin{lemma}\label{thm:fields}
For every fixed $r\ge1$ and all vertices or possible edges involved,
\begin{equation}\label{eq:field-moments}
 \norm{d_i-d_j}_r+\norm{M_e-M_f}_r\le C_r\sqrt n,\qquad \norm{L_e-L_f}_r+\norm{L_e-\bar L}_r\le C_rn^{-1/2}.
\end{equation}
\begin{equation}\label{eq:local-moments}
 \norm{V_n}_r+\norm{K_n}_r\le C_r,\qquad
 \norm{M_e-nQ^2}_r\le C_r\sqrt n.
\end{equation}
\end{lemma}
The proof is given in Section~\ref{sec:row-contrast}.
Using \eqref{eq:drift-Taylor-terms}, Minkowski's inequality and
\eqref{eq:local-moments} at order eight, we obtain
\begin{equation*}
 \norm{R_4}_2 \le\frac C{n^4N}\sum_e\norm{M_e-nQ^2}_8^4 \le\frac C{n^4N}\sum_e Cn^2\le Cn^{-2}.
\end{equation*}

For $R_3$, we need a bound on the averaged cubic term in
\eqref{eq:drift-Taylor-terms}.
Put $\meanM=N^{-1}\sum_eM_e$.
\begin{lemma}\label{thm:triangle-cubic}
For parameters $(\alpha,h)$ in a fixed compact subset of
$(-4,4)\times\R$,
\begin{equation}\label{eq:overlap-cubic-average}
 \left\lVert\frac1N\sum_e(M_e-\meanM)^3\right\rVert_2\le Cn,
 \qquad
 \left\lVert\frac1N\sum_e(M_e-nQ^2)^3\right\rVert_2\le Cn.
\end{equation}
\end{lemma}
The proof is given in Section~\ref{sec:cubic-proof}.
Since $\alpha$ and $\ell^{(3)}$ are bounded, this gives
\begin{equation*}
 \norm{R_3}_2
 \le\frac C{n^3}
        \left\lVert\frac1N\sum_e(M_e-nQ^2)^3\right\rVert_2
 \le Cn^{-2}.
\end{equation*}

To evaluate $R_1$, we count unordered pairs of neighbors of each vertex
to obtain
$\sum_eM_e=\sum_i(d_i^2-d_i)/2$.
Together with $\sum_i d_i=n(n-1)Q$, this yields
\begin{equation}\label{0201}
 \meanM = \frac{1}{N} \sum_eM_e=\frac{\sum_i d_i^2-\sum_i d_i}{n(n-1)}
       =V_n+(n-1)Q^2-Q.
\end{equation}
Hence
\begin{equation}\label{eq:triangle-first-moment}
 \frac1N\sum_e(M_e-nQ^2)=\meanM-nQ^2=V_n-Q^2-Q,
\end{equation}
and therefore
\begin{equation*}
 R_1=\frac{\alpha\ell'(h+\alpha Q^2)}n(V_n-Q^2-Q).
\end{equation*}

For $R_2$, the definition of $K_n$ in \eqref{eq:triangle-empirical} gives
\begin{equation*}
 R_2=\frac{\alpha^2\ell''(h+\alpha Q^2)}{2n}K_n.
\end{equation*}
We next express $R_1+R_2$ in terms of functions of $Q$.
For each $q\in[0,1]$, let $\varrho_{\alpha,h}(q)$ be the unique root
$r\in(0,1)$ of
\begin{equation}\label{eq:triangle-row-root}
 r-\ell(h+\alpha qr)=0.
\end{equation}
Define
\begin{equation}\label{eq:triangle-r}
\begin{aligned}
 v_{\alpha,h}(q)&=\varrho_{\alpha,h}(q)(1-\varrho_{\alpha,h}(q)),\qquad
 A_{\alpha,h}(q)=\frac{v_{\alpha,h}(q)}{1-\alpha qv_{\alpha,h}(q)},\\
 K_{\alpha,h}(q)&=v_{\alpha,h}(q)^2
       +2\varrho_{\alpha,h}(q)^2A_{\alpha,h}(q).
\end{aligned}
\end{equation}
The left side of \eqref{eq:triangle-row-root} is negative at $r=0$
and positive at $r=1$. Moreover,
\begin{equation*}
 \frac{\partial}{\partial r}\{r-\ell(h+\alpha qr)\}
 =1-\alpha q\ell'(h+\alpha qr)
 \ge1-\frac{|\alpha|}{4}>0.
\end{equation*}
Thus the root exists and is unique.
At this root, $\ell'(h+\alpha q\varrho_{\alpha,h}(q))=v_{\alpha,h}(q)$,
so $1-\alpha qv_{\alpha,h}(q)\ge1-|\alpha|/4$.
On $\mathcal U$, this lower bound is at least $1/16$.
Since $0<v_{\alpha,h}(q)\le1/4$, we also have $0<A_{\alpha,h}(q)\le4$,
uniformly for $q\in[0,1]$ and $(\alpha,h)\in\mathcal U$.
Implicit differentiation therefore gives uniform bounds for all derivatives
of each fixed order of $\varrho_{\alpha,h}$ and the functions in
\eqref{eq:triangle-r} with respect to $(q,\alpha,h)$, for
$q\in[0,1]$ and $(\alpha,h)\in\mathcal U$.
At the critical parameters, $A_{\alpha_c,h_c}=A$ and
$K_{\alpha_c,h_c}=K$, as defined in \eqref{eq:critical-row-response}.

Define
\begin{equation}\label{eq:triangle-g}
 g_{\alpha,h}(q)=\alpha F_{\alpha,h}(q)(1-F_{\alpha,h}(q)) \biggl[A_{\alpha,h}(q)-q^2-q +\frac\alpha2(1-2F_{\alpha,h}(q))K_{\alpha,h}(q)\biggr].
\end{equation}
At $(\alpha_c,h_c)$, $g_{\alpha,h}$ is the function $g$ in
\eqref{eq:critical-g}.
Under \eqref{eq:proof-parameter-family}, Taylor's formula and the
bounded parameter derivatives give, uniformly for $q\in[0,1]$,
\begin{equation}\label{eq:parameter-expansion}
\begin{aligned}
 F_{\alpha_n,h_n}(q)
 &=F(q)+F(q)(1-F(q))
       \left\{\frac{\gamma+\lambda q^2}n\right\}
       +O(n^{-2}),\\
 g_{\alpha_n,h_n}(q)&=g(q)+O(n^{-1}).
\end{aligned}
\end{equation}
Substituting the expressions for $R_1,R_2$ and these expansions into
\eqref{eq:raw-drift-expansion}, we obtain
\begin{equation*}
 b_n={}F(Q)-Q+\frac{g(Q)}n +F(Q)(1-F(Q))\left\{\frac{\gamma+\lambda Q^2}n\right\}+\varepsilon_n,
\end{equation*}
where
\begin{equation}\label{eq:closure-remainder}
 \varepsilon_n={}\frac{\alpha\ell'(h+\alpha Q^2)}n\{V_n-A_{\alpha,h}(Q)\} +\frac{\alpha^2\ell''(h+\alpha Q^2)}{2n}\{K_n-K_{\alpha,h}(Q)\} +O_{L^2}(n^{-2}).
\end{equation}
Here $O_{L^2}(a_n)$ denotes an error with $L^2$ norm at most $Ca_n$.
The last term includes $R_3+R_4$, the $O(n^{-2})$ error in
$F_{\alpha_n,h_n}$, and the $O(n^{-1})$ error in
$g_{\alpha_n,h_n}$ divided by $n$.

It remains to bound $\norm{\varepsilon_n}_2$.
For the two empirical terms, we use the following estimates.
\begin{lemma}\label{prop:triangle-moments}
For parameters $(\alpha,h)$ in a fixed compact subset of
$(-4,4)\times\R$, the functions $A_{\alpha,h}$ and $K_{\alpha,h}$
defined in \eqref{eq:triangle-r} satisfy
\begin{equation}\label{eq:triangle-moment-replacement}
 \norm{V_n-A_{\alpha,h}(Q)}_2+\norm{K_n-K_{\alpha,h}(Q)}_2\le Cn^{-1/2}.
\end{equation}
\end{lemma}
The proof is given in Section~\ref{sec:empirical-moments-proof}.
Combining these estimates in \eqref{eq:closure-remainder}, we obtain
\begin{equation*}
\begin{aligned}
 \norm{\varepsilon_n}_2
 &\le\frac Cn\{\norm{V_n-A_{\alpha,h}(Q)}_2+\norm{K_n-K_{\alpha,h}(Q)}_2\}
       +\norm{R_3+R_4}_2+Cn^{-2}\\
 &\le Cn^{-3/2}+Cn^{-2}\le Cn^{-3/2}.
\end{aligned}
\end{equation*}
This proves \textup{(i)}.

For \textup{(ii)}, use \eqref{eq:heatbath} at the current parameters
$(\alpha,h)=(\alpha_n,h_n)$ to obtain
\begin{equation}\label{eq:technical-noise-expansion}
\begin{aligned}
 v_n={}&\frac1N\sum_e\{X_e+(1-2X_e)L_e\}\\
 ={}&Q+(1-2Q)F_{\alpha,h}(Q) +\frac1N\sum_e(1-2X_e)\{L_e-F_{\alpha,h}(Q)\}.
\end{aligned}
\end{equation}
Since $|1-2X_e|\le1$ and $\sup\ell'=1/4$, \eqref{eq:local-moments}
implies
\begin{equation*}
 \left\lVert v_n-\{Q+(1-2Q)F_{\alpha,h}(Q)\}\right\rVert_1 \le\frac1N\sum_e\norm{L_e-F_{\alpha,h}(Q)}_1 \le\frac{|\alpha|}{4nN}\sum_e\norm{M_e-nQ^2}_1 \le Cn^{-1/2}.
\end{equation*}
By \eqref{eq:parameter-expansion},
$\sup_{q\in[0,1]}|F_{\alpha_n,h_n}(q)-F(q)|\le C/n$.
Since $|1-2Q|\le1$, the triangle inequality gives
\begin{equation*}
\begin{aligned}
 &\left\lVert v_n-\{Q+(1-2Q)F(Q)\}\right\rVert_1\\
 &\quad\le\left\lVert v_n-\{Q+(1-2Q)F_{\alpha,h}(Q)\}\right\rVert_1 +\sup_{q\in[0,1]}|F_{\alpha,h}(q)-F(q)| \le Cn^{-1/2}+Cn^{-1}\le Cn^{-1/2}.
\end{aligned}
\end{equation*}
This proves \eqref{eq:taylor-noise}.

For \textup{(iii)}, put $W_n=\sqrt n(Q-2/3)$ and fix
$f\in C_b^1(\R)$.
By \eqref{eq:closure-remainder}, the triangle inequality and
the Cauchy--Schwarz inequality,
\begin{equation}\label{eq:tested-error-decomposition}
\begin{aligned}
 |\E[\varepsilon_n f(W_n)]|
 &\le\frac1n\left|\E[\{V_n-A_{\alpha,h}(Q)\}
                  \alpha\ell'(h+\alpha Q^2)f(W_n)]\right|\\
 &\quad+\frac1n\left|\E[\{K_n-K_{\alpha,h}(Q)\}
                  \tfrac12\alpha^2\ell''(h+\alpha Q^2)f(W_n)]\right|
          +Cn^{-2}\norm{f(W_n)}_2.
\end{aligned}
\end{equation}
The last term is at most $Cn^{-2}\norm{f}_\infty$.
To bound the two expectations, we use the following estimate.
\begin{lemma}\label{thm:tested-triangle}
For parameters in a fixed compact subset of $(-4,4)\times\R$, put
$W_n=\sqrt n(Q-2/3)$.
For $f\in C_b^1(\R)$ and $\omega\in C^1([0,1])$, the functions
$A_{\alpha,h},K_{\alpha,h}$ satisfy
\begin{equation}\label{eq:tested-moments}
 \left|\E[(V_n-A_{\alpha,h}(Q))\omega(Q)f(W_n)]\right| +\left|\E[(K_n-K_{\alpha,h}(Q))\omega(Q)f(W_n)]\right| \le\frac Cn\norm{\omega}_{C^1}\norm{f}_{C_b^1}.
\end{equation}
Here $\norm{\omega}_{C^1}=\norm{\omega}_\infty+\norm{\omega'}_\infty$ on $[0,1]$.
The constant is independent of $f$ and $\omega$.
\end{lemma}
The proof is given in Section~\ref{sec:tested-moments-proof}.
For the first expectation in \eqref{eq:tested-error-decomposition},
take $\omega(q)=\alpha\ell'(h+\alpha q^2)$.
For $q\in[0,1]$,
\begin{equation*}
 \omega'(q)=2\alpha^2q\ell''(h+\alpha q^2),\qquad \norm{\omega}_{C^1} \le |\alpha|\norm{\ell'}_\infty +2\alpha^2\norm{\ell''}_\infty\le C.
\end{equation*}
Hence \eqref{eq:tested-moments} gives
\begin{equation*}
 \left|\E[\{V_n-A_{\alpha,h}(Q)\}\alpha\ell'(h+\alpha Q^2)f(W_n)]\right|
 \le\frac Cn\norm{f}_{C_b^1}.
\end{equation*}

For the second expectation in \eqref{eq:tested-error-decomposition}, take $\omega(q)=\alpha^2\ell''(h+\alpha q^2)/2$.
Then
\begin{equation*}
 \omega'(q)=\alpha^3q\ell^{(3)}(h+\alpha q^2),\qquad \norm{\omega}_{C^1} \le\frac{\alpha^2}2\norm{\ell''}_\infty +|\alpha|^3\norm{\ell^{(3)}}_\infty\le C.
\end{equation*}
Applying \eqref{eq:tested-moments} again, we obtain
\begin{equation*}
 \left|\E[\{K_n-K_{\alpha,h}(Q)\}\tfrac12\alpha^2\ell''(h+\alpha Q^2)f(W_n)]\right|
 \le\frac Cn\norm{f}_{C_b^1}.
\end{equation*}
Substituting these bounds into \eqref{eq:tested-error-decomposition} gives
\begin{equation*}
 |\E[\varepsilon_n f(W_n)]| \le\frac1n\left(\frac Cn+\frac Cn\right)\norm{f}_{C_b^1} +Cn^{-2}\norm{f}_\infty \le Cn^{-2}\norm{f}_{C_b^1}.
\end{equation*}
Multiplying by $n^{3/2}$ proves \eqref{eq:quartic-tested-error}.
\end{proof}

\section{Proof of Proposition~\ref{thm:collapse}}\label{sec:motif-proof}
{
We induct on the number of edges of \(H\). A single-edge update expresses the drift of \(t_H-Q^{e_H}\) as a restoring term plus the corresponding differences for graphs with one edge removed.
The induction hypothesis controls these differences, and a rooted-count estimate bounds the remaining error.
 Stationarity applied to powers of \(|t_H-Q^{e_H}|\), together with the small update size, yields the required moment bound.
}

\begin{proof}
Fix $s\ge2$ and put $T_H=t_H-Q^{e_H}$.
We prove $\norm{T_H}_s\le C_{H,s}/n$ by induction on $k=e_H$.\\
If $k=0$, then $t_H=1$; if $k=1$, then $t_H=Q$.
Thus $T_H=0$ in both cases, including graphs with isolated vertices.\\
Suppose $k\ge2$ and the bound holds for all graphs with fewer edges.

For $s\ge2$, applying Taylor's formula to $x\mapsto |x|^s$ gives
\begin{equation*}
 \bigl||x+y|^s-|x|^s-s|x|^{s-2}xy\bigr|
 \le C_s\{|x|^{s-2}y^2+|y|^s\}.
\end{equation*}
Use the edge resampling from Section~\ref{sec:stein}.
Let $X'$ be obtained by replacing $X_I$ with $X_I'$, with all other
edges unchanged. Write $t_H'=t_H(X')$ and $T_H'=t_H'-(Q')^k$.
Taking $x=T_H$ and $y=T_H'-T_H$, we obtain
\begin{equation}\label{eq:motif-moment-taylor}
 \bigl||T_H'|^s-|T_H|^s -s|T_H|^{s-2}T_H(T_H'-T_H)\bigr| \le C_s\{|T_H|^{s-2}(T_H'-T_H)^2+|T_H'-T_H|^s\}.
\end{equation}
Since $X'$ and $X$ have the same law,
$\E|T_H'|^s=\E|T_H|^s$.
Taking expectations in \eqref{eq:motif-moment-taylor}, multiplying by
$N$ and conditioning on $X$ gives
\begin{equation*}
 s\left|\E\left[|T_H|^{s-2}T_H\,N\E(T_H'-T_H\mid X)\right]\right| \le C_sN\E\{|T_H|^{s-2}(T_H'-T_H)^2+|T_H'-T_H|^s\}.
\end{equation*}
We next bound the increments on the right hand side.
Only injective maps sending an edge of $H$ to $I$ can change after the update.
For each edge of $H$, there are $2(n-2)_{v_H-2}$ such maps. Thus
\begin{equation*}
 |t_H'-t_H|\le\frac{2k(n-2)_{v_H-2}}{(n)_{v_H}}=\frac{k}{N}.
\end{equation*}
Also, $|(Q')^k-Q^k|\le k|Q'-Q|\le k/N$. Hence
\begin{equation}\label{eq:motif-jumps}
 |T_H'-T_H|\le\frac{2k}{N}.
\end{equation}
Using this bound and dividing by $s$, we obtain
\begin{equation}\label{eq:motif-moment-drift}
 \left|\E\left[|T_H|^{s-2}T_H\,N\E(T_H'-T_H\mid X)\right]\right|
 \le\frac{C_{H,s}}N\E|T_H|^{s-2}+C_{H,s}N^{1-s}.
\end{equation}

To use \eqref{eq:motif-moment-drift}, we compute
$N\E(T_H'-T_H\mid X)$.
For each $a\in E(H)$ and possible edge $e=\{i,j\}$, define
\begin{equation}\label{eq:rooted-density}
 U_e^{H,a}=\frac1{2(n-2)_{v_H-2}}
 \sum_{\substack{\phi:V(H)\hookrightarrow[n]\\ \phi(a)=\{i,j\}}}
 \prod_{b\in E(H)\setminus\{a\}}X_{\phi(b)}.
\end{equation}
The sum includes both assignments of the endpoints of $a$.
Thus $0\le U_e^{H,a}\le1$, and $U_e^{H,a}$ does not depend on $X_e$.
Here $H-a$ denotes the graph obtained from $H$ by deleting the edge $a$
and retaining all vertices.
Since $2N(n-2)_{v_H-2}=(n)_{v_H}$, counting the corresponding maps gives
\begin{equation}\label{eq:rooted-identities}
 t_H'-t_H=\frac{X_I'-X_I}{N}\sum_{a\in E(H)}U_I^{H,a},\qquad \frac1N\sum_eU_e^{H,a}=t_{H-a},\qquad
 \frac1N\sum_eX_eU_e^{H,a}=t_H.
\end{equation}
Conditioning on $X$ and averaging over $I$, we obtain
\begin{equation}\label{eq:motif-generator}
 N\E(t_H'-t_H\mid X)
 =\frac1N\sum_e\sum_{a\in E(H)}(L_e-X_e)U_e^{H,a} =\frac1N\sum_e\sum_{a\in E(H)}L_eU_e^{H,a}-kt_H.
\end{equation}
For $Q^k$, Taylor's formula and $|Q'-Q|\le N^{-1}$ give
\begin{equation*}
 |(Q')^k-Q^k-kQ^{k-1}(Q'-Q)|
 \le C_H(Q'-Q)^2\le C_HN^{-2}.
\end{equation*}
Using $N\E(Q'-Q\mid X)=\bar L-Q$, we get
\begin{equation}\label{eq:motif-power-drift}
 N\E((Q')^k-Q^k\mid X)
 =kQ^{k-1}(\bar L-Q)+O_H(N^{-1}).
\end{equation}
Subtract \eqref{eq:motif-power-drift} from \eqref{eq:motif-generator}.
Separating $L_e=\bar L+(L_e-\bar L)$ and using
$\sum_e(L_e-\bar L)=0$, we obtain
\begin{equation}\label{eq:transverse-generator}
\begin{aligned}
 N\E(T_H'-T_H\mid X)
 &=-kT_H+\bar L\sum_{a\in E(H)}T_{H-a}+\epsilon_H,\\
 \epsilon_H
 &=\frac1N\sum_e\sum_{a\in E(H)}
       (L_e-\bar L)(U_e^{H,a}-t_{H-a})+O_H(N^{-1}).
\end{aligned}
\end{equation}

To bound $\epsilon_H$, we use the next lemma for the counts in
\eqref{eq:rooted-density}.
Its proof is given in Section~\ref{sec:rooted-proof}.
\begin{lemma}\label{lem:rooted}
For every fixed $H,a$ and $r\ge1$,
\begin{align}
 \norm{U_e^{H,a}-U_f^{H,a}}_r+
 \norm{U_e^{H,a}-t_{H-a}}_r&\le C_{H,r}n^{-1/2},\label{eq:rooted-difference}\\
 \left\lVert N^{-1}\sum_e(L_e-\bar L)(U_e^{H,a}-t_{H-a})\right\rVert_r
 &\le C_{H,r}n^{-1}.\label{eq:rooted-product}
\end{align}
\end{lemma}
By \eqref{eq:rooted-product} and $N^{-1}=O(n^{-2})$,
\begin{equation*}
 \norm{\epsilon_H}_s
 \le\sum_{a\in E(H)}
 \left\lVert\frac1N\sum_e(L_e-\bar L)(U_e^{H,a}-t_{H-a})\right\rVert_s
       +\frac{C_H}{N} \le C_{H,s}n^{-1}.
\end{equation*}
Each $H-a$ has $k-1$ edges.
The induction hypothesis, $0\le\bar L\le1$ and the bound for
$\norm{\epsilon_H}_s$ give
\begin{equation}\label{eq:motif-induction-error}
 \left\lVert\bar L\sum_{a\in E(H)}T_{H-a}+\epsilon_H\right\rVert_s
 \le\sum_{a\in E(H)}\norm{T_{H-a}}_s+\norm{\epsilon_H}_s \le C_{H,s}n^{-1}.
\end{equation}
Substituting \eqref{eq:transverse-generator} into
\eqref{eq:motif-moment-drift} and applying H\"older's inequality and
\eqref{eq:motif-induction-error}, we obtain
\begin{equation*}
\begin{aligned}
 k\norm{T_H}_s^s
 &\le\left|\E\left[|T_H|^{s-2}T_H
       \left\{\bar L\sum_{a\in E(H)}T_{H-a}+\epsilon_H\right\}\right]\right| +\frac{C_{H,s}}N\E|T_H|^{s-2}+C_{H,s}N^{1-s}\\
 &\le\frac{C_{H,s}}n\norm{T_H}_s^{s-1}
       +\frac{C_{H,s}}N\norm{T_H}_s^{s-2}+C_{H,s}N^{1-s}.
\end{aligned}
\end{equation*}
For the first two terms on the right, Young's inequality gives
\begin{equation*}
 \frac{C_{H,s}}n\norm{T_H}_s^{s-1}
 \le\frac k4\norm{T_H}_s^s+C_{H,s}n^{-s},\qquad \frac{C_{H,s}}N\norm{T_H}_s^{s-2}
 \le\frac k4\norm{T_H}_s^s+C_{H,s}N^{-s/2}.
\end{equation*}
Since $N\asymp n^2$ and $s\ge2$,
$N^{-s/2}+N^{1-s}\le C_sn^{-s}$.
Substituting these estimates, we obtain
\begin{equation*}
 k\norm{T_H}_s^s\le\frac k2\norm{T_H}_s^s+C_{H,s}n^{-s}.
\end{equation*}
Thus $\norm{T_H}_s\le C_{H,s}/n$, which completes the induction.
For $1\le s<2$, the result follows from
$\norm{T_H}_s\le\norm{T_H}_2\le C_{H,2}/n$.
\end{proof}

\section{Proofs of the lemmas}\label{sec:lemma-proofs}

Sections~\ref{sec:smooth-factors}--\ref{sec:quartic-moment-proof}
prove the Stein and scalar moment estimates used in Section~\ref{sec:proofs}.
The subsequent proofs use the conditional row estimates stated in
Section~\ref{sec:conditional-rows} and proved in
Section~\ref{sec:conditional-row-proofs} to control local fields,
empirical moments and rooted subgraph counts. The one-row moment and variance calculations
are proved in Sections~\ref{sec:overlap-proof}--\ref{sec:actual-refined-proof}.

\subsection{Proof of Lemma~\ref{lem:smooth-stein}}\label{sec:smooth-factors}

\begin{proof}
We first derive \eqref{eq:stein-solution} using an integrating factor
\citep[Section~2.1]{CGS2011}.
Put $\bar\varphi=\varphi-\pi\varphi$.
Since $\frac49\pi'=b\pi$, multiplying \eqref{eq:stein-equation}
by $\pi$ gives
\begin{equation*}
 \frac49(f\pi)'=\bar\varphi\pi.
\end{equation*}
For a bounded solution, $f(z)\pi(z)\to0$ as $z\to-\infty$.
Integration therefore gives \eqref{eq:stein-solution}.
Since $\int_{\R}\bar\varphi(t)\pi(t)\dd t=0$, we also have
\begin{equation}\label{eq:stein-tail-solution}
 f(z)=-\frac{9}{4\pi(z)}\int_z^\infty\bar\varphi(t)\pi(t)\dd t
     =-\frac94\int_0^\infty
        \bar\varphi(z+s)\frac{\pi(z+s)}{\pi(z)}\dd s.
\end{equation}
Conversely, differentiation of \eqref{eq:stein-solution} verifies
\eqref{eq:stein-equation} and gives $f\in C^1$.
Since $\varphi$ is locally absolutely continuous,
$\frac49f'=\bar\varphi-bf$ implies that $f'$ is locally absolutely continuous.

We next prove the uniqueness of the solution.
For two bounded solutions $f_1,f_2$ of \eqref{eq:stein-equation}, their difference $u=f_1-f_2$ satisfies $\frac49u'+bu=0$.
Since $\frac49\pi'=b\pi$, we have $(u\pi)'=0$, hence $u=C/\pi$ for some constant $C$.
Since $\pi(z)\to0$ as $|z|\to\infty$, boundedness of $u$ forces $C=0$, proving uniqueness.

We now prove \eqref{eq:smooth-stein-factors}.
Put $m=1$ for the Gaussian target and $m=3$ for the quartic target.
The Lipschitz condition gives
\begin{equation}\label{eq:stein-test-growth}
 |\bar\varphi(t)|\le|t|+\int_{\R}|x|\pi(x)\dd x,\qquad
 |\varphi'(t)|\le1\quad\text{for almost every }t.
\end{equation}
We first bound $f$ in the right tail.
Choose $R\ge1$ so that $b(t)\le-ct^m$ for $t\ge R$.
For $z\ge R$ and $s\ge0$, the density ratio in
\eqref{eq:stein-tail-solution} satisfies
\begin{equation}\label{eq:stein-tail-kernel}
 k(z,s)=\frac{\pi(z+s)}{\pi(z)}
       =\exp\left\{\frac94\int_0^s b(z+t)\dd t\right\}
       \le\ee^{-cz^m s}.
\end{equation}
For every integer $j\ge0$,
\begin{equation}\label{eq:stein-exponential-integrals}
 \int_0^\infty s^j\ee^{-cz^m s}\dd s
       =\frac{j!}{(cz^m)^{j+1}}.
\end{equation}
Substituting \eqref{eq:stein-tail-kernel} into \eqref{eq:stein-tail-solution} and applying \eqref{eq:stein-exponential-integrals} with $j=0,1$ gives
\begin{equation}\label{eq:stein-f-tail}
 |f(z)|\le C\int_0^\infty(1+z+s)k(z,s)\dd s
          \le Cz^{1-m}.
\end{equation}

We next bound $f'$.
The Lipschitz condition and the exponential bound in
\eqref{eq:stein-tail-kernel} justify differentiation under the integral
by dominated convergence applied to difference quotients. This gives
\begin{equation}\label{eq:stein-tail-derivatives}
\begin{aligned}
 f'(z)&=-\frac94\int_0^\infty
       \{\bar\varphi'(z+s)k(z,s)+\bar\varphi(z+s)\partial_z k(z,s)\}\dd s,\\
 f''(z)&=-\frac94\int_0^\infty
       \{\bar\varphi''(z+s)k(z,s)+2\bar\varphi'(z+s)\partial_z k(z,s)
                         +\bar\varphi(z+s)\partial_z^2 k(z,s)\}\dd s.
\end{aligned}
\end{equation}
Here the first identity uses the almost everywhere derivative $\varphi'$; the second is used only for $\varphi(z)=z^j$, $j=1,2$, in the proof of \eqref{eq:quartic-polynomial-stein-factors}.
To estimate the kernel derivatives in these formulas, we use
\eqref{eq:stein-tail-kernel} to obtain
\begin{equation}\label{eq:stein-kernel-derivatives}
\begin{aligned}
 \partial_z k(z,s)
 &=\frac94\{b(z+s)-b(z)\}k(z,s),\\
 \partial_z^2 k(z,s)
 &=\left[\frac{81}{16}\{b(z+s)-b(z)\}^2
       +\frac94\{b'(z+s)-b'(z)\}\right]k(z,s).
\end{aligned}
\end{equation}
For the quartic target and $z\ge R\ge1$,
\begin{equation*}
 |b(z+s)-b(z)|\le C(z^2s+zs^2+s^3),\qquad |b'(z+s)-b'(z)|\le C(zs+s^2).
\end{equation*}
For the Gaussian target, $b(z+s)-b(z)=-27a^2s/4$ and
$b'(z+s)-b'(z)=0$.
Combining these bounds with \eqref{eq:stein-exponential-integrals}
yields, in both cases,
\begin{equation}\label{eq:stein-kernel-integrals}
\begin{aligned}
 \int_0^\infty k(z,s)\dd s&\le Cz^{-m},\qquad \int_0^\infty|\partial_z k(z,s)|\dd s\le Cz^{-m-1},\\
 \int_0^\infty|\partial_z^2 k(z,s)|\dd s&\le Cz^{-m-2},\qquad \int_0^\infty(z+s)|\partial_z k(z,s)|\dd s\le Cz^{-m}.
\end{aligned}
\end{equation}
Using the bound on $\bar\varphi$ in \eqref{eq:stein-test-growth} and \eqref{eq:stein-kernel-integrals} in the formula for $f'$ in \eqref{eq:stein-tail-derivatives}, we obtain
\begin{equation}\label{eq:stein-fprime-tail}
 |f'(z)|\le Cz^{-m},\qquad z\ge R.
\end{equation}

Reflecting the density and replacing $b(z)$ by $-b(-z)$ gives the bounds \eqref{eq:stein-f-tail} and \eqref{eq:stein-fprime-tail} for $z\le -R$, with $z$ replaced by $|z|$.
On $[-R,R]$, the density has a positive minimum, so
\eqref{eq:stein-solution} and $\frac49f'=\bar\varphi-bf$ bound $f$ and $f'$.
Differentiating the Stein equation almost everywhere gives
\begin{equation}\label{eq:stein-third-derivative}
 \frac49f''=\varphi'-b'f-bf'\qquad\text{almost everywhere}.
\end{equation}
Since $b(z)=O(|z|^m)$ and $b'(z)=O(|z|^{m-1})$ in the tails,
the bounds for $\varphi',f,f'$ in \eqref{eq:stein-test-growth},
\eqref{eq:stein-f-tail} and \eqref{eq:stein-fprime-tail}, together with
the boundedness of $f,f'$ on $[-R,R]$, imply by
\eqref{eq:stein-third-derivative} that $\operatorname*{ess\,sup}|f''|\le C$.
Local absolute continuity of $f'$ therefore gives $\operatorname{Lip}(f')\le C$.
This proves \eqref{eq:smooth-stein-factors} and the boundedness of
\eqref{eq:stein-solution}.

For \eqref{eq:quartic-polynomial-stein-factors}, take the quartic target
and $\varphi(z)=z^j$, $j=1,2$.
Its finite moments make \eqref{eq:stein-solution} well defined.
We use \eqref{eq:stein-tail-solution} and
\eqref{eq:stein-tail-derivatives} with
$\bar\varphi(z)=z^j-\pi\varphi$.
For $d,r\in\{0,1,2\}$, the bound $(z+s)^d\le C(z^d+s^d)$,
\eqref{eq:stein-kernel-derivatives} and
\eqref{eq:stein-exponential-integrals} give
\begin{equation}\label{eq:stein-polynomial-kernel}
 \int_0^\infty(z+s)^d|\partial_z^r k(z,s)|\dd s
 \le Cz^{d-3-r},\qquad z\ge R.
\end{equation}
Applying \eqref{eq:stein-polynomial-kernel} to the formulas for $f,f',f''$ in \eqref{eq:stein-tail-solution} and \eqref{eq:stein-tail-derivatives} yields
\begin{equation*}
 |f(z)|\le Cz^{j-3},\qquad
 |f'(z)|\le Cz^{j-4},\qquad
 |f''(z)|\le Cz^{j-5},\qquad z\ge R.
\end{equation*}
Reflection gives the same estimates in the left tail, with $z$ replaced
by $|z|$.
On $[-R,R]$, use \eqref{eq:stein-solution}, \eqref{eq:stein-equation} and \eqref{eq:stein-third-derivative}.
Since $j=1,2$, all three functions are bounded.
This proves \eqref{eq:quartic-polynomial-stein-factors}.
\end{proof}

\subsection{Proof of Lemma~\ref{lem:gaussian-center}}\label{sec:center-moments}

\begin{proof}
For \eqref{eq:triangle-global-derivative}, differentiation of
\eqref{eq:critical-F} gives
\begin{equation*}
\begin{aligned}
 F'(q)&=2\alpha_cqF(q)(1-F(q)),\\
 F''(q)&=2\alpha_cF(q)(1-F(q))
          \{1-2\alpha_cq^2(2F(q)-1)\}.
\end{aligned}
\end{equation*}
If $F(q)<1/2$, then $F''(q)>0$.
If $F(q)\ge1/2$ and $q>0$, then
\begin{equation*}
 \frac{\dd}{\dd q}\{2\alpha_cq^2(2F(q)-1)\}
 =4\alpha_cq(2F(q)-1)+4\alpha_cq^2F'(q)>0.
\end{equation*}
At $q=2/3$,
\begin{equation*}
 F(2/3)=\frac23,\qquad
 2\alpha_c(2/3)^2(2F(2/3)-1)=1,\qquad F'(2/3)=1.
\end{equation*}
Hence
\begin{equation*}
 F''(q)>0\quad(0\le q<2/3),\qquad
 F''(q)<0\quad(2/3<q\le1),
\end{equation*}
and $F'(q)<1$ for $q\ne2/3$.
Further differentiation and Taylor's formula give
\begin{equation}\label{eq:critical-derivative-expansion}
 F''(2/3)=0,\qquad F'''(2/3)=-\frac{27}{2},\qquad
 1-F'(2/3+x)=\frac{27}{4}x^2+O(x^3).
\end{equation}
By continuity and compactness,
\begin{equation*}
 \inf_{\substack{0\le q\le1\\q\ne2/3}}
       \frac{1-F'(q)}{(q-2/3)^2}>0.
\end{equation*}
This proves the first bound in \eqref{eq:triangle-global-derivative}.
For the second bound, by the first bound in \eqref{eq:triangle-global-derivative},
\begin{equation*}
\begin{aligned}
 -(q-2/3)\{F(q)-q\}
 &=(q-2/3)^2\int_0^1
       \{1-F'(2/3+t(q-2/3))\}\dd t\\
 &\ge c(q-2/3)^4\int_0^1t^2\dd t
   =\frac c3(q-2/3)^4.
\end{aligned}
\end{equation*}

For the root expansion \eqref{eq:endpoint-root}, further differentiation
gives $F^{(4)}(2/3)=-81/4$.
Together with \eqref{eq:beta-values}, Taylor's formula yields
\begin{equation}\label{eq:endpoint-expansion}
\begin{aligned}
 F(2/3+x)-(2/3+x)&=-\frac94x^3-\frac{27}{32}x^4+O(x^5),\\
 \beta(2/3+x)&=\tau+\beta'(2/3)x+O(x^2).
\end{aligned}
\end{equation}
Set $t=n^{-1/3}$ and $q=2/3+ty$.
By \eqref{eq:endpoint-expansion}, the equation $\mathcal B_n(q)=0$ becomes
\begin{equation*}
\begin{aligned}
 0=n\mathcal B_n(2/3+ty)
 &=\frac{F(2/3+ty)-(2/3+ty)}{t^3}+\beta(2/3+ty)\\
 &=-\frac94y^3+\tau
       +t\left(-\frac{27}{32}y^4+\beta'(2/3)y\right)+O(t^2).
\end{aligned}
\end{equation*}
The right-hand side extends to a $C^2$ function at $t=0$, and
\begin{equation*}
 -\frac94a^3+\tau=0,\qquad
 \left.\frac{\partial}{\partial y}
       \left(-\frac94y^3+\tau\right)\right|_{y=a}
       =-\frac{27}{4}a^2\ne0.
\end{equation*}
The implicit function theorem gives a root of the form
\begin{equation*}
 y(t)=a+dt+O(t^2),\qquad
 u_n=\frac23+an^{-1/3}+dn^{-2/3}+O(n^{-1}).
\end{equation*}
Substituting this into $\mathcal B_n(u_n)=0$ gives
\begin{equation*}
 0=\left(-\frac94a^3+\tau\right)n^{-1} +\left(-\frac{27}{4}a^2d-\frac{27}{32}a^4
                    +\beta'(2/3)a\right)n^{-4/3}
       +O(n^{-5/3}).
\end{equation*}
Since $a^3=4\tau/9$, comparison of the coefficients gives
\begin{equation*}
 d=\frac{-27a^4/32+\beta'(2/3)a}{(27/4)a^2}=c_2,
\end{equation*}
where the last identity follows from \eqref{eq:beta-values} and
the definition of $c_2$ in Theorem~\ref{thm:triangle-gaussian}.
This proves \eqref{eq:endpoint-root}.

We next prove uniqueness and \eqref{eq:triangle-dissipation}.
Since $\mathcal B_n(u_n)=0$, by \eqref{eq:triangle-global-derivative}, for $q\in[0,1]\setminus\{u_n\}$,
\begin{equation*}
\begin{aligned}
 -\frac{\mathcal B_n(q)}{q-u_n}
 &=\int_0^1\left\{1-F'(u_n+t(q-u_n))
                  -\frac{\beta'(u_n+t(q-u_n))}{n}\right\}\dd t\\
 &\ge c\int_0^1\{u_n-2/3+t(q-u_n)\}^2\dd t
                 -\frac{\norm{\beta'}_\infty}{n}.
\end{aligned}
\end{equation*}
The quadratic integral satisfies
\begin{equation*}
\begin{aligned}
 \int_0^1\{u_n-2/3+t(q-u_n)\}^2\dd t
 &=(u_n-2/3)^2+(u_n-2/3)(q-u_n)+\frac13(q-u_n)^2\\
 &\ge\frac1{18}\{(u_n-2/3)^2+(q-u_n)^2\}.
\end{aligned}
\end{equation*}
By \eqref{eq:endpoint-root},
$(u_n-2/3)^2=a^2n^{-2/3}+O(n^{-1})$.
Thus, for all large $n$,
\begin{equation}\label{eq:gaussian-secant-bound}
 -\frac{\mathcal B_n(q)}{q-u_n}
 \ge\frac c{18}\{(q-u_n)^2+(u_n-2/3)^2\}
              -\frac{\norm{\beta'}_\infty}{n}
 \ge c_*\{(q-u_n)^2+n^{-2/3}\}>0.
\end{equation}
The strict inequality proves uniqueness.
Multiplying \eqref{eq:gaussian-secant-bound} by $(q-u_n)^2$ gives
\eqref{eq:triangle-dissipation}.

Finally, \eqref{eq:critical-derivative-expansion} and
\eqref{eq:endpoint-root} give
\begin{equation}\label{eq:Gaussian-local-drift}
\begin{aligned}
 \mathcal B_n'(u_n)
 &=F'(u_n)-1+\frac{\beta'(u_n)}n\\
 &=-\frac{27}{4}(u_n-2/3)^2+O(|u_n-2/3|^3+n^{-1}) =-\frac{27a^2}{4}n^{-2/3}+O(n^{-1}).
\end{aligned}
\end{equation}
\end{proof}

\subsection{Proof of Lemma~\ref{lem:quartic-moment}}\label{sec:quartic-moment-proof}

\begin{proof}
We prove $\E(Q-2/3)^4\le Cn^{-2}$.
Since $Q'$ and $Q$ have the same law, \eqref{eq:heatbath} gives
\begin{equation}\label{eq:quartic-stationarity-balance}
\begin{aligned}
 0&=\E\{(Q'-2/3)^2-(Q-2/3)^2\}\\
 &=2\E[(Q-2/3)(Q'-Q)]+\E(Q'-Q)^2 =\frac2N\E[(Q-2/3)b_n]+\frac{\E v_n}{N^2}.
\end{aligned}
\end{equation}
By Proposition~\ref{prop:conditional-moments}\textup{(i)} and
\eqref{eq:beta-correction},
\begin{equation*}
 b_n=F(Q)-Q+\frac{\beta(Q)}n
      +\varepsilon_n,
 \qquad \norm{\varepsilon_n}_2\le Cn^{-3/2}.
\end{equation*}
Substituting this expression for $b_n$ into \eqref{eq:quartic-stationarity-balance} gives
\begin{equation}\label{eq:quartic-moment-balance}
\begin{aligned}
 -\E[(Q-2/3)\{F(Q)-Q\}]
 &=\frac1n\E[(Q-2/3)\beta(Q)]\\
 &\quad+\E[(Q-2/3)\varepsilon_n]+\frac{\E v_n}{2N}.
\end{aligned}
\end{equation}
For the left-hand side, \eqref{eq:triangle-global-derivative} gives
\begin{equation*}
 -\E[(Q-2/3)\{F(Q)-Q\}]\ge\frac c3\E(Q-2/3)^4.
\end{equation*}
For the right-hand side, $\beta(2/3)=0$ implies
\begin{equation*}
 |\beta(q)|\le\norm{\beta'}_\infty|q-2/3|\le C|q-2/3|.
\end{equation*}
Using this bound, the Cauchy--Schwarz inequality and
$0\le v_n\le1$, we obtain
\begin{equation*}
\begin{aligned}
 \frac1n\bigl|\E[(Q-2/3)\beta(Q)]\bigr|
 &\le\frac Cn\E(Q-2/3)^2,\\
 \bigl|\E[(Q-2/3)\varepsilon_n]\bigr|
 &\le\{\E(Q-2/3)^2\}^{1/2}\norm{\varepsilon_n}_2 \le Cn^{-3/2}\{\E(Q-2/3)^2\}^{1/2},\\
 \frac{\E v_n}{2N}&\le\frac1{2N}\le Cn^{-2}.
\end{aligned}
\end{equation*}
Substituting these estimates into \eqref{eq:quartic-moment-balance} yields
\begin{equation*}
 \frac c3\E(Q-2/3)^4
 \le Cn^{-1}\E(Q-2/3)^2
       +Cn^{-3/2}\{\E(Q-2/3)^2\}^{1/2}+Cn^{-2}.
\end{equation*}
Multiplying by $n^2$ and using $W_n=\sqrt n(Q-2/3)$ and
H\"older's inequality gives
\begin{equation}\label{eq:quartic-tightness}
 \frac c3\E|W_n|^4
 \le C\{\E W_n^2+(\E W_n^2)^{1/2}+1\} \le C\{(\E|W_n|^4)^{1/2}+(\E|W_n|^4)^{1/4}+1\}.
\end{equation}
Young's inequality now gives
\begin{equation*}
 \frac c3\E|W_n|^4\le\frac c6\E|W_n|^4+C
 \quad\Longrightarrow\quad\sup_n\E|W_n|^4\le\frac{6C}{c}<\infty.
\end{equation*}
\end{proof}

\subsection{Proof of Lemma~\ref{thm:fields}}\label{sec:row-contrast}
\begin{proof}
For $U=d_i-d_j$ or $U=M_e-M_f$, choose $\mathcal I$ to contain
the vertices involved, so that $k\le4$.
We use the conditional row law from Section~\ref{sec:conditional-rows}.
Lemma~\ref{lem:row-interpolation} at $t=1$ gives
\begin{equation}\label{eq:field-conditional-comparison}
 \norm{U-\E(U\mid\calF_{\mathcal I})}_{L^r(\PP_n(\cdot\mid\calF_{\mathcal I}))}
 \le C_{k,r}\norm{\boldsymbol\delta U}_{\ell^2}\le C_r\sqrt n.
\end{equation}
The conditional mean is bounded by \eqref{eq:interpolated-contrast-means}.
Minkowski's inequality and the tower property therefore give
\begin{align}
 \norm U_r
 &\le\norm{U-\E(U\mid\calF_{\mathcal I})}_r
       +\norm{\E(U\mid\calF_{\mathcal I})}_r,
       \label{eq:field-conditional-decomposition}\\
 \norm{U-\E(U\mid\calF_{\mathcal I})}_r^r
 &=\E\left[\norm{U-\E(U\mid\calF_{\mathcal I})}
       _{L^r(\PP_n(\cdot\mid\calF_{\mathcal I}))}^r\right]
 \le C_r n^{r/2}.\label{eq:conditional-tower-moment}
\end{align}
Thus $\norm U_r\le C_r\sqrt n+C_r\le C_r\sqrt n$, proving the
first bound in \eqref{eq:field-moments}.
By \eqref{eq:fields} and $\sup\ell'=1/4$,
\begin{equation*}
 \norm{L_e-L_f}_r
 \le\frac{|\alpha|}{4n}\norm{M_e-M_f}_r\le C_rn^{-1/2},\qquad \norm{L_e-\bar L}_r
 \le\frac1N\sum_f\norm{L_e-L_f}_r\le C_rn^{-1/2}.
\end{equation*}

For \eqref{eq:local-moments}, use $\sum_i d_i=n(n-1)Q$ to write
\begin{equation*}
 \sum_{i<j}(d_i-d_j)^2
 =n\sum_i d_i^2-\left(\sum_i d_i\right)^2 =n\sum_i\{d_i-(n-1)Q\}^2.
\end{equation*}
Dividing by $2nN$ gives
\begin{equation}\label{eq:degree-difference}
 V_n=\frac1N\sum_{i<j}\frac{(d_i-d_j)^2}{2n}.
\end{equation}
The first line of \eqref{eq:field-moments} yields
\begin{equation}\label{010}
 \norm{V_n}_r\le\frac1N\sum_{i<j}
                    \frac{\norm{d_i-d_j}_{2r}^2}{2n}\le C_r.
\end{equation}
By \eqref{eq:triangle-first-moment} and  \eqref{010}, we have
\begin{equation*}
\begin{aligned}
 \norm{\meanM-nQ^2}_r
 &\le\norm{V_n}_r+2\le C_r,\\
 \norm{M_e-nQ^2}_r
 &\le\norm{M_e-\meanM}_r+\norm{\meanM-nQ^2}_r \le\frac1N\sum_f\norm{M_e-M_f}_r+C_r\le C_r\sqrt n.
\end{aligned}
\end{equation*}
Finally, \eqref{eq:triangle-empirical} gives
\begin{equation*}
 \norm{K_n}_r\le\frac1{nN}\sum_e\norm{M_e-nQ^2}_{2r}^2\le C_r.
\end{equation*}
\end{proof}

\subsection{Proof of Lemma~\ref{thm:triangle-cubic}}\label{sec:cubic-proof}

\begin{proof}
We first prove the bound for $N^{-1}\sum_e(M_e-\meanM)^3$ in
\eqref{eq:overlap-cubic-average}. For each $e$,
\begin{equation*}
\begin{aligned}
 \frac1{N^2}\sum_{f,g}(M_e-M_f)^2(M_e-M_g)
 &=(M_e-\meanM)\frac1N\sum_f(M_e-M_f)^2\\
 &=(M_e-\meanM)\left\{(M_e-\meanM)^2
                  +\frac1N\sum_f(M_f-\meanM)^2\right\}.
\end{aligned}
\end{equation*}
Averaging over $e$ and using $\sum_e(M_e-\meanM)=0$ gives
\begin{equation}\label{eq:empirical-cubic-difference}
 \frac1N\sum_e(M_e-\meanM)^3
 =\frac1{N^3}\sum_{e,f,g}(M_e-M_f)^2(M_e-M_g).
\end{equation}
Put $Y_{efg}=n^{-3/2}(M_e-M_f)^2(M_e-M_g)$. Then
\begin{equation*}
 \left\lVert\frac1N\sum_e(M_e-\meanM)^3\right\rVert_2^2
 =n^3\E\left|\frac1{N^3}\sum_{e,f,g}Y_{efg}\right|^2.
\end{equation*}
It is therefore enough to bound this expectation by $C/n$.
Lemma~\ref{thm:fields} and H\"older's inequality give
\begin{equation*}
 \norm{Y_{efg}}_2
 \le n^{-3/2}\norm{M_e-M_f}_6^2\norm{M_e-M_g}_6\le C.
\end{equation*}
Consequently $|\E[Y_{efg}Y_{e'f'g'}]|\le C$ for every pair of triples.
Each product involves six edges, hence twelve endpoint positions.
If an endpoint is repeated, at most eleven distinct vertices occur.
There are at most $Cn^{11}$ such choices, whose contribution to the
squared average is at most $Cn^{11}/N^6\le C/n$. Thus
\begin{equation}\label{eq:cubic-averaging-budget}
 \E\left|\frac1{N^3}\sum_{e,f,g}Y_{efg}\right|^2
 =\frac1{N^6}\sum_{\substack{e,f,g\\e',f',g'}}
                      \E[Y_{efg}Y_{e'f'g'}] \le\frac Cn+
   \sup_{\substack{e,f,g,e',f',g'\\
                   |e\cup f\cup g\cup e'\cup f'\cup g'|=12}}
   |\E[Y_{efg}Y_{e'f'g'}]|.
\end{equation}

Fix six edges with twelve distinct endpoints, and let $\mathcal I$
be their endpoint set. Condition on $\calF_{\mathcal I}$ as in
Section~\ref{sec:conditional-rows}, with $k=12$ and $L=n-k$.
The conditional covariance identity gives
\begin{equation}\label{eq:cubic-conditional-product}
 \E[Y_{efg}Y_{e'f'g'}]
 =\E\left\{\cov_1(Y_{efg},Y_{e'f'g'})
                  +(\E_1Y_{efg})(\E_1Y_{e'f'g'})\right\}.
\end{equation}

We first estimate the means under the independent-row law $p_{B,0}$.
Since the six endpoints of $e,f,g$ are distinct, the centered overlaps
\begin{equation*}
 M_e-\E_0M_e,\qquad M_f-\E_0M_f,\qquad M_g-\E_0M_g
\end{equation*}
are independent and identically distributed under $p_{B,0}$:
centering removes the fixed internal-edge contributions. Put
\begin{equation*}
 a=\E_0(M_e-M_f),\qquad b=\E_0(M_e-M_g),\qquad |a|+|b|\le2k.
\end{equation*}
Using \eqref{eq:interpolated-contrast-moments} with $U=M_e-M_f$, $t=0$ and $r=2$ gives
\begin{equation*}
 \var_0(M_e)=\tfrac12\var_0(M_e-M_f)
 \le\tfrac12\E_0[(M_e-M_f)^2]\le Cn.
\end{equation*}
Expanding the product and using independence and centering, we obtain
\begin{equation}\label{eq:overlap-product-expansion}
 n^{3/2}\E_0Y_{efg}
 =\E_0[(M_e-M_f)^2(M_e-M_g)] =\E_0[(M_e-\E_0M_e)^3]+2(a+b)\var_0(M_e)+a^2b.
\end{equation}
It remains to bound $\E_0[(M_e-\E_0M_e)^3]$.
To quantify the variation of the exterior degrees, recall $B,L$ from
Section~\ref{sec:conditional-rows} and define
\begin{equation}\label{eq:triangle-background}
 q_B=\frac{\mathbf1^\top B\mathbf1}{L(L-1)},\qquad
 H_B=1+\frac{\norm{B\mathbf1-(L-1)q_B\mathbf1}_{\ell^2}}{n}.
\end{equation}
These definitions apply to any fixed vertex set $\mathcal I\subset[n]$
of size $k$, provided $L=n-k\ge2$. For every graph on $\mathcal I^c$
with adjacency matrix $B$, we have $1\le H_B\le C\sqrt n$.
The following lemma bounds the third centered moment of
$\mathbf z_1^\top\mathbf z_2$ for independent vectors
$\mathbf z_1,\mathbf z_2$ with law $p_B$.
\begin{lemma}\label{lem:overlap-cubic}
Let $\mathbf z_1=(z_{1u})_{u\in\mathcal I^c}$ and $\mathbf z_2=(z_{2u})_{u\in\mathcal I^c}$ be independent random vectors with law $p_B$ defined in \eqref{eq:one-row-law}; the subscripts $1,2$ label the copies.
 Then
\begin{equation}\label{eq:overlap-cubic}
 \left|\E_B\{\mathbf z_1^\top \mathbf z_2-\E_B(\mathbf z_1^\top \mathbf z_2)\}^3\right|
 \le CnH_B.
\end{equation}
\end{lemma}
The proof is given in Section~\ref{sec:overlap-proof}.
Combining \eqref{eq:overlap-cubic} with \eqref{eq:overlap-product-expansion} gives
\begin{equation*}
 |\E_0Y_{efg}|\le n^{-3/2}(CnH_B+Cn)\le CH_Bn^{-1/2}.
\end{equation*}
The same estimate holds for $Y_{e'f'g'}$.

We next bound the change from $\E_0$ to $\E_1$ and the conditional
covariance of the two triples.
H\"older's inequality and \eqref{eq:interpolated-contrast-moments}, applied with $r=12$ to $M_e-M_f$ and $M_e-M_g$, give
\begin{equation*}
 \sup_{0\le t\le1}\norm{Y_{efg}}_{L^4(p_{B,t})}
 \le n^{-3/2}\sup_{0\le t\le1}
       \norm{M_e-M_f}_{L^{12}(p_{B,t})}^2
       \norm{M_e-M_g}_{L^{12}(p_{B,t})}\le C.
\end{equation*}
The same bound holds for $Y_{e'f'g'}$. Since
$\norm{Y_{efg}-\E_tY_{efg}}_{L^4(p_{B,t})}
\le2\norm{Y_{efg}}_{L^4(p_{B,t})}$, their centered fourth moments are
bounded uniformly in $t$.
The two triples depend on disjoint groups of rows.
Applying \eqref{eq:separated-mean} and \eqref{eq:separated-covariance}
from Lemma~\ref{lem:separated} with $F=Y_{efg}$ and $G=Y_{e'f'g'}$ gives
\begin{equation*}
\begin{aligned}
 |\E_1Y_{efg}-\E_0Y_{efg}|
   +|\E_1Y_{e'f'g'}-\E_0Y_{e'f'g'}|&\le Cn^{-1/2},\\
 |\E_1Y_{efg}|+|\E_1Y_{e'f'g'}|
 &\le|\E_0Y_{efg}|+|\E_0Y_{e'f'g'}|+Cn^{-1/2}
   \le CH_Bn^{-1/2},\\
 |\cov_1(Y_{efg},Y_{e'f'g'})|&\le Cn^{-1}.
\end{aligned}
\end{equation*}
Substituting these estimates into \eqref{eq:cubic-conditional-product} gives
\begin{equation}\label{eq:cubic-product-background}
 |\E[Y_{efg}Y_{e'f'g'}]|
 \le\E|\cov_1(Y_{efg},Y_{e'f'g'})|
       +\E\bigl[|\E_1Y_{efg}|\,|\E_1Y_{e'f'g'}|\bigr] \le Cn^{-1}+Cn^{-1}\E H_B^2.
\end{equation}

It remains to bound $\E H_B^2$.
Let $S_{\mathcal I}$ be the number of edges with at least one endpoint
in $\mathcal I$. Then
\[
 NQ=\binom L2 q_B+S_{\mathcal I},
 \qquad
 0\le S_{\mathcal I}\le N-\binom L2.
\]
Since $0\le q_B\le1$, we obtain
\begin{equation}\label{eq:background-density}
 |q_B-Q|
 =\frac{\left|\left(N-\binom L2\right)q_B-S_{\mathcal I}\right|}{N} \le\frac{N-\binom L2}{N}
 =\frac{k(2n-k-1)}{n(n-1)}
 \le\frac{C_k}{n}.
\end{equation}
For each $u\in\mathcal I^c$, the identity
$(B\mathbf1)_u=d_u-\sum_{a\in\mathcal I}X_{au}$ gives
\[
 (B\mathbf1)_u-(L-1)q_B
 =d_u-(n-1)Q
   +kQ-\sum_{a\in\mathcal I}X_{au}
   +(L-1)(Q-q_B).
\]
By \eqref{eq:background-density} and
$\left|kQ-\sum_{a\in\mathcal I}X_{au}\right|\le k$,
the last three terms have absolute value at most $C_k$.
The Euclidean triangle inequality and the definition of $V_n$
in \eqref{eq:triangle-empirical} therefore yield
\[
\begin{aligned}
 &\norm{B\mathbf1-(L-1)q_B\mathbf1}_{\ell^2}\\
 &\quad\le
 \left(\sum_{u\in\mathcal I^c}\{d_u-(n-1)Q\}^2\right)^{1/2}
 +C_k\sqrt L\\
 &\quad\le \sqrt{n(n-1)V_n}+C_k\sqrt n.
\end{aligned}
\]
Using the definition of $H_B$ in \eqref{eq:triangle-background},
we conclude that
\begin{equation}\label{eq:background-degree-bound}
 H_B
 \le1+\frac{\sqrt{n(n-1)V_n}}{n}+C_kn^{-1/2}
 \le1+\sqrt{V_n}+C_kn^{-1/2}.
\end{equation}
For every fixed $s\ge1$, Minkowski's and Cauchy--Schwarz inequalities,
together with \eqref{eq:background-degree-bound}, give
\[
\begin{aligned}
 \norm{H_B}_s
 &\le1+\norm{\sqrt{V_n}}_s+C_kn^{-1/2}\\
 &\le1+\norm{V_n}_s^{1/2}+C_kn^{-1/2}.
\end{aligned}
\]
Applying \eqref{eq:local-moments} with $r=s$ proves
\begin{equation}\label{eq:background-H-moments}
 \sup_{n\ge k+2}\E H_B^s\le C_{k,s}.
\end{equation}
Taking $s=2$ in \eqref{eq:background-H-moments} gives
$\E H_B^2\le C_k$.
Combining this moment bound with
\eqref{eq:cubic-product-background}, with $k=12$, yields
\begin{equation}\label{eq:cubic-disjoint-product}
 \sup_{\substack{e,f,g,e',f',g'\\
                 |e\cup f\cup g\cup e'\cup f'\cup g'|=12}}
 \bigl|\E[Y_{efg}Y_{e'f'g'}]\bigr|
 \le \frac Cn.
\end{equation}
Substituting \eqref{eq:cubic-disjoint-product} into
\eqref{eq:cubic-averaging-budget} and then using
\eqref{eq:empirical-cubic-difference}, we conclude that
\begin{equation*}
 \left\lVert\frac1N\sum_e(M_e-\meanM)^3\right\rVert_2
 \le n^{3/2}\left\{\E\left|\frac1{N^3}\sum_{e,f,g}Y_{efg}\right|^2\right\}^{1/2}
 \le n^{3/2}(C/n)^{1/2}\le Cn.
\end{equation*}
This proves the first bound in \eqref{eq:overlap-cubic-average}.

For the second bound, expand about $\meanM$ and use
$\sum_e(M_e-\meanM)=0$ to obtain
\begin{equation*}
 \frac1N\sum_e(M_e-nQ^2)^3
 =\frac1N\sum_e(M_e-\meanM)^3 +3(\meanM-nQ^2)\frac1N\sum_e(M_e-\meanM)^2
       +(\meanM-nQ^2)^3.
\end{equation*}
By \eqref{eq:triangle-first-moment} and Lemma~\ref{thm:fields},
\begin{equation*}
 \norm{\meanM-nQ^2}_r\le C_r,\qquad
 \left\lVert\frac1N\sum_e(M_e-\meanM)^2\right\rVert_4
 \le\frac1N\sum_e\norm{M_e-\meanM}_8^2\le Cn.
\end{equation*}
H\"older's inequality now yields
\begin{equation*}
\begin{aligned}
 \left\lVert\frac1N\sum_e(M_e-nQ^2)^3\right\rVert_2
 &\le Cn+3\norm{\meanM-nQ^2}_4
       \left\lVert\frac1N\sum_e(M_e-\meanM)^2\right\rVert_4
       +\norm{\meanM-nQ^2}_6^3\\
 &\le Cn+Cn+C\le Cn.
\end{aligned}
\end{equation*}
This proves the second bound in \eqref{eq:overlap-cubic-average}.
\end{proof}

\subsection{Proof of Lemma~\ref{prop:triangle-moments}}\label{sec:empirical-moments-proof}

\begin{proof}
We first prove the bound for $V_n-A_{\alpha,h}(Q)$ in
\eqref{eq:triangle-moment-replacement}.
Define
\begin{equation}\label{eq:degree-contrast-square}
 J^{(d)}_{ij}=\frac{(d_i-d_j)^2}{2n},\qquad i<j.
\end{equation}
By \eqref{eq:degree-difference},
\begin{equation}\label{eq:degree-square-expansion}
\begin{aligned}
 \E|V_n-A_{\alpha,h}(Q)|^2
 &=\E\left|\frac1N\sum_{i<j}\{J^{(d)}_{ij}-A_{\alpha,h}(Q)\}\right|^2\\
 &=\frac1{N^2}\sum_{i<j,\,i'<j'}
       \E[(J^{(d)}_{ij}-A_{\alpha,h}(Q))(J^{(d)}_{i'j'}-A_{\alpha,h}(Q))].
\end{aligned}
\end{equation}
Lemma~\ref{thm:fields} and boundedness of $A_{\alpha,h}$ give
\begin{equation*}
 \norm{J^{(d)}_{ij}-A_{\alpha,h}(Q)}_2
 \le\frac1{2n}\norm{d_i-d_j}_4^2+\norm{A_{\alpha,h}}_\infty\le C.
\end{equation*}
There are at most $Cn^3$ choices with a repeated vertex among
$i,j,i',j'$. Their contribution to the squared average is at most
$Cn^3/N^2\le C/n$.
It remains to bound the products with four distinct vertices.

Fix such a choice and condition on $\calF_{\mathcal I}$, with
$\mathcal I=\{i,j,i',j'\}$.
We first replace $A_{\alpha,h}(Q)$ by $A_{\alpha,h}(q_B)$, which is fixed under this conditioning.
The product difference is
\begin{equation*}
\begin{aligned}
 &(J^{(d)}_{ij}-A_{\alpha,h}(Q))(J^{(d)}_{i'j'}-A_{\alpha,h}(Q))
       -(J^{(d)}_{ij}-A_{\alpha,h}(q_B))(J^{(d)}_{i'j'}-A_{\alpha,h}(q_B))\\
 &\qquad=(A_{\alpha,h}(q_B)-A_{\alpha,h}(Q))
              \{J^{(d)}_{ij}+J^{(d)}_{i'j'}-A_{\alpha,h}(Q)-A_{\alpha,h}(q_B)\}.
\end{aligned}
\end{equation*}
Since $A_{\alpha,h}$ is bounded and Lipschitz, \eqref{eq:background-density} shows
that the change in the product expectation is at most
\begin{equation*}
 \frac Cn\E(|J^{(d)}_{ij}|+|J^{(d)}_{i'j'}|+1)\le\frac Cn.
\end{equation*}
For the product centered at $A_{\alpha,h}(q_B)$, conditioning gives
\begin{equation}\label{eq:degree-conditional-product}
\begin{aligned}
 &\E[(J^{(d)}_{ij}-A_{\alpha,h}(q_B))(J^{(d)}_{i'j'}-A_{\alpha,h}(q_B))]\\
 &\quad=\E\left\{\cov_1(J^{(d)}_{ij},J^{(d)}_{i'j'})
       +(\E_1J^{(d)}_{ij}-A_{\alpha,h}(q_B))(\E_1J^{(d)}_{i'j'}-A_{\alpha,h}(q_B))\right\}.
\end{aligned}
\end{equation}
The next two lemmas give the required one-row variance and conditional mean estimates.
We use Lemma~\ref{lem:separated} for the covariance bound.
\begin{lemma}\label{lem:row-coarse}
Let $\mathbf z_1$ have the one-row law $p_B$, and let $\mathbf z_2$ be an
independent copy. Put $\boldsymbol\mu=\E_B\mathbf z_1$ and $\Sigma=\cov_B(\mathbf z_1)$.
For any fixed vertex set $\mathcal I\subset[n]$ of size $k$, with $L=n-k\ge2$, the following estimates hold uniformly over all adjacency matrices $B$ of graphs on $\mathcal I^c$:
\begin{equation}
 \norm{\boldsymbol\mu-\varrho_{\alpha,h}(q_B)\mathbf1}_{\ell^2}\le CH_B,\qquad
 \norm{\Sigma}_{\rm op}\le C,\qquad
 |\Sigma_{uv}|\le C/n\quad(u\ne v),\label{eq:row-coarse-a}
\end{equation}
The row-sum and overlap variances satisfy
\begin{equation}\label{eq:row-refined}
 |\var_B(\mathbf1^\top\mathbf z_1)-nA_{\alpha,h}(q_B)|+
 |\var_B(\mathbf z_1^\top\mathbf z_2)-nK_{\alpha,h}(q_B)|\le CH_B^2.
\end{equation}
In particular, since $1\le H_B\le C\sqrt n$,
\begin{align}
 \var_B(\mathbf1^\top \mathbf z_1)&=nA_{\alpha,h}(q_B)+O(\sqrt nH_B),\label{eq:row-coarse-b}\\
 \var_B(\mathbf z_1^\top \mathbf z_2)&=nK_{\alpha,h}(q_B)+O\{\sqrt n(H_B+H_B^2)\}.
 \label{eq:row-coarse-c}
\end{align}
The absolute row sums of $\Sigma$ are bounded uniformly.
\end{lemma}
The proof is given in Section~\ref{sec:row-response}.

\begin{lemma}\label{lem:actual-refined}
Let $\mathcal I\subset[n]$ be a fixed vertex set, and let $\calF_{\mathcal I}$ fix all edges within $\mathcal I$ and within $\mathcal I^c$.
In the first estimate, take $i,j\in\mathcal I$ with $i<j$.
In the second, $e,f$ have four distinct endpoints in $\mathcal I$.
Then
\begin{align}
 \left|\E\left(\frac{(d_i-d_j)^2}{2n}\mid\calF_{\mathcal I}\right)
                         -A_{\alpha,h}(q_B)\right|&\le CH_B^2/n,\label{eq:actual-refined-degree}\\
 \left|\E\left(\frac{(M_e-M_f)^2}{2n}\mid\calF_{\mathcal I}\right)
                         -K_{\alpha,h}(q_B)\right|&\le CH_B^2/n.\label{eq:actual-refined-overlap}
\end{align}
\end{lemma}
The proof is given in Section~\ref{sec:actual-refined-proof}.
By \eqref{eq:actual-refined-degree} in Lemma~\ref{lem:actual-refined},
with $|\mathcal I|=4$, and $H_B\le C\sqrt n$, we have
\begin{equation*}
 |\E_1J^{(d)}_{ij}-A_{\alpha,h}(q_B)|
 \le CH_B^2/n\le CH_Bn^{-1/2}.
\end{equation*}
The same mean estimate also holds for $J^{(d)}_{i'j'}$.
Using \eqref{eq:degree-contrast-square} and
\eqref{eq:interpolated-contrast-moments} with $U=d_i-d_j$, we obtain
\begin{equation*}
 \sup_{0\le t\le1}\norm{J^{(d)}_{ij}}_{L^4(p_{B,t})}
 \le\frac1{2n}\sup_{0\le t\le1}
                 \norm{d_i-d_j}_{L^8(p_{B,t})}^2\le C.
\end{equation*}
The same bound holds for $J^{(d)}_{i'j'}$, so their centered fourth moments
are bounded.
Since variables $J^{(d)}_{ij}$ and $J^{(d)}_{i'j'}$ depend only on $(\mathbf z_i,\mathbf z_j)$ and $(\mathbf z_{i'},\mathbf z_{j'})$, respectively,
Lemma~\ref{lem:separated} therefore gives
\begin{equation*}
 |\cov_1(J^{(d)}_{ij},J^{(d)}_{i'j'})|\le Cn^{-1}.
\end{equation*}
Substituting into \eqref{eq:degree-conditional-product} and combining \eqref{eq:background-density}, we obtain
\begin{equation}\label{eq:empirical-square-covariance}
\begin{aligned}
 &|\E[(J^{(d)}_{ij}-A_{\alpha,h}(Q))(J^{(d)}_{i'j'}-A_{\alpha,h}(Q))]|\\
 &\le\E|\cov_1(J^{(d)}_{ij},J^{(d)}_{i'j'})|
 +\E|\{\E_1J^{(d)}_{ij}-A_{\alpha,h}(q_B)\}\{\E_1J^{(d)}_{i'j'}-A_{\alpha,h}(q_B)\}|+\frac Cn\\
 &\le\frac Cn(1+\E H_B^2)\le\frac Cn,
\end{aligned}
\end{equation}
where the last inequality uses \eqref{eq:background-H-moments}.
Then by \eqref{eq:degree-square-expansion}, we obtain
\begin{equation*}
 \E|V_n-A_{\alpha,h}(Q)|^2\le\frac{Cn^3}{N^2}+\frac Cn\le\frac Cn.
\end{equation*}
Thus
\begin{equation}\label{eq:degree-empirical-L2-bound}
 \norm{V_n-A_{\alpha,h}(Q)}_2\le Cn^{-1/2}.
\end{equation}

For the bound on $K_n$, centering at $\meanM$  and using \eqref{0201} gives
\begin{equation}\label{eq:overlap-difference}
\begin{aligned}
 K_n&=\frac1{nN}\sum_e(M_e-\meanM)^2
             +\frac{(\meanM-nQ^2)^2}{n}\\
    &=\frac1{N^2}\sum_{e,f}\frac{(M_e-M_f)^2}{2n}
             +\frac{(V_n-Q^2-Q)^2}{n}.
\end{aligned}
\end{equation}
By Lemma~\ref{thm:fields},
\begin{equation}\label{eq:overlap-extra-square-bound}
 \left\lVert\frac{(V_n-Q^2-Q)^2}{n}\right\rVert_2
 \le\frac1n(\norm{V_n}_4+2)^2\le\frac Cn.
\end{equation}

It remains to control the average in \eqref{eq:overlap-difference}.
Define
\begin{equation}\label{eq:overlap-contrast-square}
 J^{(M)}_{ef}=\frac{(M_e-M_f)^2}{2n}.
\end{equation}
Then
\begin{equation}\label{eq:overlap-square-expansion}
\begin{aligned}
 &\E\left|\frac1{N^2}\sum_{e,f}\{J^{(M)}_{ef}-K_{\alpha,h}(Q)\}\right|^2\\
 &\quad=\frac1{N^4}\sum_{e,f,e',f'}
           \E[(J^{(M)}_{ef}-K_{\alpha,h}(Q))(J^{(M)}_{e'f'}-K_{\alpha,h}(Q))].
\end{aligned}
\end{equation}
Lemma~\ref{thm:fields} and boundedness of $K_{\alpha,h}$ imply
\begin{equation*}
 \norm{J^{(M)}_{ef}-K_{\alpha,h}(Q)}_2
 \le\frac1{2n}\norm{M_e-M_f}_4^2+\norm{K_{\alpha,h}}_\infty\le C.
\end{equation*}
Among these $N^4$ products, at most $Cn^7$ choices have a repeated
endpoint. By the Cauchy--Schwarz inequality, their total contribution is bounded by
\begin{equation}\label{eq:overlap-repeated-endpoints}
\begin{aligned}
 &\frac1{N^4}\sum_{\substack{e,f,e',f'\\
                 |e\cup f\cup e'\cup f'|<8}}
   \bigl|\E[(J^{(M)}_{ef}-K_{\alpha,h}(Q))
              (J^{(M)}_{e' f'}-K_{\alpha,h}(Q))]\bigr|\\
 &\qquad\le\frac{Cn^7}{N^4}\le\frac Cn.
\end{aligned}
\end{equation}

Fix four edges with eight distinct endpoints, and condition on
$\calF_{\mathcal I}$ with $\mathcal I$ equal to this endpoint set.
All distributions and background quantities below refer to this
same conditioning.
By \eqref{eq:actual-refined-overlap} in Lemma~\ref{lem:actual-refined},
with $|\mathcal I|=8$, and $H_B\le C\sqrt n$, we have
\begin{equation*}
 |\E_1J^{(M)}_{ef}-K_{\alpha,h}(q_B)|
 \le CH_B^2/n\le CH_Bn^{-1/2}
 \le C(H_B+H_B^2)n^{-1/2}.
\end{equation*}
The same mean estimate holds for $J^{(M)}_{e'f'}$.
Using \eqref{eq:overlap-contrast-square} and \eqref{eq:interpolated-contrast-moments} with $U=M_e-M_f$ also gives
\begin{equation*}
 \sup_{0\le t\le1}\norm{J^{(M)}_{ef}}_{L^4(p_{B,t})}
 \le\frac1{2n}\sup_{0\le t\le1}
                 \norm{M_e-M_f}_{L^8(p_{B,t})}^2\le C.
\end{equation*}
The same bounds hold for $J^{(M)}_{e'f'}$.
Applying Lemma~\ref{lem:separated} as in the estimate for $V_n$ yields
\begin{equation*}
 |\cov_1(J^{(M)}_{ef},J^{(M)}_{e'f'})|\le Cn^{-1}.
\end{equation*}
Since $K_{\alpha,h}$ is bounded and Lipschitz,
$|K_{\alpha,h}(Q)-K_{\alpha,h}(q_B)|\le C/n$ by \eqref{eq:background-density}.
By a similar decomposition as in \eqref{eq:empirical-square-covariance}, it follows that
\begin{equation}\label{eq:overlap-distinct-endpoints}
 |\E[(J^{(M)}_{ef}-K_{\alpha,h}(Q))(J^{(M)}_{e'f'}-K_{\alpha,h}(Q))]|
 \le\frac Cn\{1+\E(H_B+H_B^2)^2\}\le\frac Cn.
\end{equation}
Substituting \eqref{eq:overlap-repeated-endpoints} and
\eqref{eq:overlap-distinct-endpoints} into
\eqref{eq:overlap-square-expansion}, we obtain
\begin{equation*}
 \E\left|\frac1{N^2}\sum_{e,f}\{J^{(M)}_{ef}-K_{\alpha,h}(Q)\}\right|^2
 \le\frac{Cn^7}{N^4}+\frac Cn\le\frac Cn.
\end{equation*}
Combining this estimate with \eqref{eq:overlap-difference} and \eqref{eq:overlap-extra-square-bound}, we conclude that
\begin{equation*}
 \norm{K_n-K_{\alpha,h}(Q)}_2\le Cn^{-1/2}+Cn^{-1}\le Cn^{-1/2}.
\end{equation*}
Together with \eqref{eq:degree-empirical-L2-bound}, this proves
\eqref{eq:triangle-moment-replacement}.
\end{proof}

\subsection{Proof of Lemma~\ref{thm:tested-triangle}}\label{sec:tested-moments-proof}

\begin{proof}
Put $\psi_n(q)=\omega(q)f(\sqrt n(q-2/3))$.
We first bound the expectation involving $V_n$ in
\eqref{eq:tested-moments}.
By \eqref{eq:degree-contrast-square} and \eqref{eq:degree-difference},
\begin{equation*}
 \left|\E[(V_n-A_{\alpha,h}(Q))\psi_n(Q)]\right|
 \le\frac1N\sum_{i<j}\left|\E[(J^{(d)}_{ij}-A_{\alpha,h}(Q))\psi_n(Q)]\right|.
\end{equation*}
Fix $i<j$ and condition on $\calF_{\mathcal I}$, with
$\mathcal I=\{i,j\}$. Put $q_{\mathcal I}=\E(Q\mid\calF_{\mathcal I})$.
Both $q_{\mathcal I}$ and $q_B$ are fixed under this conditioning.
Adding and subtracting $A_{\alpha,h}(q_B)$ and $\psi_n(q_{\mathcal I})$ gives
\begin{equation}\label{eq:tested-decomposition}
\begin{aligned}
 \E[(J^{(d)}_{ij}-A_{\alpha,h}(Q))\psi_n(Q)]
 ={}&\E[(J^{(d)}_{ij}-A_{\alpha,h}(q_B))\{\psi_n(Q)-\psi_n(q_{\mathcal I})\}]\\
 &+\E[\{\E(J^{(d)}_{ij}\mid\calF_{\mathcal I})-A_{\alpha,h}(q_B)\}
                          \psi_n(q_{\mathcal I})]\\
 &+\E[(A_{\alpha,h}(q_B)-A_{\alpha,h}(Q))\psi_n(Q)].
\end{aligned}
\end{equation}
For the first term, the conditional Cauchy--Schwarz inequality gives
\begin{equation}\label{eq:tested-conditional-cs}
\begin{aligned}
 &\left|\E[(J^{(d)}_{ij}-A_{\alpha,h}(q_B))\{\psi_n(Q)-\psi_n(q_{\mathcal I})\}]\right|\\
 &\quad\le\E\left[
    \norm{J^{(d)}_{ij}-A_{\alpha,h}(q_B)}_{L^2(\PP_n(\cdot\mid\calF_{\mathcal I}))}
    \norm{\psi_n(Q)-\psi_n(q_{\mathcal I})}_{L^2(\PP_n(\cdot\mid\calF_{\mathcal I}))}
             \right].
\end{aligned}
\end{equation}
The bound \eqref{eq:interpolated-contrast-moments} with $U=d_i-d_j$, $t=1$ and $r=4$, and boundedness of $A_{\alpha,h}$ imply
\begin{equation}\label{eq:tested-degree-factor}
 \norm{J^{(d)}_{ij}-A_{\alpha,h}(q_B)}_{L^2(\PP_n(\cdot\mid\calF_{\mathcal I}))}
 \le\frac1{2n}\norm{d_i-d_j}_{L^4(\PP_n(\cdot\mid\calF_{\mathcal I}))}^2
           +\norm{A_{\alpha,h}}_\infty\le C.
\end{equation}

To bound the factor involving $\psi_n$, we estimate
$\var(Q\mid\calF_{\mathcal I})$ for any fixed vertex set $\mathcal I\subset[n]$.
Write $k=|\mathcal I|$ and $L=n-k$. There are $kL$ free coordinates,
with $\delta_{(a,u)}Q=N^{-1}$ for each one.
By \eqref{eq:comparison-cov} and \eqref{eq:row-influence},
\begin{equation}\label{eq:conditional-Q-variance}
 \var(Q\mid\calF_{\mathcal I})
 \le\tfrac14(\boldsymbol\delta Q)^\top(I-\mathcal C)^{-1}\boldsymbol\delta Q \le\frac{\norm{\boldsymbol\delta Q}_{\ell^2}^2}{4-|\alpha|}
   =\frac{kL}{(4-|\alpha|)N^2}\le C_kn^{-3}.
\end{equation}
For $\psi_n$,
\begin{equation*}
\begin{aligned}
 \norm{\psi_n}_\infty&\le\norm{\omega}_\infty\norm f_\infty,\qquad \psi_n'(q)=\omega'(q)f(\sqrt n(q-2/3))
       +\sqrt n\,\omega(q)f'(\sqrt n(q-2/3)),\\
 \norm{\psi_n'}_\infty
 &\le\norm{\omega'}_\infty\norm f_\infty
        +\sqrt n\,\norm{\omega}_\infty\norm{f'}_\infty
  \le C\sqrt n\,\norm{\omega}_{C^1}\norm f_{C_b^1}.
\end{aligned}
\end{equation*}
Since $q_{\mathcal I}$ is the conditional mean of $Q$,
\begin{equation}\label{eq:slow-test-oscillation}
\begin{aligned}
 &\norm{\psi_n(Q)-\psi_n(q_{\mathcal I})}_{L^2(\PP_n(\cdot\mid\calF_{\mathcal I}))} \le\norm{\psi_n'}_\infty
       \left\{\E[(Q-q_{\mathcal I})^2\mid\calF_{\mathcal I}]\right\}^{1/2}\\
 &\quad=\norm{\psi_n'}_\infty\sqrt{\var(Q\mid\calF_{\mathcal I})}
  \le\frac{C_k}{n}\norm{\omega}_{C^1}\norm f_{C_b^1}.
\end{aligned}
\end{equation}
Substituting \eqref{eq:tested-degree-factor} and \eqref{eq:slow-test-oscillation} into \eqref{eq:tested-conditional-cs} gives
\begin{equation*}
 \left|\E[(J^{(d)}_{ij}-A_{\alpha,h}(q_B))\{\psi_n(Q)-\psi_n(q_{\mathcal I})\}]\right|
 \le\frac Cn\norm{\omega}_{C^1}\norm f_{C_b^1}.
\end{equation*}

For the second term in \eqref{eq:tested-decomposition}, we need the
conditional mean of $J^{(d)}_{ij}$.
By \eqref{eq:actual-refined-degree} in Lemma~\ref{lem:actual-refined} and \eqref{eq:background-H-moments},
\begin{equation*}
\begin{aligned}
 &\left|\E[\{\E(J^{(d)}_{ij}\mid\calF_{\mathcal I})-A_{\alpha,h}(q_B)\}
 \psi_n(q_{\mathcal I})]\right|\\
 &\quad\le\norm{\psi_n}_\infty
       \E|\E(J^{(d)}_{ij}\mid\calF_{\mathcal I})-A_{\alpha,h}(q_B)| \le\frac Cn\norm{\omega}_\infty\norm f_\infty\E H_B^2
  \le\frac Cn\norm{\omega}_{C^1}\norm f_{C_b^1}.
\end{aligned}
\end{equation*}

For the third term in \eqref{eq:tested-decomposition}, the Lipschitz bound on $A_{\alpha,h}$ and \eqref{eq:background-density} give
\begin{equation*}
 \left|\E[(A_{\alpha,h}(q_B)-A_{\alpha,h}(Q))\psi_n(Q)]\right|
 \le\norm{A'_{\alpha,h}}_\infty\E|q_B-Q|\norm{\psi_n}_\infty
 \le\frac Cn\norm{\omega}_{C^1}\norm f_{C_b^1}.
\end{equation*}
Substituting the three estimates into \eqref{eq:tested-decomposition}
and averaging over $i<j$, we obtain
\begin{equation}\label{eq:tested-degree-bound}
 \left|\E[(V_n-A_{\alpha,h}(Q))\psi_n(Q)]\right|
 \le\frac Cn\norm{\omega}_{C^1}\norm f_{C_b^1}.
\end{equation}
This proves the first bound in \eqref{eq:tested-moments}.

For the expectation involving $K_n$, \eqref{eq:overlap-contrast-square}
and \eqref{eq:overlap-difference} give
\begin{equation*}
\begin{aligned}
 \left|\E[(K_n-K_{\alpha,h}(Q))\psi_n(Q)]\right|
 &\le\frac1{N^2}\sum_{e,f}
       \left|\E[(J^{(M)}_{ef}-K_{\alpha,h}(Q))\psi_n(Q)]\right|\\
 &\quad+\left|\E\left[\frac{(V_n-Q^2-Q)^2}{n}\psi_n(Q)\right]\right|.
\end{aligned}
\end{equation*}
The bound \eqref{eq:overlap-extra-square-bound} controls the last term:
\begin{equation}\label{eq:tested-overlap-extra-square}
 \norm{\psi_n}_\infty
 \left\lVert\frac{(V_n-Q^2-Q)^2}{n}\right\rVert_2
 \le\frac Cn\norm{\omega}_{C^1}\norm f_{C_b^1}.
\end{equation}
For the remaining average, Lemma~\ref{thm:fields} gives
$\norm{J^{(M)}_{ef}}_1=\norm{M_e-M_f}_2^2/(2n)\le C$ for all $e,f$.
Among the $N^2$ ordered pairs, at most $Cn^3$ have a repeated endpoint.
Since $K_{\alpha,h}$ is bounded, their contribution is at most
\begin{equation}\label{eq:tested-overlap-repeated-endpoints}
 \frac{Cn^3}{N^2}\sup_{e,f}\E|J^{(M)}_{ef}-K_{\alpha,h}(Q)|\norm{\psi_n}_\infty
 \le\frac Cn\norm{\omega}_{C^1}\norm f_{C_b^1}.
\end{equation}

Fix $e,f$ with four distinct endpoints, and take $\mathcal I$ to be
this endpoint set. Put $q_{\mathcal I}=\E(Q\mid\calF_{\mathcal I})$
for this conditioning.
The bound \eqref{eq:interpolated-contrast-moments} with $U=M_e-M_f$, $t=1$ and $r=4$, and boundedness of $K_{\alpha,h}$ give
\begin{equation*}
 \norm{J^{(M)}_{ef}-K_{\alpha,h}(q_B)}_{L^2(\PP_n(\cdot\mid\calF_{\mathcal I}))}
 \le\frac1{2n}\norm{M_e-M_f}_{L^4(\PP_n(\cdot\mid\calF_{\mathcal I}))}^2
          +\norm{K_{\alpha,h}}_\infty\le C.
\end{equation*}
We use the same decomposition as in \eqref{eq:tested-decomposition}.
Equations \eqref{eq:slow-test-oscillation},
\eqref{eq:actual-refined-overlap} and \eqref{eq:background-density}
then give
\begin{equation}\label{eq:tested-overlap-distinct-endpoints}
\begin{aligned}
 \left|\E[(J^{(M)}_{ef}-K_{\alpha,h}(Q))\psi_n(Q)]\right|
 &\le\frac Cn\norm{\omega}_{C^1}\norm f_{C_b^1}
        +\frac Cn\norm{\psi_n}_\infty\E H_B^2
        +\frac Cn\norm{\psi_n}_\infty\\
 &\le\frac Cn\norm{\omega}_{C^1}\norm f_{C_b^1}.
\end{aligned}
\end{equation}
Combining \eqref{eq:tested-overlap-extra-square},
\eqref{eq:tested-overlap-repeated-endpoints}, and
\eqref{eq:tested-overlap-distinct-endpoints}, we obtain
\begin{equation*}
 \left|\E[(K_n-K_{\alpha,h}(Q))\psi_n(Q)]\right|
 \le\frac Cn\norm{\omega}_{C^1}\norm f_{C_b^1}.
\end{equation*}
Together with \eqref{eq:tested-degree-bound}, this proves \eqref{eq:tested-moments}.
\end{proof}

\subsection{Proof of Lemma~\ref{lem:rooted}}\label{sec:rooted-proof}

\begin{proof}
We first prove the pairwise bound in \eqref{eq:rooted-difference}.
If $v_H=2$, then $U_e^{H,a}=t_{H-a}=1$, and both conclusions hold.
Assume $v_H\ge3$.
It is immediate when $e=f$. Otherwise, let
$\mathcal I$ be the union of the endpoint sets of $e$ and $f$, and condition on
$\calF_{\mathcal I}$, which fixes all edges within $\mathcal I$ and within $\mathcal I^c$.
Only the edges joining the two sets remain random.
Write $V(H)\setminus a$ for the set of vertices of $H$ other than
the two endpoints of $a$.
For $g\in\{e,f\}$, restrict the sum in \eqref{eq:rooted-density} to injections
that map every vertex in $V(H)\setminus a$ into $\mathcal I^c$:
\begin{equation}\label{eq:rooted-restricted-count}
 \widetilde U_g
 =\frac1{2(n-2)_{v_H-2}}
   \sum_{\substack{\phi:V(H)\hookrightarrow[n]\\
         \phi(a)=g,\ \phi(V(H)\setminus a)\subset\mathcal I^c}}
       \prod_{b\in E(H)\setminus\{a\}}X_{\phi(b)}.
\end{equation}
Here $\phi(b)=\{\phi(r),\phi(s)\}$ for $b=\{r,s\}$.
Each omitted embedding sends at least one of the $v_H-2$ vertices in
$V(H)\setminus a$ into $\mathcal I$.
There are two assignments of the endpoints of $a$ to those of $g$,
at most $|\mathcal I|$ choices for the image of that vertex, and at most
$(n-3)_{v_H-3}$ injections of the remaining vertices.
Since every summand lies in $[0,1]$, a union bound gives
\begin{equation*}
 0\le U_g^{H,a}-\widetilde U_g
 \le\frac{2(v_H-2)|\mathcal I|(n-3)_{v_H-3}}
              {2(n-2)_{v_H-2}}
 \le\frac{C_H}{n}.
\end{equation*}
Put $F=\widetilde U_e-\widetilde U_f$. Under the conditional law
$p_{B,1}$ defined by \eqref{eq:interpolated-row-law} with $t=1$,
the identity
\begin{equation*}
 U_e^{H,a}-U_f^{H,a}
 =F+(U_e^{H,a}-\widetilde U_e)-(U_f^{H,a}-\widetilde U_f)
\end{equation*}
yields
\begin{equation}\label{eq:rooted-conditional-decomposition}
 \norm{U_e^{H,a}-U_f^{H,a}}_{L^r(\PP_n(\cdot\mid\calF_{\mathcal I}))}
 \le \norm{F-\E_1F}_{L^r(p_{B,1})}+|\E_1F|+\frac{2C_H}{n}.
\end{equation}
To bound $\norm{F-\E_1F}_{L^r(p_{B,1})}$, we estimate the coordinate oscillations of $F$.
By \eqref{eq:rooted-restricted-count}, a free coordinate $z_{bu}$ can affect
$\widetilde U_g$ only if $b$ is an endpoint of $g$ and a vertex in
$V(H)\setminus a$ maps to $u$. There are at most $v_H-2$ choices for that vertex,
two assignments of the endpoints of $a$ to those of $g$,
and $(n-3)_{v_H-3}$ injections of the remaining vertices.
Consequently,
\begin{equation*}
\begin{aligned}
 \delta_{(b,u)}\widetilde U_g
 &\le\frac{2(v_H-2)(n-3)_{v_H-3}}{2(n-2)_{v_H-2}}
 =\frac{v_H-2}{n-2},\\
 \delta_{(b,u)}F
 &\le\delta_{(b,u)}\widetilde U_e+\delta_{(b,u)}\widetilde U_f
 \le\frac{2(v_H-2)}{n-2}\le\frac{C_H}{n}.
\end{aligned}
\end{equation*}
The influence bound \eqref{eq:row-influence} holds for every $p_{B,t}$,
$0\le t\le1$. Since there are at most $4n$ free
coordinates, \eqref{eq:comparison-moment} gives
\begin{equation*}
 \norm{\boldsymbol\delta F}_{\ell^2}^2\le4n(C_H/n)^2\le C_H/n,\qquad
 \norm{F-\E_tF}_{L^r(p_{B,t})}\le C_{H,r}n^{-1/2},
 \quad 0\le t\le1.
\end{equation*}

To bound $|\E_1F|$, first work under the product law $p_{B,0}$
in \eqref{eq:one-row-law}. Under this law, the vectors
$\mathbf z_b=(z_{bu})_{u\in\mathcal I^c}$, $b\in\mathcal I$,
are independent and each has law $p_B$.
For $g=\{i,j\}$, \eqref{eq:rooted-restricted-count} shows that
$\widetilde U_g$ depends only on $B,\mathbf z_i,\mathbf z_j$.
With $B$ fixed, the same counting formula applies to both $g=e$ and $g=f$.
If $e=\{i,j\}$ and $f=\{i',j'\}$ have four distinct endpoints, then
$(\mathbf z_i,\mathbf z_j)$ and $(\mathbf z_{i'},\mathbf z_{j'})$ both have law
$p_B^{\otimes2}$, so $\E_0\widetilde U_e=\E_0\widetilde U_f$.
If instead $e=\{i,j\}$ and $f=\{i,j'\}$ share one endpoint, then, conditional on
$\mathbf z_i$, both $\mathbf z_j$ and $\mathbf z_{j'}$ have law $p_B$.
Hence
\begin{equation*}
 \E_0[\widetilde U_e\mid\mathbf z_i]
 =\E_0[\widetilde U_f\mid\mathbf z_i].
\end{equation*}
Taking expectations gives $\E_0\widetilde U_e=\E_0\widetilde U_f$ also
in this case. Therefore $\E_0F=0$.
By \eqref{eq:row-interpolation}, the Cauchy--Schwarz inequality, and
\eqref{eq:cross-moments},
\begin{equation*}
 |\E_1F|
 =\left|\int_0^1\cov_t(F,J_{\rm cross})\dd t\right|
 \le\int_0^1
 \norm{F-\E_tF}_{L^2(p_{B,t})}
 \norm{J_{\rm cross}-\E_tJ_{\rm cross}}_{L^2(p_{B,t})}\dd t
 \le C_H/n.
\end{equation*}
Substituting the centered-moment and mean bounds into
\eqref{eq:rooted-conditional-decomposition} yields
\begin{equation*}
 \norm{U_e^{H,a}-U_f^{H,a}}_{L^r(\PP_n(\cdot\mid\calF_{\mathcal I}))}
 \le C_{H,r}n^{-1/2}+C_Hn^{-1}\le C_{H,r}n^{-1/2}.
\end{equation*}
The bound is uniform in the conditioned edges. By the tower property,
\begin{equation*}
 \E|U_e^{H,a}-U_f^{H,a}|^r
 =\E\bigl[\E(|U_e^{H,a}-U_f^{H,a}|^r\mid\calF_{\mathcal I})\bigr]
 \le C_{H,r}^r n^{-r/2}.
\end{equation*}
By \eqref{eq:rooted-identities} and Minkowski's inequality,
\begin{equation*}
 \norm{U_e^{H,a}-t_{H-a}}_r
 =\left\lVert\frac1N\sum_f(U_e^{H,a}-U_f^{H,a})\right\rVert_r
 \le\frac1N\sum_f\norm{U_e^{H,a}-U_f^{H,a}}_r
 \le C_{H,r}n^{-1/2}.
\end{equation*}
This completes \eqref{eq:rooted-difference}.

For \eqref{eq:rooted-product}, apply Minkowski's and H\"older's
inequalities, by \eqref{eq:field-moments} and
\eqref{eq:rooted-difference}, we obtain
\begin{equation*}
\begin{aligned}
 \left\lVert\frac1N\sum_e(L_e-\bar L)(U_e^{H,a}-t_{H-a})\right\rVert_r
 &\le\frac1N\sum_e
       \norm{L_e-\bar L}_{2r}\norm{U_e^{H,a}-t_{H-a}}_{2r}\\
 &\le\frac1N\sum_e C_{H,r}n^{-1/2}n^{-1/2}
 \le C_{H,r}n^{-1},
\end{aligned}
\end{equation*}
which proves \eqref{eq:rooted-product}.
\end{proof}

\subsection{Proof of Lemma~\ref{lem:overlap-cubic}}\label{sec:overlap-proof}

\begin{proof}
Put $\boldsymbol\mu=\E_B\mathbf z_1$, $\Sigma=\cov_B(\mathbf z_1)$ and
$\boldsymbol\xi=\mathbf z_1-\boldsymbol\mu$.
We first prove the following bound for three centered linear sums.

\begin{lemma}\label{lem:row-refined-third}
For fixed $K_0\ge0$, let $\mathbf a,\mathbf b,\mathbf c\in\R^L$ be deterministic
vectors with Euclidean norms at most one and all coordinates bounded
in absolute value by $K_0/\sqrt n$. Then
\begin{equation}\label{eq:dense-third}
 |\E_B(\mathbf a^\top\boldsymbol\xi)(\mathbf b^\top\boldsymbol\xi)
       (\mathbf c^\top\boldsymbol\xi)|\le C_{K_0}n^{-1/2}.
\end{equation}
In particular, three centered linear sums with uniformly bounded deterministic
weights have a mixed third moment of order $O(n)$.
\end{lemma}

\begin{proof}[{Proof of Lemma \ref{lem:row-refined-third}}]
The one-row law \eqref{eq:one-row-law} gives, for
$\mathbf x=(x_v)_{v\in\mathcal I^c}\in\{0,1\}^{\mathcal I^c}$,
\begin{equation*}
 p_u(\mathbf x)
 =\PP_{p_B}(z_{1u}=1\mid\mathbf z_{1,-u}=\mathbf x_{-u})
 =\ell(h+(J_B\mathbf x)_u).
\end{equation*}
Since $\sup\ell'=1/4$, the influence matrix
$\mathcal C_{uv}=|(J_B)_{uv}|/4$ is symmetric and satisfies
\begin{equation*}
 \norm{\mathcal C}_{\rm op}
 \le\max_u\sum_v\mathcal C_{uv}
 \le\frac{|\alpha|(L-1)}{4n}\le\frac{|\alpha|}{4}<1.
\end{equation*}
Put $\Sigma=\cov_B(\mathbf z_1)$. For any unit vector $\mathbf a$,
the coordinate oscillations of $\mathbf a^\top\mathbf z_1$ equal $|\mathbf a|$,
so \eqref{eq:comparison-cov} gives
\begin{equation*}
 \mathbf a^\top\Sigma\mathbf a=\var_B(\mathbf a^\top\mathbf z_1)
 \le\tfrac14|\mathbf a|^\top(I-\mathcal C)^{-1}|\mathbf a|
 \le\tfrac14\norm{(I-\mathcal C)^{-1}}_{\rm op}.
\end{equation*}
Taking the supremum over unit vectors and using
$\norm{\mathcal C}_{\rm op}\le|\alpha|/4<1$, we obtain
\begin{equation*}
 \norm\Sigma_{\rm op}
 \le\tfrac14\norm{(I-\mathcal C)^{-1}}_{\rm op}
 \le\frac1{4(1-|\alpha|/4)}.
\end{equation*}

Put
\begin{equation}\label{eq:row-fields}
 m_u=h+(J_B\boldsymbol\mu)_u,\qquad H_u=(J_B\boldsymbol\xi)_u.
\end{equation}
Since $\delta_vH_u=|(J_B)_{uv}|$, \eqref{eq:comparison-moment} gives,
for each fixed $s\ge1$,
\begin{equation}\label{eq:row-field-moment}
 \norm{H_u}_{L^s(p_B)}
 \le C_s\left\{\sum_v(J_B)_{uv}^2\right\}^{1/2}
 \le C_sn^{-1/2},\qquad \E_BH_u=0.
\end{equation}
Since $\mu_u=\E_B\ell(m_u+H_u)$, Taylor's formula yields
\begin{equation}\label{eq:row-mean-equation}
 |\mu_u-\ell(m_u)|\le C\E_BH_u^2\le C/n,\qquad
 \norm{(\mu_u-\ell(m_u))_u}_{\ell^2}\le C\sqrt L/n.
\end{equation}
The symmetry and row sums of $J_B$ imply
\begin{equation}\label{eq:row-interaction-norms}
 \norm{J_B}_{\rm op}
 \le\max_u\sum_v|(J_B)_{uv}|
 \le |\alpha|(L-1)/n\le|\alpha|.
\end{equation}

For $r=1,2$, define
\begin{equation}\label{eq:row-linearization}
 \rho_u(\mathbf z_r)
 =p_u(\mathbf z_r)-\mu_u
  -\ell'(m_u)[J_B(\mathbf z_r-\boldsymbol\mu)]_u.
\end{equation}
Since $\E_Bp_u(\mathbf z_r)=\mu_u$ and
$\E_B(\mathbf z_r-\boldsymbol\mu)=\mathbf0$, we have
$\E_B\rho_u(\mathbf z_r)=0$.
For $r=1$, $p_u(\mathbf z_1)=\ell(m_u+H_u)$, so Taylor's formula gives
\begin{equation*}
 |\rho_u(\mathbf z_1)-\{\ell(m_u)-\mu_u\}|
 =|\ell(m_u+H_u)-\ell(m_u)-\ell'(m_u)H_u|\le CH_u^2.
\end{equation*}
Applying \eqref{eq:row-field-moment} and \eqref{eq:row-mean-equation}, and
using that $\mathbf z_1,\mathbf z_2$ have the same law, yields, for every fixed $s\ge1$,
\begin{equation}\label{eq:row-linearization-remainder}
 \norm{\rho_u(\mathbf z_r)}_{L^s(p_B)}
 \le C\norm{H_u}_{L^{2s}(p_B)}^2+|\ell(m_u)-\mu_u|
 \le C_sn^{-1},\qquad r=1,2.
\end{equation}
Let $D_\ell=\diag\{\ell'(m_u)\}_{u\in\mathcal I^c}$ and
$\boldsymbol\rho(\mathbf z_r)=(\rho_u(\mathbf z_r))_{u\in\mathcal I^c}$.

For $m\ge1$, let $\mathbf z=(\mathbf z_r)_{r=1}^m$ have law $p_B^{\otimes m}$.
Write $\mathbf z^{ru,b}$ for $\mathbf z$ with only $z_{ru}$ replaced by $b$.
An update of $(r,u)$ draws
$\eta_{ru}\sim\operatorname{Bernoulli}(p_u(\mathbf z_r))$.
For real-valued functions $F,G$, define
\begin{equation}\label{eq:row-resampling-operators}
\begin{aligned}
 \calL_BF(\mathbf z)
 &=\sum_{r=1}^m\sum_{u\in\mathcal I^c}
    \E[F(\mathbf z^{ru,\eta_{ru}})-F(\mathbf z)\mid\mathbf z],\\
 \Gamma_B(F,G)(\mathbf z)
 &=\sum_{r=1}^m\sum_{u\in\mathcal I^c}
    \E\big[\{F(\mathbf z^{ru,\eta_{ru}})-F(\mathbf z)\}
           \{G(\mathbf z^{ru,\eta_{ru}})-G(\mathbf z)\}\mid\mathbf z\big].
\end{aligned}
\end{equation}
Each coordinate is updated at rate one. Conditional resampling preserves
$p_B^{\otimes m}$, so $\E_B\calL_BF=0$; the operator acts on vectors
componentwise. Here we take $m=1$.

Define the conditional probability that coordinate $(r,u)$ changes by
\begin{equation}\label{eq:row-update-second-moment}
 q_u(\mathbf z_r)
 =z_{ru}+p_u(\mathbf z_r)-2z_{ru}p_u(\mathbf z_r)
 =\E[(\eta_{ru}-z_{ru})^2\mid\mathbf z].
\end{equation}
Fix $(r,v)$ and change $z_{rv}$ from $0$ to $1$.
Since $(J_B)_{vv}=0$, $p_v(\mathbf z_r)$ does not change, so
$q_v(\mathbf z_r)$ changes by at most $1$.
For $u\ne v$, only $p_u(\mathbf z_r)$ changes in the definition of
$q_u(\mathbf z_r)$. Writing $\mathbf z_r^{v,b}$ for $\mathbf z_r$
with its $v$th coordinate replaced by $b$, we have
\begin{equation}\label{eq:row-probability-influence}
 |q_u(\mathbf z_r^{v,1})-q_u(\mathbf z_r^{v,0})|
 \le|p_u(\mathbf z_r^{v,1})-p_u(\mathbf z_r^{v,0})|
 \le\tfrac14|(J_B)_{uv}|\le C/n.
\end{equation}

For a deterministic vector $\mathbf v\in\R^L$, put
$L_{\mathbf v}=\mathbf v^\top\boldsymbol\xi$ and
$\norm{\mathbf v}_{\ell^\infty}=\max_{u\in\mathcal I^c}|v_u|$. Define
\begin{equation*}
\begin{aligned}
 \mathcal A_n(K_0)&=\{\mathbf v\in\R^L:
       \norm{\mathbf v}_{\ell^2}\le1,\quad
       \norm{\mathbf v}_{\ell^\infty}\le K_0/\sqrt n\},\\
 M_n&=\sup_{\mathbf a,\mathbf b,\mathbf c\in\mathcal A_n(K_0)}
           |\E_BL_{\mathbf a}L_{\mathbf b}L_{\mathbf c}|.
\end{aligned}
\end{equation*}
Our goal is $M_n\le C_{K_0}n^{-1/2}$.
Since $\delta_uL_{\mathbf v}=|v_u|$ and $\E_BL_{\mathbf v}=0$,
\eqref{eq:comparison-moment} gives, for each fixed $s\ge1$,
\begin{equation*}
 \sup_{\mathbf v\in\mathcal A_n(K_0)}\norm{L_{\mathbf v}}_{L^s(p_B)}
 \le C_s\sup_{\mathbf v\in\mathcal A_n(K_0)}\norm{\mathbf v}_{\ell^2}\le C_s.
\end{equation*}
In particular, H\"older's inequality gives $M_n\le C_3^3<\infty$.

Use $\calL_B,\Gamma_B$ from \eqref{eq:row-resampling-operators} with $m=1$.
Recall $m_u$ from \eqref{eq:row-fields}, and put
$D_\ell=\diag\{\ell'(m_u)\}_{u\in\mathcal I^c}$.
Write $\boldsymbol\rho(\mathbf z_1)=(\rho_u(\mathbf z_1))_u$, with
$\rho_u$ from \eqref{eq:row-linearization}.
For $\mathbf a,\mathbf b,\mathbf c\in\mathcal A_n(K_0)$, symmetry of $J_B$
and \eqref{eq:row-linearization} give
\begin{equation*}
 \calL_BL_{\mathbf a}
 =-L_{\mathbf a}+L_{J_BD_\ell\mathbf a}
       +\mathbf a^\top\boldsymbol\rho(\mathbf z_1).
\end{equation*}
An update of coordinate $u$ has increment
$\Delta_uL_{\mathbf a}=a_u(\eta_{1u}-z_{1u})$.
All increments in a product refer to the same update, and expectation
includes the draw $\eta_{1u}$.
Expanding the product after this update and using
$\E_B\calL_B(L_{\mathbf a}L_{\mathbf b}L_{\mathbf c})=0$, yields
\begin{equation}\label{eq:dense-third-stationarity}
\begin{aligned}
 3\E_BL_{\mathbf a}L_{\mathbf b}L_{\mathbf c}
 ={}&\E_BL_{J_BD_\ell\mathbf a}L_{\mathbf b}L_{\mathbf c}
       +\E_BL_{\mathbf a}L_{J_BD_\ell\mathbf b}L_{\mathbf c}\\
 &+\E_BL_{\mathbf a}L_{\mathbf b}L_{J_BD_\ell\mathbf c}\\
 &+\E_B[(\mathbf a^\top\boldsymbol\rho(\mathbf z_1))L_{\mathbf b}L_{\mathbf c}
       +(\mathbf b^\top\boldsymbol\rho(\mathbf z_1))L_{\mathbf a}L_{\mathbf c}\\
 &\hspace{3em}+(\mathbf c^\top\boldsymbol\rho(\mathbf z_1))L_{\mathbf a}L_{\mathbf b}]\\
 &+\E_B[\Gamma_B(L_{\mathbf a},L_{\mathbf b})L_{\mathbf c}
       +\Gamma_B(L_{\mathbf a},L_{\mathbf c})L_{\mathbf b}\\
 &\hspace{3em}+\Gamma_B(L_{\mathbf b},L_{\mathbf c})L_{\mathbf a}]\\
 &+\sum_u\E_B[\Delta_uL_{\mathbf a}\Delta_uL_{\mathbf b}\Delta_uL_{\mathbf c}].
\end{aligned}
\end{equation}
We bound the terms containing $\boldsymbol\rho$, $\Gamma_B$, and three increments
by proving, uniformly over $\mathbf a,\mathbf b,\mathbf c\in\mathcal A_n(K_0)$,
\begin{align}
 |\E_B[(\mathbf a^\top\boldsymbol\rho(\mathbf z_1))L_{\mathbf b}L_{\mathbf c}]|
 &\le Cn^{-1/2},\label{eq:dense-third-taylor-error}\\
 |\E_B[\Gamma_B(L_{\mathbf a},L_{\mathbf b})L_{\mathbf c}]|
 &\le C_{K_0}n^{-1/2},\label{eq:dense-third-covariance-error}\\
 \sum_u\E_B|\Delta_uL_{\mathbf a}\Delta_uL_{\mathbf b}\Delta_uL_{\mathbf c}|
 &\le K_0n^{-1/2}.\label{eq:dense-third-jump-error}
\end{align}

\par\medskip
\noindent\textit{Proof of \eqref{eq:dense-third-taylor-error}}: Equation \eqref{eq:row-linearization-remainder} and the Cauchy--Schwarz inequality give
\begin{equation*}
 \norm{\mathbf a^\top\boldsymbol\rho(\mathbf z_1)}_{L^2(p_B)}
 \le\sum_u|a_u|\norm{\rho_u(\mathbf z_1)}_{L^2(p_B)}
 \le\frac Cn\sqrt L\norm{\mathbf a}_{\ell^2}\le Cn^{-1/2}.
\end{equation*}
Applying H\"older's inequality and the uniform $L^4(p_B)$ bounds above, we obtain
\begin{equation*}
\begin{aligned}
 |\E_B[(\mathbf a^\top\boldsymbol\rho(\mathbf z_1))L_{\mathbf b}L_{\mathbf c}]|
 &\le\norm{\mathbf a^\top\boldsymbol\rho(\mathbf z_1)}_{L^2(p_B)}
       \norm{L_{\mathbf b}}_{L^4(p_B)}\norm{L_{\mathbf c}}_{L^4(p_B)}\\
 &\le Cn^{-1/2}.
\end{aligned}
\end{equation*}

\par\medskip
\noindent\textit{Proof of \eqref{eq:dense-third-covariance-error}}: Equation \eqref{eq:row-update-second-moment} gives
\begin{equation*}
 \Gamma_B(L_{\mathbf a},L_{\mathbf b})
 =\sum_u a_ub_uq_u(\mathbf z_1).
\end{equation*}
Changing coordinate $v$ changes $q_v$ by at most one.
For $u\ne v$, \eqref{eq:row-probability-influence} bounds the change
in $q_u$ by $|\alpha|/(4n)$. Hence
\begin{equation*}
\begin{aligned}
 \delta_v\Gamma_B(L_{\mathbf a},L_{\mathbf b})
 &\le |a_vb_v|+\frac{|\alpha|}{4n}\sum_{u\ne v}|a_ub_u|\\
 &\le\frac{K_0^2}{n}
       +\frac{|\alpha|}{4n}\norm{\mathbf a}_{\ell^2}\norm{\mathbf b}_{\ell^2}
 \le\frac{C_{K_0}}n,\\
 \norm{\boldsymbol\delta\Gamma_B(L_{\mathbf a},L_{\mathbf b})}_{\ell^2}
 &\le\frac{C_{K_0}\sqrt L}{n}\le C_{K_0}n^{-1/2}.
\end{aligned}
\end{equation*}
Since $\E_BL_{\mathbf c}=0$ and $\boldsymbol\delta L_{\mathbf c}=|\mathbf c|$,
the covariance bound \eqref{eq:comparison-cov} now gives
\begin{equation*}
\begin{aligned}
 |\E_B[\Gamma_B(L_{\mathbf a},L_{\mathbf b})L_{\mathbf c}]|
 &=|\cov_B(\Gamma_B(L_{\mathbf a},L_{\mathbf b}),L_{\mathbf c})|\\
 &\le C\norm{\boldsymbol\delta\Gamma_B(L_{\mathbf a},L_{\mathbf b})}_{\ell^2}
       \norm{\mathbf c}_{\ell^2}\le C_{K_0}n^{-1/2}.
\end{aligned}
\end{equation*}

\par\medskip
\noindent\textit{Proof of \eqref{eq:dense-third-jump-error}}: $|\eta_{1u}-z_{1u}|\le1$ gives
\begin{equation*}
 \sum_u\E_B|\Delta_uL_{\mathbf a}\Delta_uL_{\mathbf b}\Delta_uL_{\mathbf c}|
 \le\sum_u|a_ub_uc_u|
 \le\norm{\mathbf c}_{\ell^\infty}\norm{\mathbf a}_{\ell^2}\norm{\mathbf b}_{\ell^2}
 \le K_0n^{-1/2}.
\end{equation*}

The first two estimates also hold after permuting
$\mathbf a,\mathbf b,\mathbf c$.
Substituting \eqref{eq:dense-third-taylor-error}--\eqref{eq:dense-third-jump-error}
into \eqref{eq:dense-third-stationarity} gives
\begin{equation}\label{eq:dense-third-recursion}
\begin{aligned}
 3\E_BL_{\mathbf a}L_{\mathbf b}L_{\mathbf c}
 ={}&\E_BL_{J_BD_\ell\mathbf a}L_{\mathbf b}L_{\mathbf c}
      +\E_BL_{\mathbf a}L_{J_BD_\ell\mathbf b}L_{\mathbf c}\\
 &+\E_BL_{\mathbf a}L_{\mathbf b}L_{J_BD_\ell\mathbf c}
      +O_{K_0}(n^{-1/2}).
\end{aligned}
\end{equation}
By \eqref{eq:row-interaction-norms} and $0\le\ell'(m_u)\le1/4$,
\begin{equation*}
\begin{aligned}
 \norm{J_BD_\ell\mathbf v}_{\ell^2}
 &\le\norm{J_B}_{\rm op}\norm{D_\ell}_{\rm op}\norm{\mathbf v}_{\ell^2}
 \le\frac{|\alpha|}{4}\norm{\mathbf v}_{\ell^2},\\
 \norm{J_BD_\ell\mathbf v}_{\ell^\infty}
 &\le\max_u\sum_w|(J_B)_{uw}|\ell'(m_w)\norm{\mathbf v}_{\ell^\infty}
 \le\frac{|\alpha|}{4}\norm{\mathbf v}_{\ell^\infty}.
\end{aligned}
\end{equation*}
For $\alpha\ne0$, these bounds imply
$4J_BD_\ell\mathbf v/|\alpha|\in\mathcal A_n(K_0)$ whenever
$\mathbf v\in\mathcal A_n(K_0)$, so
\begin{equation*}
 \left|\E_BL_{J_BD_\ell\mathbf a}L_{\mathbf b}L_{\mathbf c}\right|
 =\frac{|\alpha|}{4}
    \left|\E_BL_{4J_BD_\ell\mathbf a/|\alpha|}L_{\mathbf b}L_{\mathbf c}\right|
 \le\frac{|\alpha|}{4}M_n.
\end{equation*}
For $\alpha=0$, $J_B=0$ and this expectation vanishes.
The same bound holds for the other two expectations in
\eqref{eq:dense-third-recursion}. Taking the supremum gives
\begin{equation*}
 3M_n\le3\frac{|\alpha|}{4}M_n+C_{K_0}n^{-1/2},\qquad
 M_n\le\frac{C_{K_0}}{3(1-|\alpha|/4)}n^{-1/2}.
\end{equation*}
The denominator is uniformly positive on the fixed compact parameter set.
This proves \eqref{eq:dense-third}.

For the bounded-weight assertion, suppose all three
coefficient vectors have coordinates bounded by $C_0$ in absolute value.
For any one of these vectors $\mathbf v$, $L\le n$ implies
\begin{equation*}
 \left\|\frac{\mathbf v}{(1+C_0)\sqrt n}\right\|_{\ell^2}
 \le\frac{C_0\sqrt L}{(1+C_0)\sqrt n}\le1,\qquad
 \left\|\frac{\mathbf v}{(1+C_0)\sqrt n}\right\|_{\ell^\infty}\le n^{-1/2}.
\end{equation*}
The normalized vectors belong to $\mathcal A_n(1)$.
Applying \eqref{eq:dense-third} and restoring the scale factors gives
\begin{equation*}
 |\E_BL_{\mathbf a}L_{\mathbf b}L_{\mathbf c}|
 \le(1+C_0)^3n^{3/2}C_1n^{-1/2}\le C_{C_0}n.
\end{equation*}
\end{proof}

{
To complete \eqref{eq:overlap-cubic}, we also need bounds on the absolute row sums of $\Sigma$.
}
To bound $|\Sigma_{uv}|$ for $u\ne v$
and $\max_u\sum_v|\Sigma_{uv}|$, use the one-row influence bound
\begin{equation*}
 \mathcal C_{uv}=|(J_B)_{uv}|/4,\qquad
 \max_u\sum_v\mathcal C_{uv}\le|\alpha|/4<1.
\end{equation*}
Since $\mathcal C$ is nonnegative and its row sums are strictly less
than one, $(I-\mathcal C)^{-1}$ exists and has nonnegative entries.
Multiplying the componentwise inequality
$(I-\mathcal C)\mathbf1\ge(1-|\alpha|/4)\mathbf1$
by this inverse gives
$(I-\mathcal C)^{-1}\mathbf1\le(1-|\alpha|/4)^{-1}\mathbf1$.
Using $(I-\mathcal C)^{-1}=I+(I-\mathcal C)^{-1}\mathcal C$
and $\max_{r,s}\mathcal C_{rs}\le|\alpha|/(4n)$, we obtain, for $u\ne v$,
\begin{equation*}
 [(I-\mathcal C)^{-1}]_{uv}
 =\sum_w[(I-\mathcal C)^{-1}]_{uw}\mathcal C_{wv}
 \le\frac{|\alpha|}{4n}\sum_w[(I-\mathcal C)^{-1}]_{uw}
 \le\frac{|\alpha|}{4n(1-|\alpha|/4)}.
\end{equation*}
Applying \eqref{eq:comparison-cov} to $z_{1u},z_{1v}$ therefore gives
\begin{equation}\label{eq:row-covariance-entries}
\begin{aligned}
 |\Sigma_{uv}|&\le\tfrac14[(I-\mathcal C)^{-1}]_{uv},\qquad
 \max_u\sum_v|\Sigma_{uv}|\le\frac1{4(1-|\alpha|/4)},\\
 |\Sigma_{uv}|&\le\frac{|\alpha|}{16n(1-|\alpha|/4)}
 \quad(u\ne v).
\end{aligned}
\end{equation}

Put $\widetilde{\boldsymbol\xi}=\mathbf z_2-\boldsymbol\mu$.
Conditioning on $\mathbf z_2$ in
\begin{equation*}
 \mathbf z_1^\top\mathbf z_2-\boldsymbol\mu^\top\boldsymbol\mu
 =\mathbf z_2^\top\boldsymbol\xi
   +\boldsymbol\mu^\top\widetilde{\boldsymbol\xi}
\end{equation*}
gives
\begin{equation*}
\begin{aligned}
 \E_B(\mathbf z_1^\top\mathbf z_2-\boldsymbol\mu^\top\boldsymbol\mu)^3
 ={}&\E_B\!\left[\E_B\{(\mathbf z_2^\top\boldsymbol\xi)^3\mid\mathbf z_2\}\right]
     +\E_B(\boldsymbol\mu^\top\widetilde{\boldsymbol\xi})^3\\
 &+3\cov_B(\boldsymbol\mu^\top\mathbf z_2,
            \mathbf z_2^\top\Sigma\mathbf z_2).
\end{aligned}
\end{equation*}
The first two terms are $O(n)$ by the bounded-weight conclusion of
Lemma~\ref{lem:row-refined-third}: conditionally, every coordinate of
$\mathbf z_2$ lies in $[0,1]$, and $0\le\mu_u\le1$.
Equation~\eqref{eq:row-covariance-entries} gives
\begin{equation*}
 \delta_u(\mathbf z_2^\top\Sigma\mathbf z_2)
 \le |\Sigma_{uu}|+2\sum_{v\ne u}|\Sigma_{uv}|\le C,
 \qquad \delta_u(\boldsymbol\mu^\top\mathbf z_2)\le1.
\end{equation*}
Lemma~\ref{lem:comparison} therefore bounds the covariance by $Cn$.
The resulting bound $Cn$ implies \eqref{eq:overlap-cubic}, since $H_B\ge1$.
\end{proof}

\subsection{Proof of Lemma~\ref{lem:row-coarse}}\label{sec:row-response}

\begin{proof}[Proof of \eqref{eq:row-coarse-a}:\nopunct]
Put $\boldsymbol\mu=\E_B\mathbf z_1$, $\Sigma=\cov_B(\mathbf z_1)$,
$\boldsymbol\xi=\mathbf z_1-\boldsymbol\mu$, and
$\mathbf e=\boldsymbol\mu-\varrho_{\alpha,h}(q_B)\mathbf1$.
Recall $m_u,H_u$ from \eqref{eq:row-fields} and
$D_\ell=\diag\{\ell'(m_u)\}_{u\in\mathcal I^c}$.
The covariance estimates and the absolute row-sum bound follow from
\eqref{eq:row-covariance-entries}. We first bound $\norm{\mathbf e}_{\ell^2}$.
Write
\begin{equation*}
 e_u=\{\mu_u-\ell(m_u)\}+\{\ell(m_u)-\varrho_{\alpha,h}(q_B)\}.
\end{equation*}
Since $\boldsymbol\mu=\mathbf e+\varrho_{\alpha,h}(q_B)\mathbf1$,
\begin{equation*}
 m_u-\{h+\alpha q_B\varrho_{\alpha,h}(q_B)\}
 =(J_B\mathbf e)_u
 +\varrho_{\alpha,h}(q_B)(J_B\mathbf1-\alpha q_B\mathbf1)_u.
\end{equation*}
Together with \eqref{eq:triangle-row-root} and $\sup\ell'=1/4$, this gives
\begin{equation}\label{eq:row-mean-error-decomposition}
\begin{aligned}
 \norm{\mathbf e}_{\ell^2}
 \le{}&\norm{(\mu_u-\ell(m_u))_u}_{\ell^2}+\tfrac14\norm{J_B\mathbf e}_{\ell^2}\\
 &+\tfrac14\varrho_{\alpha,h}(q_B)
       \norm{J_B\mathbf1-\alpha q_B\mathbf1}_{\ell^2}.
\end{aligned}
\end{equation}
For the first two terms, \eqref{eq:row-mean-equation} and
\eqref{eq:row-interaction-norms} give
\begin{equation*}
 \norm{(\mu_u-\ell(m_u))_u}_{\ell^2}\le C\sqrt L/n,\qquad
 \norm{J_B\mathbf e}_{\ell^2}\le|\alpha|\norm{\mathbf e}_{\ell^2}.
\end{equation*}
For the third term, $L=n-|\mathcal I|$ and the definitions in
\eqref{eq:triangle-background} give
\begin{equation}\label{eq:exterior-field-deviation}
\begin{aligned}
 \norm{J_B\mathbf1-\alpha q_B\mathbf1}_{\ell^2}
 &\le\frac{|\alpha|}{n}\norm{B\mathbf1-(L-1)q_B\mathbf1}_{\ell^2}+Cn^{-1/2}\\
 &=|\alpha|(H_B-1)+Cn^{-1/2}\le CH_B.
\end{aligned}
\end{equation}
Substituting \eqref{eq:row-mean-equation}, \eqref{eq:row-interaction-norms}, and
\eqref{eq:exterior-field-deviation}
into \eqref{eq:row-mean-error-decomposition} and using
$0<\varrho_{\alpha,h}(q_B)<1$, we obtain
\begin{equation*}
 \left(1-\frac{|\alpha|}{4}\right)\norm{\mathbf e}_{\ell^2}
 \le C\sqrt L/n+CH_B.
\end{equation*}
Since $L\le n$, $H_B\ge1$, and $1-|\alpha|/4$ is uniformly positive on the fixed
compact parameter set, we conclude, together with the covariance bound, that
\begin{equation}\label{eq:row-basic-mean-cov}
 \norm{\boldsymbol\mu-\varrho_{\alpha,h}(q_B)\mathbf1}_{\ell^2}\le CH_B,
 \qquad \norm\Sigma_{\rm op}\le C.
\end{equation}
This proves \eqref{eq:row-coarse-a}.

Since $\ell'$ is Lipschitz and
$\ell'(h+\alpha q_B\varrho_{\alpha,h}(q_B))=v_{\alpha,h}(q_B)$,
\eqref{eq:row-interaction-norms}, \eqref{eq:exterior-field-deviation}, and
\eqref{eq:row-basic-mean-cov} imply
\begin{equation}\label{eq:row-slope-coefficient-bound}
 \norm{D_\ell\mathbf1-v_{\alpha,h}(q_B)\mathbf1}_{\ell^2}
 \le C\norm{J_B\mathbf e+
     \varrho_{\alpha,h}(q_B)(J_B\mathbf1-\alpha q_B\mathbf1)}_{\ell^2}
 \le CH_B.
\end{equation}
\par\medskip
To prove  \eqref{eq:row-refined}, we start with
\begin{equation}\label{eq:row-response-vector}
 \norm{\Sigma\mathbf1-A_{\alpha,h}(q_B)\mathbf1}_{\ell^2}\le CH_B.
\end{equation}
We first derive an equation for the difference in
\eqref{eq:row-response-vector} and then bound its right-hand side.
The conditional identity
$p_u=\E_B(z_{1u}\mid\mathbf z_{1,-u})$ yields
\begin{equation*}
 \cov_B(p_u,z_{1u})
 =\cov_B\{p_u,\E_B(z_{1u}\mid\mathbf z_{1,-u})\}
 =\var_B(p_u).
\end{equation*}
Conditioning on $\mathbf z_{1,-u}$ in each off-diagonal covariance gives
\begin{equation*}
\begin{aligned}
 (\Sigma\mathbf1)_u
 &=\mu_u(1-\mu_u)
      +\cov_B\left(p_u,\sum_{v\ne u}\xi_v\right) =\mu_u(1-\mu_u)+\cov_B(p_u,\mathbf1^\top\boldsymbol\xi)
                          -\var_B(p_u)\\
 &=\mu_u(1-\mu_u)+\ell'(m_u)(J_B\Sigma\mathbf1)_u
      +\cov_B(\rho_u(\mathbf z_1),\mathbf1^\top\boldsymbol\xi)-\var_B(p_u),
\end{aligned}
\end{equation*}
where we use  $\cov_B((J_B\boldsymbol\xi)_u,\mathbf1^\top\boldsymbol\xi)
=(J_B\Sigma\mathbf1)_u$ and
 \eqref{eq:row-linearization}.
Recall that $D_\ell=\diag\{\ell'(m_u)\}_{u\in\mathcal I^c}$.
Thus, with $\boldsymbol\epsilon=(\epsilon_u)_{u\in\mathcal I^c}$,
\begin{equation}\label{eq:exact-response}
(I-D_\ell J_B)\Sigma\mathbf1
 =\bigl(\mu_u(1-\mu_u)\bigr)_u+\boldsymbol\epsilon,\qquad \epsilon_u=\cov_B(\mathbf1^\top\boldsymbol\xi,\rho_u(\mathbf z_1))-\var_B(p_u).
\end{equation}
By \eqref{eq:triangle-r},
\begin{equation*}
 A_{\alpha,h}(q_B)
 =v_{\alpha,h}(q_B)
     +\alpha q_Bv_{\alpha,h}(q_B)A_{\alpha,h}(q_B).
\end{equation*}
Subtracting $(I-D_\ell J_B)A_{\alpha,h}(q_B)\mathbf1$
from both sides of \eqref{eq:exact-response} gives
\begin{equation}\label{eq:row-response-error-decomposition}
\begin{aligned}
 &(I-D_\ell J_B)
       (\Sigma\mathbf1-A_{\alpha,h}(q_B)\mathbf1) =\bigl(\mu_u(1-\mu_u)-v_{\alpha,h}(q_B)\bigr)_u\\
 &\qquad+A_{\alpha,h}(q_B)
       \{D_\ell J_B\mathbf1-\alpha q_Bv_{\alpha,h}(q_B)\mathbf1\}
       +\boldsymbol\epsilon.
\end{aligned}
\end{equation}
To bound the first two vectors on the right, the Lipschitz property of
$x\mapsto x(1-x)$ and
\eqref{eq:row-basic-mean-cov}, \eqref{eq:row-slope-coefficient-bound} and
\eqref{eq:exterior-field-deviation} give
\begin{equation}\label{eq:row-coefficient-bounds}
\begin{aligned}
 \norm{\bigl(\mu_u(1-\mu_u)-v_{\alpha,h}(q_B)\bigr)_u}_{\ell^2}
 &\le C\norm{\mathbf e}_{\ell^2}\le CH_B,\\
 \norm{D_\ell J_B\mathbf1-\alpha q_Bv_{\alpha,h}(q_B)\mathbf1}_{\ell^2}
 &\le C\norm{D_\ell\mathbf1-v_{\alpha,h}(q_B)\mathbf1}_{\ell^2}
      +C\norm{J_B\mathbf1-\alpha q_B\mathbf1}_{\ell^2}\le CH_B.
\end{aligned}
\end{equation}
For the remainder $\boldsymbol\epsilon$, first use $p_u=\ell(m_u+H_u)$
and \eqref{eq:row-field-moment} to obtain
\begin{equation}\label{eq:row-probability-variance}
\begin{aligned}
 \var_B(p_u)
 &\le\E_B|p_u-\ell(m_u)|^2
 \le\tfrac1{16}\E_BH_u^2\le C/n.
\end{aligned}
\end{equation}
By \eqref{eq:row-basic-mean-cov} and \eqref{eq:row-linearization-remainder},
we obtain
\begin{equation*}
\begin{aligned}
|\epsilon_u|
 &\le\norm{\mathbf1^\top\boldsymbol\xi}_{L^2(p_B)}\norm{\rho_u(\mathbf z_1)}_{L^2(p_B)}
                      +\var_B(p_u)\le C\sqrt n\,n^{-1}+C/n\le Cn^{-1/2},\\
 \norm{\boldsymbol\epsilon}_{\ell^2}&\le\sqrt L\max_u|\epsilon_u|\le C.
\end{aligned}
\end{equation*}
Since $0<\ell'(m_u)\le1/4$, \eqref{eq:row-interaction-norms} gives
$\norm{D_\ell J_B}_{\rm op}\le|\alpha|/4<1$.
Hence $I-D_\ell J_B$ is invertible, and the Neumann series gives
\begin{equation}\label{eq:row-response-inverse}
 \norm{(I-D_\ell J_B)^{-1}}_{\rm op}
 \le\frac1{1-\norm{D_\ell J_B}_{\rm op}}
 \le\frac1{1-|\alpha|/4}.
\end{equation}
The coefficient $A_{\alpha,h}(q_B)$ is uniformly bounded by \eqref{eq:triangle-r}.
Using \eqref{eq:row-coefficient-bounds}, $\norm{\boldsymbol\epsilon}_{\ell^2}\le C$
and \eqref{eq:row-response-inverse} in
\eqref{eq:row-response-error-decomposition} gives
\begin{equation*}
 \norm{\Sigma\mathbf1-A_{\alpha,h}(q_B)\mathbf1}_{\ell^2}
 \le\frac{CH_B+C}{1-|\alpha|/4}\le CH_B,
\end{equation*}
which proves \eqref{eq:row-response-vector}.

\par\medskip
\noindent\textit{Proof of \eqref{eq:row-refined}}: We first estimate $\var_B(\mathbf1^\top\mathbf z_1)$.
Write
\begin{equation*}
 \boldsymbol\mu=\E_B\mathbf z_1,\qquad
 \Sigma=\cov_B(\mathbf z_1),\qquad
 \mathbf e=\boldsymbol\mu-\varrho_{\alpha,h}(q_B)\mathbf1.
\end{equation*}
Put
\begin{equation*}
 \mathbf w=J_BD_\ell\mathbf1-\alpha q_Bv_{\alpha,h}(q_B)\mathbf1.
\end{equation*}
By symmetry of $J_B$,
\begin{equation*}
\begin{aligned}
 \sum_u\ell'(m_u)(J_B\Sigma\mathbf1)_u
 &=(J_BD_\ell\mathbf1)^\top\Sigma\mathbf1\\
 &=\alpha q_Bv_{\alpha,h}(q_B)\var_B(\mathbf1^\top\mathbf z_1)
   +\mathbf w^\top\Sigma\mathbf1.
\end{aligned}
\end{equation*}
Summing \eqref{eq:exact-response} therefore gives
\begin{equation}\label{eq:refined-response-sum}
\begin{aligned}
 &(1-\alpha q_Bv_{\alpha,h}(q_B))\var_B(\mathbf1^\top\mathbf z_1)\\
 &\qquad=\sum_u\mu_u(1-\mu_u)
          +\mathbf w^\top\Sigma\mathbf1+\sum_u\epsilon_u.
\end{aligned}
\end{equation}
By \eqref{eq:triangle-r},
\begin{equation*}
\begin{aligned}
 1-\alpha q_Bv_{\alpha,h}(q_B)&\ge1-|\alpha|/4>0,\\
 (1-\alpha q_Bv_{\alpha,h}(q_B))A_{\alpha,h}(q_B)
 &=v_{\alpha,h}(q_B).
\end{aligned}
\end{equation*}
It thus suffices to prove
\begin{align}
 \left|\sum_u\mu_u(1-\mu_u)-nv_{\alpha,h}(q_B)\right|
 &\le CH_B^2,\label{eq:refined-diagonal-target}\\
 |\mathbf w^\top\Sigma\mathbf1|&\le CH_B^2,\label{eq:refined-cross-target}\\
 \left|\sum_u\epsilon_u\right|&\le C.\label{eq:refined-error-target}
\end{align}
We prove these three bounds in order.

\par\medskip
\noindent\textit{Proof of \eqref{eq:refined-diagonal-target}}: Put
$\mathbf s=J_B\mathbf1-\alpha q_B\mathbf1$.
Equations \eqref{eq:row-basic-mean-cov}, \eqref{eq:exterior-field-deviation},
\eqref{eq:row-slope-coefficient-bound} and \eqref{eq:row-response-vector} give
\begin{equation*}
\begin{aligned}
 \norm{\mathbf e}_{\ell^2}+\norm{\mathbf s}_{\ell^2}&\le CH_B,\\
 \norm{D_\ell\mathbf1-v_{\alpha,h}(q_B)\mathbf1}_{\ell^2}&\le CH_B,\\
 \norm{\Sigma\mathbf1-A_{\alpha,h}(q_B)\mathbf1}_{\ell^2}&\le CH_B.
\end{aligned}
\end{equation*}
We first bound $\mathbf1^\top\mathbf e$.
The definition of $q_B$ in \eqref{eq:triangle-background} and $L=n-k$ imply
\begin{equation}\label{eq:refined-field-sum}
\begin{aligned}
 \mathbf1^\top\mathbf s
 &=\frac\alpha nL(L-1)q_B-\alpha q_BL\\
 &=-\frac{\alpha q_BL(k+1)}n=O(1).
\end{aligned}
\end{equation}
Since $\boldsymbol\mu=\varrho_{\alpha,h}(q_B)\mathbf1+\mathbf e$,
\begin{equation*}
 m_u-\{h+\alpha q_B\varrho_{\alpha,h}(q_B)\}
 =(J_B\mathbf e+\varrho_{\alpha,h}(q_B)\mathbf s)_u.
\end{equation*}
By \eqref{eq:row-mean-equation}, $\mu_u=\ell(m_u)+O(n^{-1})$.
Taylor's formula at $h+\alpha q_B\varrho_{\alpha,h}(q_B)$, together with
\eqref{eq:triangle-row-root} and \eqref{eq:triangle-r}, gives
\begin{equation*}
\begin{aligned}
 e_u={}&v_{\alpha,h}(q_B)
          (J_B\mathbf e+\varrho_{\alpha,h}(q_B)\mathbf s)_u\\
 &+O\{n^{-1}+
          (J_B\mathbf e+\varrho_{\alpha,h}(q_B)\mathbf s)_u^2\}.
\end{aligned}
\end{equation*}
Using \eqref{eq:row-interaction-norms} and the preceding bounds on
$\mathbf e,\mathbf s$, we can sum the remainders:
\begin{equation*}
\begin{aligned}
 &\sum_u\{n^{-1}+
       (J_B\mathbf e+\varrho_{\alpha,h}(q_B)\mathbf s)_u^2\}\\
 &\qquad=L/n+
          \norm{J_B\mathbf e+\varrho_{\alpha,h}(q_B)\mathbf s}_{\ell^2}^2\\
 &\qquad\le1+
          (\norm{J_B}_{\rm op}\norm{\mathbf e}_{\ell^2}+\norm{\mathbf s}_{\ell^2})^2
          \le CH_B^2.
\end{aligned}
\end{equation*}
Symmetry gives
\begin{equation*}
 \mathbf1^\top J_B\mathbf e
 =(J_B\mathbf1)^\top\mathbf e
 =\alpha q_B\mathbf1^\top\mathbf e+\mathbf s^\top\mathbf e.
\end{equation*}
Summing the expansion of $e_u$ and using \eqref{eq:refined-field-sum}, we get
\begin{equation}\label{eq:refined-mean-sum}
\begin{aligned}
 (1-\alpha q_Bv_{\alpha,h}(q_B))\mathbf1^\top\mathbf e
 &=v_{\alpha,h}(q_B)\mathbf s^\top\mathbf e
          +v_{\alpha,h}(q_B)\varrho_{\alpha,h}(q_B)\mathbf1^\top\mathbf s
          +O(H_B^2),\\
 |\mathbf1^\top\mathbf e|
 &\le C\{\norm{\mathbf s}_{\ell^2}\norm{\mathbf e}_{\ell^2}
            +|\mathbf1^\top\mathbf s|+H_B^2\}
       \le CH_B^2.
\end{aligned}
\end{equation}
The inequality uses $1-\alpha q_Bv_{\alpha,h}(q_B)\ge1-|\alpha|/4>0$.
Expanding $\mu_u(1-\mu_u)$ now yields
\begin{equation}\label{eq:refined-diagonal-sum}
\begin{aligned}
 \sum_u\mu_u(1-\mu_u)
 &=Lv_{\alpha,h}(q_B)
    +(1-2\varrho_{\alpha,h}(q_B))\mathbf1^\top\mathbf e
    -\norm{\mathbf e}_{\ell^2}^2\\
 &=nv_{\alpha,h}(q_B)+O(H_B^2).
\end{aligned}
\end{equation}
This proves \eqref{eq:refined-diagonal-target}.

\par\medskip
\noindent\textit{Proof of \eqref{eq:refined-cross-target}}: Taylor's formula for $\ell'$ gives
\begin{equation*}
\begin{aligned}
 \ell'(m_u)-v_{\alpha,h}(q_B)
 &=\ell''(h+\alpha q_B\varrho_{\alpha,h}(q_B))
             (J_B\mathbf e+\varrho_{\alpha,h}(q_B)\mathbf s)_u\\
 &\qquad+O\{(J_B\mathbf e+\varrho_{\alpha,h}(q_B)\mathbf s)_u^2\}.
\end{aligned}
\end{equation*}
By \eqref{eq:refined-field-sum} and \eqref{eq:refined-mean-sum},
\begin{equation*}
\begin{aligned}
 \mathbf1^\top(J_B\mathbf e+\varrho_{\alpha,h}(q_B)\mathbf s)=\alpha q_B\mathbf1^\top\mathbf e+\mathbf s^\top\mathbf e
          +\varrho_{\alpha,h}(q_B)\mathbf1^\top\mathbf s
          =O(H_B^2).
\end{aligned}
\end{equation*}
The squared terms sum to
$\norm{J_B\mathbf e+\varrho_{\alpha,h}(q_B)\mathbf s}_{\ell^2}^2\le CH_B^2$,
as shown above. Summing the Taylor expansion therefore gives
\begin{equation*}
\begin{aligned}
 \mathbf1^\top(D_\ell\mathbf1-v_{\alpha,h}(q_B)\mathbf1)
 &=\ell''(h+\alpha q_B\varrho_{\alpha,h}(q_B))
             \mathbf1^\top(J_B\mathbf e+\varrho_{\alpha,h}(q_B)\mathbf s)\\
 &\qquad+O(\norm{J_B\mathbf e+\varrho_{\alpha,h}(q_B)\mathbf s}_{\ell^2}^2)
          =O(H_B^2).
\end{aligned}
\end{equation*}
The definition of $\mathbf w$ gives
\begin{equation*}
 \mathbf w=J_B(D_\ell\mathbf1-v_{\alpha,h}(q_B)\mathbf1)
             +v_{\alpha,h}(q_B)\mathbf s.
\end{equation*}
Using $J_B\mathbf1=\alpha q_B\mathbf1+\mathbf s$ and symmetry once more,
we obtain
\begin{equation}\label{eq:refined-sums}
\begin{aligned}
 \norm{\mathbf w}_{\ell^2}
 &\le\norm{J_B}_{\rm op}
        \norm{D_\ell\mathbf1-v_{\alpha,h}(q_B)\mathbf1}_{\ell^2}
       +v_{\alpha,h}(q_B)\norm{\mathbf s}_{\ell^2}
       \le CH_B,\\
 \mathbf1^\top\mathbf w
 &=\alpha q_B\mathbf1^\top(D_\ell\mathbf1-v_{\alpha,h}(q_B)\mathbf1)\\
 &\quad+\mathbf s^\top(D_\ell\mathbf1-v_{\alpha,h}(q_B)\mathbf1)
        +v_{\alpha,h}(q_B)\mathbf1^\top\mathbf s,\\
 |\mathbf1^\top\mathbf w|
 &\le C\{H_B^2+
       \norm{\mathbf s}_{\ell^2}\norm{D_\ell\mathbf1-v_{\alpha,h}(q_B)\mathbf1}_{\ell^2}+1\}
       \le CH_B^2.
\end{aligned}
\end{equation}
Finally,
\begin{equation}\label{eq:refined-response-cross-term}
\begin{aligned}
 \mathbf w^\top\Sigma\mathbf1
 &=A_{\alpha,h}(q_B)\mathbf1^\top\mathbf w
    +\mathbf w^\top(\Sigma\mathbf1-A_{\alpha,h}(q_B)\mathbf1),\\
 |\mathbf w^\top\Sigma\mathbf1|
 &\le |A_{\alpha,h}(q_B)|\,|\mathbf1^\top\mathbf w|\\
 &\quad+\norm{\mathbf w}_{\ell^2}
          \norm{\Sigma\mathbf1-A_{\alpha,h}(q_B)\mathbf1}_{\ell^2}
       \le CH_B^2.
\end{aligned}
\end{equation}
The last bound follows from \eqref{eq:row-response-vector} and
\eqref{eq:refined-sums}. This proves \eqref{eq:refined-cross-target}.

\par\medskip
\noindent\textit{Proof of \eqref{eq:refined-error-target}}: Put
$R=\mathbf1^\top\boldsymbol\xi$.
Recall $H_u=(J_B\boldsymbol\xi)_u$ from \eqref{eq:row-fields}.
Write $(J_B)_{u\cdot}$ for row $u$ of $J_B$.
Since $J_B=\alpha B/n$ and $|\alpha|<4$, the coefficient vectors of
$R/\sqrt n$ and $\sqrt nH_u/4$ satisfy
\begin{equation*}
\begin{aligned}
 \norm{\frac{\mathbf1}{\sqrt n}}_{\ell^2}&\le1,
 &\norm{\frac{\mathbf1}{\sqrt n}}_{\ell^\infty}&=\frac1{\sqrt n},\\
 \norm{\frac{\sqrt n}{4}(J_B)_{u\cdot}}_{\ell^2}&\le\frac{|\alpha|}{4}<1,
 &\norm{\frac{\sqrt n}{4}(J_B)_{u\cdot}}_{\ell^\infty}
 &\le\frac{|\alpha|}{4\sqrt n}\le\frac1{\sqrt n}.
\end{aligned}
\end{equation*}
Applying \eqref{eq:dense-third} to these two linear sums, with the second
used twice, and using $\E_BR=0$, gives
\begin{equation*}
 |\cov_B(R,H_u^2)|
 =\frac{16}{\sqrt n}
   \left|\E_B\left[\frac R{\sqrt n}
                \left(\frac{\sqrt nH_u}{4}\right)^2\right]\right|
 \le\frac Cn.
\end{equation*}
The remainder $\rho_u(\mathbf z_1)$ in \eqref{eq:row-linearization} is centered.
Since $\E_BH_u=0$, centering the second-order Taylor expansion of
$p_u=\ell(m_u+H_u)$ yields
\begin{equation*}
 \norm{\rho_u(\mathbf z_1)
       -\tfrac12\ell''(m_u)(H_u^2-\E_BH_u^2)}_{L^2(p_B)}
 \le C\norm{H_u}_{L^6(p_B)}^3\le Cn^{-3/2},
\end{equation*}
where the last bound is \eqref{eq:row-field-moment}.
Also, \eqref{eq:comparison-moment} gives
$\norm R_{L^2(p_B)}\le C\sqrt L\le C\sqrt n$, and
\eqref{eq:row-probability-variance} gives $\var_B(p_u)\le C/n$.
The definition of $\epsilon_u$ in \eqref{eq:exact-response} therefore implies
\begin{equation}\label{eq:refined-response-error}
\begin{aligned}
 |\cov_B(R,\rho_u(\mathbf z_1))|
 &\le C|\cov_B(R,H_u^2)|
       +C\norm R_{L^2(p_B)}n^{-3/2}\le C/n,\\
 |\epsilon_u|
 &\le|\cov_B(R,\rho_u(\mathbf z_1))|+\var_B(p_u)\le C/n,\\
 \left|\sum_u\epsilon_u\right|&\le LC/n\le C.
\end{aligned}
\end{equation}
This proves \eqref{eq:refined-error-target}.

Substituting \eqref{eq:refined-diagonal-target}--\eqref{eq:refined-error-target}
into \eqref{eq:refined-response-sum}, and using $H_B\ge1$, gives
\begin{equation}\label{eq:refined-row-sum-variance}
\begin{aligned}
 &(1-\alpha q_Bv_{\alpha,h}(q_B))
        \var_B(\mathbf1^\top\mathbf z_1)
   =nv_{\alpha,h}(q_B)+O(H_B^2),\\
 &\var_B(\mathbf1^\top\mathbf z_1)
   =\frac{nv_{\alpha,h}(q_B)+O(H_B^2)}
          {1-\alpha q_Bv_{\alpha,h}(q_B)}
   =nA_{\alpha,h}(q_B)+O(H_B^2).
\end{aligned}
\end{equation}
Here we used $1-\alpha q_Bv_{\alpha,h}(q_B)\ge1-|\alpha|/4>0$
and the definition of $A_{\alpha,h}$ in \eqref{eq:triangle-r}.
This proves the row-sum variance estimate in \eqref{eq:row-refined}.

For the overlap variance, put
$\widetilde{\boldsymbol\xi}=\mathbf z_2-\boldsymbol\mu$.
Expanding the overlap gives
\begin{equation*}
 \mathbf z_1^\top\mathbf z_2-\boldsymbol\mu^\top\boldsymbol\mu
 =\boldsymbol\mu^\top\boldsymbol\xi
      +\boldsymbol\mu^\top\widetilde{\boldsymbol\xi}
      +\boldsymbol\xi^\top\widetilde{\boldsymbol\xi}.
\end{equation*}
By independence and centering, the three terms have mean zero, and their
pairwise products have expectation zero. For example,
\begin{equation*}
\E_B[(\boldsymbol\mu^\top\boldsymbol\xi)
          (\boldsymbol\mu^\top\widetilde{\boldsymbol\xi})]=0,\qquad \E_B[(\boldsymbol\mu^\top\boldsymbol\xi)
          (\boldsymbol\xi^\top\widetilde{\boldsymbol\xi})]
 =\E_B[(\boldsymbol\mu^\top\boldsymbol\xi)
          \boldsymbol\xi^\top\E_B\widetilde{\boldsymbol\xi}]=0.
\end{equation*}
Writing $\Tr M=\sum_uM_{uu}$ for matrix trace, we also have
\begin{equation*}
 \E_B[(\boldsymbol\xi^\top\widetilde{\boldsymbol\xi})^2]
 =\sum_{u,v}\E_B(\xi_u\xi_v)
                \E_B(\widetilde\xi_u\widetilde\xi_v)
 =\sum_{u,v}\Sigma_{uv}^2=\Tr(\Sigma^2).
\end{equation*}
The two linear terms each have variance
$\boldsymbol\mu^\top\Sigma\boldsymbol\mu$. Hence
\begin{equation}\label{eq:independent-overlap}
 \var_B(\mathbf z_1^\top\mathbf z_2)
 =2\boldsymbol\mu^\top\Sigma\boldsymbol\mu+\Tr(\Sigma^2).
\end{equation}
By \eqref{eq:row-basic-mean-cov}, \eqref{eq:row-response-vector}
and \eqref{eq:refined-mean-sum},
\begin{equation*}
\begin{aligned}
 |\mathbf e^\top\Sigma\mathbf1|
 &\le |A_{\alpha,h}(q_B)|\,|\mathbf1^\top\mathbf e|
      +\norm{\mathbf e}_{\ell^2}
       \norm{\Sigma\mathbf1-A_{\alpha,h}(q_B)\mathbf1}_{\ell^2}
 \le CH_B^2,\\
 |\mathbf e^\top\Sigma\mathbf e|
 &\le\norm{\Sigma}_{\rm op}\norm{\mathbf e}_{\ell^2}^2\le CH_B^2.
\end{aligned}
\end{equation*}
Thus $\boldsymbol\mu=\varrho_{\alpha,h}(q_B)\mathbf1+\mathbf e$ and
\eqref{eq:refined-row-sum-variance} give
\begin{equation*}
\begin{aligned}
 \boldsymbol\mu^\top\Sigma\boldsymbol\mu
 &=\varrho_{\alpha,h}(q_B)^2\mathbf1^\top\Sigma\mathbf1
    +2\varrho_{\alpha,h}(q_B)\mathbf e^\top\Sigma\mathbf1
    +\mathbf e^\top\Sigma\mathbf e\\
 &=n\varrho_{\alpha,h}(q_B)^2A_{\alpha,h}(q_B)+O(H_B^2).
\end{aligned}
\end{equation*}

For the trace, use
\begin{equation*}
 \Sigma_{uu}=\mu_u(1-\mu_u)
 =v_{\alpha,h}(q_B)+(1-2\varrho_{\alpha,h}(q_B))e_u-e_u^2.
\end{equation*}
Since $|e_u|\le1$, squaring and summing gives
\begin{equation*}
\begin{aligned}
 \sum_u\Sigma_{uu}^2
 &=Lv_{\alpha,h}(q_B)^2
   +2v_{\alpha,h}(q_B)(1-2\varrho_{\alpha,h}(q_B))
         \mathbf1^\top\mathbf e
   +O(\norm{\mathbf e}_{\ell^2}^2)\\
 &=nv_{\alpha,h}(q_B)^2+O(H_B^2).
\end{aligned}
\end{equation*}
The last equality uses $L=n-k$, \eqref{eq:row-basic-mean-cov}
and \eqref{eq:refined-mean-sum}.
By \eqref{eq:row-covariance-entries},
\begin{equation*}
 \sum_{u\ne v}\Sigma_{uv}^2\le C L(L-1)/n^2\le C.
\end{equation*}
Since $\Sigma$ is symmetric and $H_B\ge1$,
\begin{equation*}
 \Tr(\Sigma^2)
 =\sum_u\Sigma_{uu}^2+\sum_{u\ne v}\Sigma_{uv}^2
 =nv_{\alpha,h}(q_B)^2+O(H_B^2).
\end{equation*}
Substituting these estimates into \eqref{eq:independent-overlap}
and using the definition of $K_{\alpha,h}$ in \eqref{eq:triangle-r} gives
\begin{equation*}
\begin{aligned}
 \var_B(\mathbf z_1^\top\mathbf z_2)
 &=n\{2\varrho_{\alpha,h}(q_B)^2A_{\alpha,h}(q_B)
                  +v_{\alpha,h}(q_B)^2\}+O(H_B^2)\\
 &=nK_{\alpha,h}(q_B)+O(H_B^2).
\end{aligned}
\end{equation*}
This proves the overlap variance estimate in \eqref{eq:row-refined}.
Finally, \eqref{eq:triangle-background} gives $1\le H_B\le C\sqrt n$, so
$H_B^2\le C\sqrt nH_B$. Thus \eqref{eq:row-refined} implies
\eqref{eq:row-coarse-b} and \eqref{eq:row-coarse-c}.
\end{proof}

\subsection{Proof of Lemma~\ref{lem:actual-refined}}\label{sec:actual-refined-proof}

{
Condition on $\calF_{\mathcal I}$ and label the vertices of $\mathcal I$
as $1,\ldots,k$, where $k=|\mathcal I|$.
Take $(D,\mathcal T)=(d_1-d_2,A_{\alpha,h})$ or
$(M_{\{1,2\}}-M_{\{3,4\}},K_{\alpha,h})$, and put
$m(t)=\E_t[D^2/(2n)]$.
Taylor's formula gives
\[
 |m(1)-\mathcal T(q_B)|
 \le\underbrace{|m(0)-\mathcal T(q_B)|}_{O(H_B^2/n)}
    +\underbrace{|m'(0)|}_{O(H_B/n)}
    +\underbrace{\tfrac12\sup_{0\le t\le1}|m''(t)|}_{O(n^{-1})}.
\]
At $t=0$, the row-sum and overlap variance errors are $O(H_B^2)$.
Division by $n$, together with the bounded internal-edge contributions,
gives the product-law error.

Conditioning on shared rows in
$m'(0)=\cov_0(D^2/(2n),J_{\rm cross})$ yields signed third moments of
overlaps and linear sums.
Their bounds and the $n^{-1}$ interaction coefficient give $|m'(0)|\le CH_B/n$.

Uniform moment bounds for $D$ and $J_{\rm cross}-\E_tJ_{\rm cross}$ give,
for $0\le t\le1$,
\[
\begin{aligned}
 |m''(t)|
 &\le\left\|\frac{D^2}{2n}-m(t)\right\|_{L^2(p_{B,t})}
      \norm{J_{\rm cross}-\E_tJ_{\rm cross}}_{L^4(p_{B,t})}^2\\
 &\le C(Cn^{-1/2})^2\le C/n.
\end{aligned}
\]
Since $H_B\ge1$, the three estimates give the required $CH_B^2/n$ bound.
\par}

\begin{proof}
Condition on $\calF_{\mathcal I}$ and label the $k$ vertices of $\mathcal I$
as $1,\ldots,k$. For the two estimates, set
\begin{equation*}
 (D,\mathcal T)=
 \begin{cases}
  (d_1-d_2,A_{\alpha,h}),&\text{for \eqref{eq:actual-refined-degree}},\\
  (M_{\{1,2\}}-M_{\{3,4\}},K_{\alpha,h}),
       &\text{for \eqref{eq:actual-refined-overlap}}.
 \end{cases}
\end{equation*}
Put $J_0=D^2/(2n)$ and $m(t)=\E_tJ_0$, with $p_{B,t}$ from
\eqref{eq:interpolated-row-law}. Then
$m(1)=\E(J_0\mid\calF_{\mathcal I})$.
By \eqref{eq:row-interpolation}, $m'(0)=\cov_0(J_0,J_{\rm cross})$.
Taylor's formula gives
\begin{equation}\label{eq:conditional-bias-budget}
\begin{aligned}
 |m(1)-\mathcal T(q_B)|
 \le\underbrace{|m(0)-\mathcal T(q_B)|}_{I_1}
      +\underbrace{|\cov_0(J_0,J_{\rm cross})|}_{I_2}+\underbrace{\frac12\sup_{0\le t\le1}|m''(t)|}_{I_3}.
\end{aligned}
\end{equation}
Since $H_B\ge1$, it suffices to prove
\begin{align}
 I_1&\le CH_B^2/n,\label{eq:actual-refined-first-term}\\
 I_2&\le CH_B/n,\label{eq:refined-first-interpolation}\\
 I_3&\le C/n.\label{eq:actual-refined-third-term}
\end{align}
\par\medskip
\noindent\textit{Proof of \eqref{eq:actual-refined-first-term}}: Equation \eqref{eq:conditional-degree-overlap} gives $|\E_0D|\le C_k$ and
\begin{equation*}
 D-\E_0D=
 \begin{cases}
  \mathbf1^\top(\mathbf z_1-\mathbf z_2),&\mathcal T=A_{\alpha,h},\\
  \mathbf z_1^\top\mathbf z_2-\mathbf z_3^\top\mathbf z_4,
       &\mathcal T=K_{\alpha,h}.
 \end{cases}
\end{equation*}
Under $p_{B,0}$, the two terms in each difference are independent and
have the same law. Hence
\begin{equation*}
\begin{aligned}
 m(0)=\frac{\var_0D+(\E_0D)^2}{2n}=\begin{cases}
 n^{-1}\var_B(\mathbf1^\top\mathbf z_1)+(\E_0D)^2/(2n),
       &\mathcal T=A_{\alpha,h},\\
 n^{-1}\var_B(\mathbf z_1^\top\mathbf z_2)+(\E_0D)^2/(2n),
       &\mathcal T=K_{\alpha,h}.
 \end{cases}
\end{aligned}
\end{equation*}
\eqref{eq:row-refined} in Lemma~\ref{lem:row-coarse},
$|\E_0D|\le C_k$, and $H_B\ge1$ give
\begin{equation*}
 I_1\le CH_B^2/n+(\E_0D)^2/(2n)\le CH_B^2/n.
\end{equation*}
This proves \eqref{eq:actual-refined-first-term}.

\par\medskip
\noindent\textit{Proof of \eqref{eq:refined-first-interpolation}}: Put
$\boldsymbol\mu=\E_B\mathbf z_1$ and, for $a\ne b$, define
\begin{equation*}
\begin{aligned}
 \boldsymbol\xi_a=\mathbf z_a-\boldsymbol\mu, \qquad
 U_{ab}=\mathbf z_a^\top\mathbf z_b-\boldsymbol\mu^\top\boldsymbol\mu
       =\boldsymbol\mu^\top\boldsymbol\xi_a
        +\boldsymbol\mu^\top\boldsymbol\xi_b
        +\boldsymbol\xi_a^\top\boldsymbol\xi_b.
\end{aligned}
\end{equation*}
The variables $\boldsymbol\mu^\top\boldsymbol\xi_a$ and $U_{ab}$ are centered under $p_{B,0}$.
By \eqref{eq:triangle-row},
\begin{equation}\label{eq:cross-covariance-sum}
\begin{aligned}
 J_{\rm cross}-\E_0J_{\rm cross}
 &=\frac\alpha n\sum_{a<b}e_{ab}U_{ab},\\
 I_2=|\cov_0(J_0,J_{\rm cross})|
 &\le\frac{|\alpha|}{n}\sum_{a<b}|\cov_0(J_0,U_{ab})|.
\end{aligned}
\end{equation}
There are at most $\binom{k}{2}$ pairs. It therefore suffices to prove
\begin{equation}\label{eq:interaction-covariance-target}
 |\cov_0(J_0,U_{ab})|\le CH_B,\qquad 1\le a<b\le k.
\end{equation}
\par\medskip
\noindent\textit{Proof of \eqref{eq:interaction-covariance-target}}: Since $J_0=D^2/(2n)$ and $\E_0U_{ab}=0$,
\begin{equation}\label{eq:interaction-centering}
 \cov_0(J_0,U_{ab})
 =\frac1{2n}\E_0[(D-\E_0D)^2U_{ab}]
  +\frac{\E_0D}{n}\cov_0(D,U_{ab}).
\end{equation}
By \eqref{eq:conditional-degree-overlap}, changing one free coordinate
changes $D$ by at most one. The same holds for $U_{ab}$ by its definition.
The variable $D$ depends on at most four rows, and $U_{ab}$ on two, so
\begin{equation*}
 \norm{\boldsymbol\delta D}_{\ell^2}^2\le4L,
 \qquad \norm{\boldsymbol\delta U_{ab}}_{\ell^2}^2\le2L.
\end{equation*}
The moment bound \eqref{eq:comparison-moment} gives
\begin{equation*}
 \norm{D-\E_0D}_{L^2(p_{B,0})}
 +\norm{U_{ab}}_{L^2(p_{B,0})}\le C\sqrt n.
\end{equation*}
Using $|\E_0D|\le C_k$ and the Cauchy--Schwarz inequality, we obtain
\begin{equation*}
 \frac{|\E_0D|}{n}|\cov_0(D,U_{ab})|
 \le\frac{C_k}{n}
       \norm{D-\E_0D}_{L^2(p_{B,0})}\norm{U_{ab}}_{L^2(p_{B,0})}
 \le C_k.
\end{equation*}
We now estimate the first term in \eqref{eq:interaction-centering}.
For the overlap difference, \eqref{eq:conditional-degree-overlap} gives
$D-\E_0D=U_{12}-U_{34}$. We consider the following three cases:
\par\smallskip
\noindent\textit{Case I:} $\{a,b\}\subset\{1,2,3,4\}$.\par\nopagebreak
\noindent\textit{Case II:} $|\{a,b\}\cap\{1,2,3,4\}|=1$.\par\nopagebreak
\noindent\textit{Case III:} $\{a,b\}\cap\{1,2,3,4\}=\varnothing$.\par

\par\smallskip
\noindent\textit{Case I}: By interchanging the rows, it suffices to consider
$\{a,b\}=\{1,2\}$ and $\{a,b\}=\{1,3\}$.
For $\{a,b\}=\{1,2\}$, by independence, we have
\begin{equation}\label{eq:overlap-cov-1}
 \frac1{2n}\E_0[(D-\E_0D)^2U_{12}]
 =\frac{\E_0U_{12}^3}{2n}.
\end{equation}
Lemma~\ref{lem:overlap-cubic} gives $|\E_0U_{12}^3|\le CnH_B$,
so the absolute value in \eqref{eq:overlap-cov-1} is at most $CH_B$.

Put $\Sigma=\cov_B(\mathbf z_1)$.
Since $U_{12}=\boldsymbol\mu^\top\boldsymbol\xi_1+\mathbf z_1^\top\boldsymbol\xi_2$, independence gives
\begin{equation}\label{eq:appendix-conditional}
\begin{aligned}
 \E_0(U_{12}\mid\mathbf z_1)&=\boldsymbol\mu^\top\boldsymbol\xi_1,\\
 \E_0(U_{12}^2\mid\mathbf z_1)
 &=(\boldsymbol\mu^\top\boldsymbol\xi_1)^2+\mathbf z_1^\top\Sigma\mathbf z_1.
\end{aligned}
\end{equation}
For $\{a,b\}=\{1,3\}$, \eqref{eq:appendix-conditional} gives
\begin{equation*}
 \E_0(U_{12}^2U_{13})=\E_0(U_{34}^2U_{13})
 =\E_0[(\boldsymbol\mu^\top\boldsymbol\xi_1)^3]+\cov_0(\mathbf z_1^\top\Sigma\mathbf z_1,\boldsymbol\mu^\top\boldsymbol\xi_1).
\end{equation*}
For $\E_0(U_{12}U_{34}U_{13})$, condition on $(\mathbf z_1,\mathbf z_3)$.
Given these two rows, $\mathbf z_2$ and $\mathbf z_4$ remain independent
with law $p_B$. By \eqref{eq:appendix-conditional} and the definition of $U_{13}$,
\begin{equation*}
\begin{aligned}
 \E_0(U_{12}U_{34}U_{13})
 &=\E_0[(\boldsymbol\mu^\top\boldsymbol\xi_1)(\boldsymbol\mu^\top\boldsymbol\xi_3)U_{13}]=\E_0[(\boldsymbol\mu^\top\boldsymbol\xi_1)(\boldsymbol\mu^\top\boldsymbol\xi_3)\boldsymbol\xi_1^\top\boldsymbol\xi_3]\\
 &=\sum_{u\in\mathcal I^c}
       \E_0[(\boldsymbol\mu^\top\boldsymbol\xi_1)\xi_{1u}]\E_0[(\boldsymbol\mu^\top\boldsymbol\xi_3)\xi_{3u}]
   =\norm{\Sigma\boldsymbol\mu}_{\ell^2}^2.
\end{aligned}
\end{equation*}
Consequently,
\begin{equation}\label{eq:overlap-cov-3}
 \frac1{2n}\E_0[(D-\E_0D)^2U_{13}]
 =\frac{\E_0[(\boldsymbol\mu^\top\boldsymbol\xi_1)^3]+\cov_0(\mathbf z_1^\top\Sigma\mathbf z_1,\boldsymbol\mu^\top\boldsymbol\xi_1)
              -\norm{\Sigma\boldsymbol\mu}_{\ell^2}^2}{n}.
\end{equation}
Since $0\le\mu_u\le1$, the conclusion of
Lemma~\ref{lem:row-refined-third} gives
$|\E_0[(\boldsymbol\mu^\top\boldsymbol\xi_1)^3]|\le Cn$.
The row-sum bound \eqref{eq:row-covariance-entries} gives
\begin{equation*}
\begin{aligned}
 \delta_u(\mathbf z_1^\top\Sigma\mathbf z_1)
=\sup_{\mathbf z_{1,-u}}
   \left|\Sigma_{uu}+2\sum_{v\ne u}\Sigma_{uv}z_{1v}\right|\le |\Sigma_{uu}|+2\sum_{v\ne u}|\Sigma_{uv}|\le C.
\end{aligned}
\end{equation*}
Since $\delta_u(\boldsymbol\mu^\top\boldsymbol\xi_1)=\mu_u\le1$, \eqref{eq:comparison-cov} implies
\begin{equation*}
 |\cov_0(\mathbf z_1^\top\Sigma\mathbf z_1,\boldsymbol\mu^\top\boldsymbol\xi_1)|
 \le C\norm{\boldsymbol\delta(\mathbf z_1^\top\Sigma\mathbf z_1)}_{\ell^2}
        \norm{\boldsymbol\mu}_{\ell^2}
 \le Cn.
\end{equation*}
By \eqref{eq:row-coarse-a},
\begin{equation*}
 \norm{\Sigma\boldsymbol\mu}_{\ell^2}^2
 \le\norm\Sigma_{\rm op}^2\norm{\boldsymbol\mu}_{\ell^2}^2\le Cn.
\end{equation*}
Together with the preceding bounds, this bounds the absolute value in
\eqref{eq:overlap-cov-3} by $C$.

\par\smallskip
\noindent\textit{Case II}: By interchanging the rows, we may write $\{a,b\}=\{1,5\}$. Under $p_{B,0}$, $U_{34}$ is independent of
$(U_{12},U_{15})$. Since $\E_0U_{34}=\E_0U_{15}=0$,
\begin{equation*}
\begin{aligned}
 \E_0(U_{12}U_{34}U_{15})
 &=\E_0(U_{12}U_{15})\E_0U_{34}=0,\\
 \E_0(U_{34}^2U_{15})
 &=\E_0U_{34}^2\,\E_0U_{15}=0.
\end{aligned}
\end{equation*}
Given $\mathbf z_1$, the variables $U_{12}$ and $U_{15}$ depend on the
independent rows $\mathbf z_2$ and $\mathbf z_5$, respectively, and are
therefore conditionally independent.
Using $D-\E_0D=U_{12}-U_{34}$, we obtain from \eqref{eq:appendix-conditional}
\begin{equation}\label{eq:overlap-cov-2}
\begin{aligned}
 \frac1{2n}\E_0[(D-\E_0D)^2U_{15}]
 &=\frac1{2n}\E_0[(U_{12}^2-2U_{12}U_{34}+U_{34}^2)U_{15}]\\
 &=\frac1{2n}\E_0\!\left[\E_0(U_{12}^2U_{15}\mid\mathbf z_1)\right]\\
 &=\frac1{2n}\E_0\!\left[\E_0(U_{12}^2\mid\mathbf z_1)
                         \E_0(U_{15}\mid\mathbf z_1)\right]\\
 &=\frac1{2n}\E_0[((\boldsymbol\mu^\top\boldsymbol\xi_1)^2+\mathbf z_1^\top\Sigma\mathbf z_1)(\boldsymbol\mu^\top\boldsymbol\xi_1)]\\
 &=\frac{\E_0[(\boldsymbol\mu^\top\boldsymbol\xi_1)^3]+
          \cov_0(\mathbf z_1^\top\Sigma\mathbf z_1,\boldsymbol\mu^\top\boldsymbol\xi_1)}{2n}.
\end{aligned}
\end{equation}
Thus the absolute value in \eqref{eq:overlap-cov-2} is at most $C$.

\par\smallskip
\noindent\textit{Case III}: The variable $U_{ab}$ is independent of $(U_{12},U_{34})$ and has mean zero,
so $\E_0[(D-\E_0D)^2U_{ab}]=0$.

For the degree difference, \eqref{eq:conditional-degree-overlap} gives
$D-\E_0D=\mathbf1^\top\boldsymbol\xi_1-\mathbf1^\top\boldsymbol\xi_2$.
For a pair meeting $\{1,2\}$ only at $1$, denote its other index by $3$.
The first identity in \eqref{eq:appendix-conditional} gives
$\E_0(U_{13}\mid\mathbf z_1)=\boldsymbol\mu^\top\boldsymbol\xi_1$.
Independence and centering yield
\begin{equation*}
\begin{aligned}
 \E_0[(\mathbf1^\top\boldsymbol\xi_1)^2U_{13}]
 &=\E_0[(\mathbf1^\top\boldsymbol\xi_1)^2(\boldsymbol\mu^\top\boldsymbol\xi_1)],\\
 \E_0[(\mathbf1^\top\boldsymbol\xi_2)^2U_{13}]
 &=\E_0[(\mathbf1^\top\boldsymbol\xi_2)^2]\,\E_0U_{13}=0,\\
 \E_0[(\mathbf1^\top\boldsymbol\xi_1)(\mathbf1^\top\boldsymbol\xi_2)U_{13}]
 &=\E_0(\mathbf1^\top\boldsymbol\xi_2)\,
   \E_0[(\mathbf1^\top\boldsymbol\xi_1)U_{13}]=0.
\end{aligned}
\end{equation*}
For $U_{12}$, conditioning on either row gives
\begin{equation*}
 \E_0[(\mathbf1^\top\boldsymbol\xi_1)^2U_{12}]
 =\E_0[(\mathbf1^\top\boldsymbol\xi_2)^2U_{12}]
 =\E_0[(\mathbf1^\top\boldsymbol\xi_1)^2(\boldsymbol\mu^\top\boldsymbol\xi_1)].
\end{equation*}
In the mixed term, the contributions of $\boldsymbol\mu^\top\boldsymbol\xi_1+\boldsymbol\mu^\top\boldsymbol\xi_2$ vanish. Thus
\begin{equation*}
\begin{aligned}
 \E_0[(\mathbf1^\top\boldsymbol\xi_1)(\mathbf1^\top\boldsymbol\xi_2)U_{12}]
 &=\E_0[(\mathbf1^\top\boldsymbol\xi_1)(\mathbf1^\top\boldsymbol\xi_2)
          \boldsymbol\xi_1^\top\boldsymbol\xi_2]\\
 &=\sum_u\E_0[(\mathbf1^\top\boldsymbol\xi_1)\xi_{1u}]
        \E_0[(\mathbf1^\top\boldsymbol\xi_2)\xi_{2u}]
 =\norm{\Sigma\mathbf1}_{\ell^2}^2.
\end{aligned}
\end{equation*}
Expanding $(D-\E_0D)^2$ now gives
\begin{equation}\label{eq:appendix-degree-covariances}
\begin{aligned}
 \frac1{2n}\E_0[(D-\E_0D)^2U_{13}]
 &=\frac{\E_0[(\mathbf1^\top\boldsymbol\xi_1)^2(\boldsymbol\mu^\top\boldsymbol\xi_1)]}{2n},\\
 \frac1{2n}\E_0[(D-\E_0D)^2U_{12}]
 &=\frac{\E_0[(\mathbf1^\top\boldsymbol\xi_1)^2(\boldsymbol\mu^\top\boldsymbol\xi_1)]
          -\norm{\Sigma\mathbf1}_{\ell^2}^2}{n}.
\end{aligned}
\end{equation}
Pairs meeting $\{1,2\}$ only at $2$ follow by interchanging the two rows.
Pairs disjoint from $\{1,2\}$ give zero by independence.
The weights of $\mathbf1^\top\boldsymbol\xi_1$ and $\boldsymbol\mu^\top\boldsymbol\xi_1$ are $1$ and $\mu_u\in[0,1]$.
Lemma~\ref{lem:row-refined-third}
and \eqref{eq:row-coarse-a} therefore give
\begin{equation*}
 \bigl|\E_0[(\mathbf1^\top\boldsymbol\xi_1)^2(\boldsymbol\mu^\top\boldsymbol\xi_1)]\bigr|\le Cn,\qquad
 \norm{\Sigma\mathbf1}_{\ell^2}^2
 \le\norm{\Sigma}_{\rm op}^2L\le Cn.
\end{equation*}
Both quantities in \eqref{eq:appendix-degree-covariances} are $O(1)$.

Combining the degree and overlap bounds with
\eqref{eq:interaction-centering}, and using $H_B\ge1$, proves
\eqref{eq:interaction-covariance-target} for every $a<b$.
Substitution into \eqref{eq:cross-covariance-sum} gives
\begin{equation*}
 I_2=|\cov_0(J_0,J_{\rm cross})|
 \le\frac{|\alpha|}{n}\sum_{a<b}CH_B
 \le\frac{CH_B}{n}.
\end{equation*}
This proves \eqref{eq:refined-first-interpolation}.

\par\medskip
\noindent\textit{Proof of \eqref{eq:actual-refined-third-term}}: Differentiate
$m'(t)=\cov_t(J_0,J_{\rm cross})$ using
\eqref{eq:row-tilt-derivative} to obtain
\begin{equation*}
\begin{aligned}
 m''(t)
 &=\E_t[J_0(J_{\rm cross}-\E_tJ_{\rm cross})^2]
       -\E_tJ_0\,\var_t(J_{\rm cross})\\
 &=\E_t[(J_0-\E_tJ_0)(J_{\rm cross}-\E_tJ_{\rm cross})^2].
\end{aligned}
\end{equation*}
Since $J_0=D^2/(2n)$, \eqref{eq:interpolated-contrast-moments} with $r=4$ gives
\begin{equation*}
\begin{aligned}
 \sup_{0\le t\le1}\norm{J_0-\E_tJ_0}_{L^2(p_{B,t})}
 \le\sup_{0\le t\le1}\norm{J_0}_{L^2(p_{B,t})}=\frac1{2n}\sup_{0\le t\le1}\norm D_{L^4(p_{B,t})}^2\le C.
\end{aligned}
\end{equation*}
The Cauchy--Schwarz inequality and \eqref{eq:cross-moments} imply
\begin{equation*}
\begin{aligned}
 I_3
 &\le\frac12\sup_{0\le t\le1}
   \norm{J_0-\E_tJ_0}_{L^2(p_{B,t})}
   \norm{J_{\rm cross}-\E_tJ_{\rm cross}}_{L^4(p_{B,t})}^2
 \le\frac Cn.
\end{aligned}
\end{equation*}
Substituting \eqref{eq:actual-refined-first-term},
\eqref{eq:refined-first-interpolation} and \eqref{eq:actual-refined-third-term}
into \eqref{eq:conditional-bias-budget} gives
\begin{equation*}
 |m(1)-\mathcal T(q_B)|
 \le\frac{CH_B^2}{n}+\frac{CH_B}{n}+\frac Cn
 \le\frac{CH_B^2}{n}.
\end{equation*}
This proves \eqref{eq:actual-refined-degree} and \eqref{eq:actual-refined-overlap}.
\end{proof}

\section*{Acknowledgements}
Liu S.H. was partially supported by the Fundamental Research Funds for the Central Universities DUT25RC(3)133. Shao Q.M. was partially supported by National Nature Science Foundation of China NSFC 12031005 and Shenzhen Outstanding Talents Training Fund, China.

\section*{AI Statement}
During the preparation of this manuscript, the authors used ChatGPT and
Codex (OpenAI) for language editing, LaTeX formatting, and assistance with
proof exposition and organization. These tools were also used to check
references, notation, cross-references, and mathematical derivations.
The authors reviewed and revised the manuscript, verified the mathematical
arguments, and take full responsibility for its content.

\end{document}